\documentclass[11pt]{article}

\usepackage{graphicx,,psfrag}

\usepackage{amssymb,amsmath,amsthm,enumerate}
\usepackage{fullpage}
\newcommand{\url}{}

\RequirePackage[colorlinks,citecolor=blue,urlcolor=blue]{hyperref}

\newtheorem {theoreme} {Theorem} [section]
\newtheorem {lemme} [theoreme] {Lemma}
\newtheorem {proposition}  [theoreme]{Proposition}
\newtheorem {corollaire} [theoreme] {Corollary}
\newtheorem {remarque} {Remark} [section]

\numberwithin{equation}{section}
\numberwithin{figure}{section}

\newcommand{\tr}{{\rm tr}}
\newcommand{\Tr}{{\rm Tr}}
\newcommand{\E}{\mathbb {E}}\newcommand{\dE}{\E}
\newcommand{\dP}{\mathbb{P}}

\newcommand{\R}{\mathbb {R}}\newcommand{\dR}{\R}
\newcommand{\C}{\mathbb {C}} \newcommand{\dC}{\C}

\newcommand{\cW}{\mathcal {W}}
\newcommand{\cP}{\mathcal {P}}
\newcommand{\cN}{\mathcal {N}}
\newcommand{\cS}{\mathcal {S}}

\newcommand{\supp}{ \mathrm{supp}}
\newcommand{\bl}{[\hspace{-1pt}[}
\newcommand{\br}{]\hspace{-1pt}]}

\newcommand{\1}{1\!\!{\sf I}}\newcommand{\IND}{\1}

\newcommand{\veps}{\varepsilon}
\newcommand{\ibf}{\mathbf i}
\newcommand{\jbf}{\mathbf j}
\newcommand{\DIST}{\mathrm{dist}}
\newcommand{\DIAG}{\mathrm{diag}}

\title{Outlier eigenvalues for deformed i.i.d. random matrices}

\author{Charles Bordenave \and Mireille Capitaine}

\date{}
\begin{document}
\maketitle
\begin{abstract}
We consider a square random matrix of size $N$ of the form $A + Y$ where $A$ is deterministic and $Y$ has iid entries with variance $1/N$. Under mild assumptions, as $N$ grows, the empirical distribution of the eigenvalues of $A+Y$ converges weakly to a limit probability  measure $\beta$ on the complex plane. This work is devoted to the study of the outlier eigenvalues,  i.e. eigenvalues in the complement of the support of $\beta$. Even in the simplest cases, a variety of interesting phenomena can occur. As in earlier works, we give a sufficient condition to guarantee that outliers are stable and provide examples where their fluctuations vary with the
particular distribution of the entries of $Y$ or the Jordan decomposition of $A$.  We also  exhibit concrete examples where the outlier eigenvalues converge in distribution to the zeros of a Gaussian analytic function. 
\end{abstract}

%% MOTS CLES  random matrices, outlier eigenvalues, Gaussian analytic functions
%% (prim) 15B52, 60B20 (sec) 15A18,  60F05

\section{Introduction}

\subsection{Numerical instability of eigenvalues}

The instability of the eigenvalues of badly conditioned matrices has dramatic consequences in the numerical computation of eigenvalues. As an example, take $N\geq 1$ be an integer and consider the standard nilpotent matrix
\begin{equation}\label{eq:Anil}
A_N = \sum_{i=1}^{N-1} e_{i+1} e_i^*  = 
 \begin{pmatrix}
  0 & 1 & 0 &\cdots  & 0 \\
  0 & 0 &  1 & \cdots & 0  \\
  \vdots  & \vdots   &  \ddots  & \ddots& \vdots  
 \end{pmatrix}.
\end{equation}
Its eigenvalues are obviously all zero. Let $U_N$ be a Haar-distributed orthogonal matrix and consider the unitarily equivalent matrix $B_N = U_N A_N U_N^*$. If we ask a computer to compute the eigenvalues of $B_N$, we obtain a surprising answer. Figure \ref{fig:nummat} is a plot of these numerically computed eigenvalues of $B_N$.  

\begin{figure}[htb]
\begin{center}
\includegraphics[width = 8.1cm]{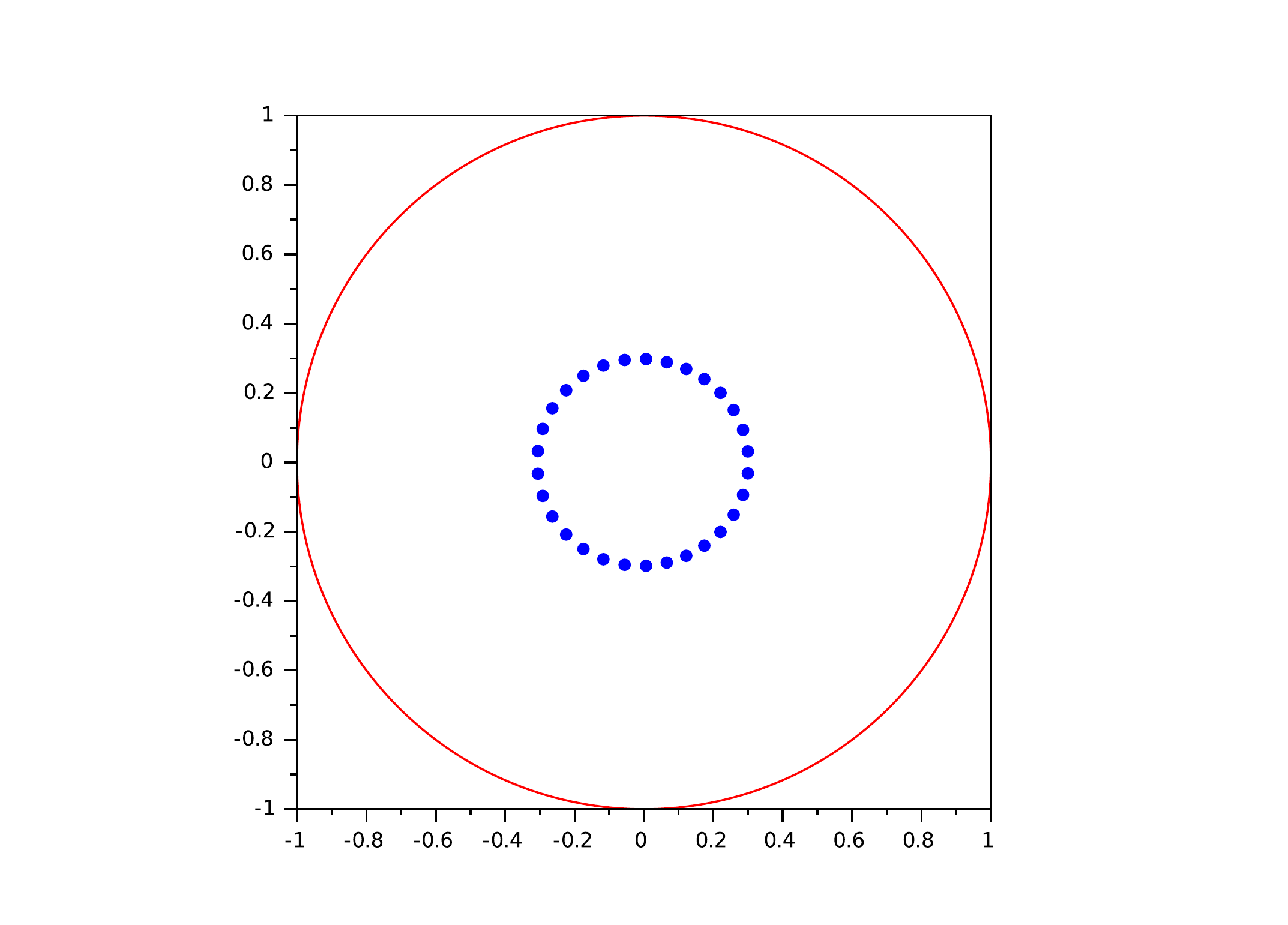}
\includegraphics[width = 8.1cm]{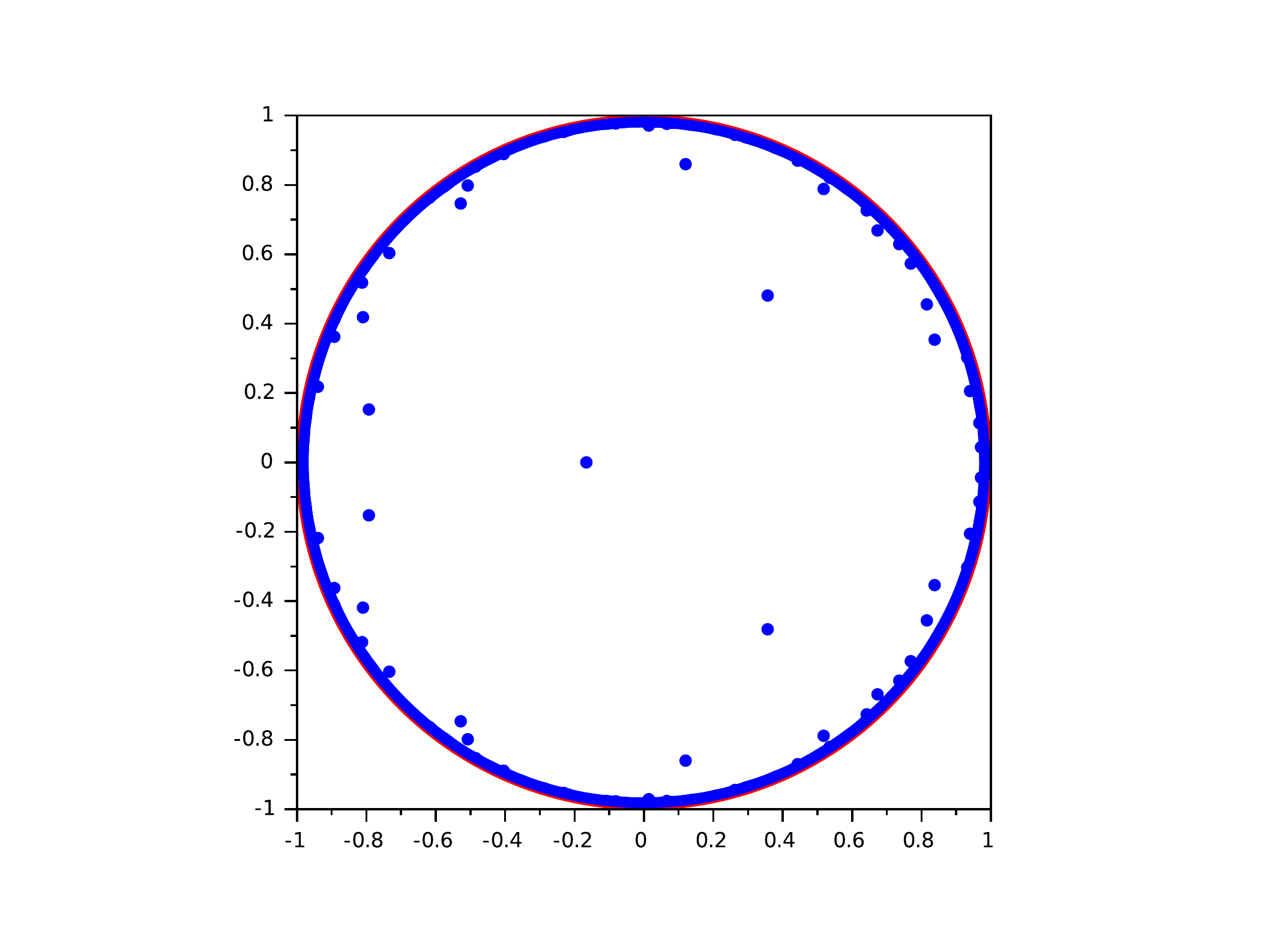}
\end{center}
\caption{The blue dots are the numerically computed eigenvalues of $B_N = U_N A_N U_N^*$, the nilpotent matrix \eqref{eq:Anil} conjugated by a Haar-distributed orthogonal matrix $U_N$ for $N = 30$ and $N = 2000$. 
}
\label{fig:nummat}\end{figure}

In the spirit of von Neumann and Goldstine \cite{MR0024235}, Spielman and Teng \cite{MR1989210} or Edelman and Rao \cite{MR2168344},  a possible way to try to explain this phenomenon is to approximate numerical rounding errors by randomness and study the spectrum of the matrix
$$
A_N+ \sigma Y_N,
$$
where $Y_N$ is a random matrix normalized to have an operator norm of order $1$ and $\sigma$ is small positive parameter. As we shall see, in Subsection \ref{subsec:unsou},  in the limit $N \to \infty$ and then $\sigma \to 0$, one obtain a reasonable explanation of the right picture of Figure \ref{fig:nummat}. A phenomenon related to the left figure, namely that the numerically estimated eigenvalues of a Jordan block are close to be roots of  a small complex number will also be illustrated in Subsection \ref{subsec:fluctstou}

\subsection{Deformed random matrices}

For any  $N\times N$ matrix $A$, denote by 
$\lambda_1(A), \ldots, \lambda_N(A)$
\noindent the  eigenvalues of $A$ and by $\mu_{A}$  the empirical spectral measure of $A$: $$\mu _{A} := \frac{1}{N} \sum_{i=1}^N \delta_{\lambda _{i}(A)}.$$

We will consider the deformed model:
\begin{equation}\label{modele}M_N=  A_N + \sigma Y_N,\end{equation}
where  $\sigma >0$,  $Y_N$ is a $N\times N$ random matrix and $A_N$ is a $N\times N$ deterministic matrix. The matrix $M_N$ can be thought as a random perturbation of the matrix $A_N$.  Throughout this paper, we set 
\begin{equation}\label{defY}
Y_N = \frac { X_N}{\sqrt N},
\end{equation}
and we shall consider the following set of statistical assumptions on the matrices $X_N = (X_{ij})_{1 \leq i , j \leq N}$: 
\begin{enumerate}[(X1)]
\item $(X_{ij})_{i,j \geq 1}$  are independent and identically distributed complex random variables with $\dE X_{ij} = 0$, $\dE |X_{ij}|^2 =1$. 
\item $\dE | X_{ij} |^4  < \infty$.
\item There exists $c >0$ such that for all $k \geq 1$ integer $\dE |X_{ij} |^k \leq ( c k )^c$.  
\end{enumerate}

Our first assumptions on the matrices $A_N$ are as follows:
\begin{enumerate}
\item[(A1)] There exists $M >0$ such that for all $N$, $\|A_N\| \leq M$.  

\item[(A2)] For all $z \in \dC$, $\mu_{(A_N - z I_N)(A_N - z I_N)^* }$ converges weakly to a probability measure $\nu_z$.    
\end{enumerate}

We start by recalling a generalization of the circular law to the deformed matrix model $M_N$. 
\begin{theoreme}\label{convbeta}
Under assumptions (X1) and (A1-A2), there exists a deterministic probability measure $\beta$ on $\dC$ such that, almost surely, $\mu_{M_N}$ converges weakly to $\beta$. Moreover, for any $z \in \dC$, there exists a  deterministic probability measure $\mu_z$ on $\dR_+$ such that, almost surely, $\mu_{(M_N - z I_N)(M_N - zI_N)^* }$ converges weakly to $\mu_z$. \end{theoreme}

The first statement was first proved in Tao and Vu \cite{tao-vu-cirlaw-bis}, the second statement was an ingredient of the proof, it is due to Dozier and Silverstein \cite{DozierSilver}. For more references, we refer to the surveys  \cite{MR2507275,MR2908617}. In subsection \ref{subses:characbeta}, we will give a precise characterization of $\beta$ and $\mu_z$ in terms of $\sigma$ and $\nu_z$. The assumption (A1) can be weakened, see \cite{ECP2011-10}. If $A_N = 0$ then $\beta$ is the uniform distribution on $B(0,\sigma)$, the closed ball of radius $\sigma$ and center $0$ in $\dC$. Beware that (i) assumptions (A1-A2) do not imply that $\mu_{A_N}$ converges weakly to a measure $\alpha$ on $\dC$ (neither the opposite) and (ii) even though, it is not always the case  that as $\sigma$ goes to $0$, $\beta$ converges weakly to $\alpha$.
% It is however the case if $A_N$ is a sequence of normal matrices. 

We have not been able to compute the support of $\beta$ in general. There is however a typical situation where it takes a nice form. Observe first that $\lambda$ is an eigenvalue of $A_N$ if and only if $0$ is an eigenvalue of $(A_N - \lambda I_N)(A_N - \lambda I_N)^*$. We will assume that a similar property holds for $\beta$ and $\mu_z$, i.e.
\begin{enumerate}
\item[(A3)] $\supp( \beta ) = \{ z \in \dC : 0 \in \supp ( \mu_z)\}$. 
\end{enumerate}

We are not aware of an example where (A3) fails to hold. We shall prove that (A3) holds if $\nu_z$ is the law of $|L-z|^2$ where $L$ is a random variable on $\dC$ with distribution $\alpha$, (see the forthcoming Lemma \ref{le:A3normalcase}). This case will occur  if $A_N$ is a normal matrix and $\mu_{A_N}$ converges weakly to $\alpha$ (or if, for some  normal matrix $B_N$,  either $A_N -B_N$ has rank $o(N)$ or $\Tr (A_N - B_N)(A_N - B_N) ^* = o(N)$).  Under assumption (A3) the support of $\beta$ takes a particularly simple expression. We introduce $$S = \{ z \in \dC : 0 \in \supp( \nu_z ) \}.$$ 

\begin{proposition}\label{prop:supportbeta}
Suppose that assumptions (X1) and (A1-A3) hold. Then 
\begin{equation}\label{eq:supportbeta}
\supp (\beta) = \left\{ z \in \dC : z \in S \hbox{ or }   \int \lambda^{-1} d \nu_z (\lambda)  \geq \sigma^{-2}\right\}. 
\end{equation}
\end{proposition}

For example, if $A_N = 0$, then $\nu_z$ is a Dirac mass at $|z|^2$ and we retrieve the support of the circular law. If $A_N$ is given by \eqref{eq:Anil}, then $\nu_z$ is the law of $|L -z|^2 $ with $L$ uniformly distributed on the unit complex circle. We find that $\supp( \beta)$ is the annulus  with inner radius $ \sqrt {( 1- \sigma^2)_+} $ and outer radius $\sqrt{1 + \sigma^2}$. In the limit $\sigma \to 0$, $ \supp( \beta)$ converges to the unit complex circle. This is consistent with Figure \ref{fig:nummat}.

\subsection{Stable outliers}

We are now interested by describing the individual eigenvalues of $M_N$ outside $B ( \supp ( \beta), \veps)$ for some $\veps >0$. To this end, we shall fix a set $\Gamma \subset \dC $ and assume that all but $O(1)$ of the eigenvalues of the matrix $A_N$ are outside $\Gamma$. To this end, we write 
$$
A_N = A'_N + A''_N.
$$
We first extend assumption (A1) to both $A'_N$ and $A''_N$: 
\begin{enumerate}
\item[(A1')] There exists $M >0$ such that for all $N$, $\|A'_N\| + \|A''_N\| \leq M$.  
\end{enumerate}

Our next key assumption asserts that $A'_N - z I_N$ is well conditioned in $\Gamma$ while $A''_N$ has small rank. We fix some integer $r \geq 0$.

\begin{enumerate}
\item[(A4)] 
$A''_N$ has rank $r$, and for any  $z \in \Gammaŝ$,  there exists $\eta = \eta_z >0$ such that for all $N$ large enough, $(A'_N - z I_N)$ has no singular value in $[0, \eta]$. 
\end{enumerate}

When $\Gamma$ is compact, observe that (A4) implies that for some $\veps >0$, the eigenvalues of $A'_N$ are in $\dC\backslash B ( \Gamma, \veps)$ for all $N$ large enough. In the case where $A'_N$ is a normal matrix then the singular values of $(A'_N - z I_N)$ are $|\lambda_k (A'_N) - z |, 1 \leq k \leq N$. Hence,  the assumption (A4) for $A'_N$ normal and $\Gamma$ compact holds if and only if for some $\veps >0$ and all $N$ large enough, the eigenvalues of $A'_N$ lie in $\dC \backslash B(\Gamma,\veps)$.

Our first main result gives a sufficient condition to guarantee that outliers are stable.

\begin{theoreme}\label{th:main}
Suppose that assumptions (X1-X2) and assumptions (A1'-A4) hold for $\Gamma \subset \dC \backslash \supp (\beta)$ a compact set with continuous boundary. If for some $\veps >0$ and all $N$ large enough,
\begin{equation}\label{eq:ratioAA'}
\min_{z  \in \partial \Gamma} \left| \frac { \det ( A_N - z) }{\det ( A'_N -z ) } \right| \geq \veps, 
\end{equation}
then a.s. for all $N$ large enough, the number of eigenvalues of $A_N$ and $M_N$ in $\Gamma$ is equal. 
\end{theoreme}

Above, the notation (A1'-A4) stands for (A1')-(A2)-(A3)-(A4). We will first prove Theorem \ref{th:main} in the case $r =0$. 
\begin{theoreme}\label{inclusion}
Suppose that assumptions (X1-X2) and assumptions (A1-A4) hold with $A''_N = 0$, $A_N = A'_N$ and $\Gamma \subset \dC \backslash \supp (\beta)$ a compact set. Then,   a.s. for all $N$ large enough, $M_N$ has no eigenvalue in $\Gamma$. 

 In particular, if (A4) holds with $r =0$ and $\Gamma = \dC \backslash \supp ( \beta)$ then  for any $\veps >0$, a.s. for all $N$ large enough, all eigenvalues of $M_N$ are in $B(\supp (\beta) , \veps)$. 
\end{theoreme}

Let us give a concrete application of Theorem \ref{th:main} with a specific decomposition of $A_N = A'_N + A''_N$. Assume that for all $N$, there exists a subset $J \subset \{1,\cdots, N\}$ of cardinal at most $r$ such that for any $\veps >0$, for all $N$ large enough and $k \in J$, $\lambda_k (A_N) \notin B(\supp ( \beta), \veps)$. We consider a triangular decomposition of $A_N$:
$$
A_N = P \left( \begin{array}{c|c}
T'' & * \\
\hline
0 & T' 
 \end{array} \right) P^{-1},
$$
where $P$ is an invertible matrix, $T'$ is an upper triangular matrix of size $N - |J|$ with the eigenvalues $\lambda_k (A_N)$, $k \notin J$, on the diagonal, and $T''$ is an upper triangular matrix of size $|J|$ with diagonal entries $\lambda_k (A_N)$, $k \in J$. 
Fix some $a \in  \supp (\beta)$,  we decompose $A_N$ as $A_N= A'_N + A''_N$
with
\begin{equation}\label{eq:defA'A"}
A'_N =  P \left( \begin{array}{c|c}
a I_J  & * \\
\hline
0 & T'
 \end{array} \right) P^{-1}
\quad \hbox{ and } \quad A''_N =  P \left( \begin{array}{c|c}
T'' - a I_J  & 0 \\
\hline
0 & 0
 \end{array} \right) P^{-1}.
\end{equation}

In particular, 
\begin{equation}\label{eq:detA'A"}
 \frac { \det ( A_N - z) }{\det ( A'_N -z ) }  =  \prod_{k\in J} \frac{ \lambda_k(A_N) - z }{ a - z }.
\end{equation}

The next statement will be an easy consequence of Theorem \ref{th:main}. It generalizes Tao \cite[Theorem 1.7]{tao-outliers} where $A'_N = 0$. When $A'_N$ is a Wigner random matrix, it is a special case of O'Rourke and Renfrew \cite[Theorem 2.4]{ORR13}.

\begin{corollaire}\label{cor:main}
Assume that assumptions (X1-X2) and assumptions (A1'-A4) hold with $A'_N$, $A''_N$ given by \eqref{eq:defA'A"} and $\Gamma = \dC \backslash \supp ( \beta)$. Fix $\veps >0$.  Assume that for all $N$ large enough, $\forall k \in J $, $\lambda_k (A_N) \notin  B( \supp ( \beta), 3\veps)$. Then, a.s. for all $N$ large enough, there are exactly $|J|$ eigenvalues of $M_N$ in  $\dC \backslash B( \supp ( \beta), 2\veps)$. Moreover, if we index them by $\lambda_k (M_N)$, $k \in J$, after labeling properly, a.s. 
$$\max_{ k \in J} | \lambda_k ( M_N)  - \lambda_k (A_N)| \to 0.$$ 
\end{corollaire}

  In Figure \ref{fig:stable}, we illustrate numerically Corollary \ref{cor:main} when 
 \begin{equation}\label{eq:ex1A}
A_N =  \left( \begin{array}{c|c}
B & 0 \\
\hline
0 & C
 \end{array} \right)\quad  \hbox{with }  B \in M_r ( \dC) \hbox{ and  } C =   \sum_{i = r+1}^{N-1} e_{i} e_{i+1}^*   + e_{N} e_{r+1}^* .
\end{equation}

\begin{figure}[htb]
\begin{center}
\includegraphics[width = 10cm]{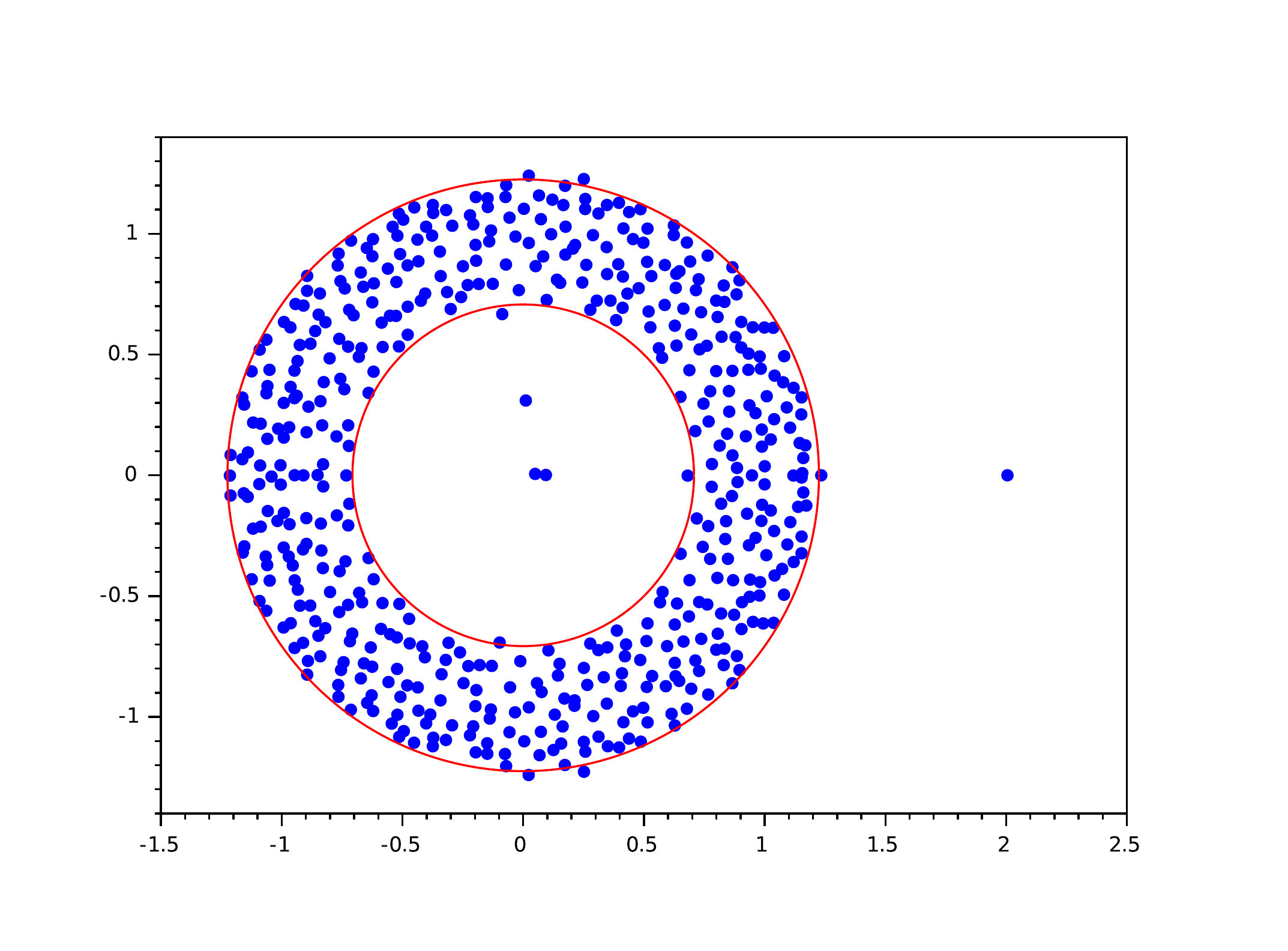}
\end{center}
\caption{Eigenvalues of $M_N$ where $A_N$ is given by \eqref{eq:ex1A} with $r =5$ and $B$ has eigenvalues $(0,0,i/3,1,2)$, $N = 500$, $\sigma^2 = 1/2$ and $X_N$ has real Gaussian entries. The support of $\beta$ is $\{z \in \mathbb{C} : 1 / \sqrt 2 \leq |  z |  \leq \sqrt{3/2}\}$, the stable outliers are $(0,0,i/3,2)$.}
\label{fig:stable}\end{figure}

Assumption \eqref{eq:ratioAA'} is the key assumption for the stability of the outliers. It holds generically in the unbounded component of the complement of $S$. 

\begin{lemme}\label{le:noouter}
Suppose that  assumptions (A1'-A4) hold with $ \Gamma = D$ an unbounded component of $\dC \backslash S$. The family of $D \to \dC$ functions $f_N : z \mapsto   \det ( A_{N} - z)  / \det ( A'_{N} -z ) $ is a precompact family of analytic functions and any subsequential limit of $(f_N)$ is non-zero. In particular, along any converging subsequence $(f_ {N_k})_{k \geq 1}$, for any $B(z,t) \subset D$ and any $\delta >0$, there exist  $\veps >0$ and $\Gamma = B(z,t')$ with $t - \delta < t' \leq t$  such that  \eqref{eq:ratioAA'} holds for all $N_k$, $k \geq 1$. 
\end{lemme}

In the sequel, in order to circumvent the multiple possible choices to order the eigenvalues, we will consider the finite point set $\sum_{i=1} ^n \delta_{x_i}$ of a vector $(x_1, \cdots, x_n)$.  Let us first recall some basic facts on finite point processes (we refer to Daley and Vere-Jones \cite[Appendix A.2.5]{MR1950431} for details and terminology). If $\cS$ is a complete separable metric space, we denote by $\mathcal {N}(\mathcal S)$, the set of finite point (integer valued) measures on $\mathcal S$, equipped with the usual weak topology.  Recall that a point process is random variable on $\mathcal N (\mathcal S)$. The set $\cP( \mathcal N (\mathcal S)) $ of probability measures on $\cN(\cS)$ is a complete separable metric space and the L\'evy-Prohorov distance is a metric for the weak convergence of measures in $\cP ( \cN(\cS))$ (or with a slight abuse of language, for weak convergence of point processes on $\cS$).

\subsection{Fluctuations of stable outliers}
\label{subsec:fluctstou}

We have studied the fluctuations of the convergence of outliers eigenvalues in the simplest case for the decomposition of $A_N$ on its outlier eigenspace. More precisely, we will suppose that $A_N$ has the following decomposition, for some integer $r \geq 1$ and complex number $\theta_N$, 
\begin{equation}\label{eq:decompAN}
A_N =   \left( \begin{array}{c|c}
\theta_N I_r& 0 \\
\hline
0 & \hat{A}_{N-r}  
 \end{array} \right). 
\end{equation}

\begin{theoreme}\label{th:TCL}
Suppose that assumptions (X1-X2) and assumptions (A1-A3) hold with $A_N$ given by \eqref{eq:decompAN}. We suppose further that $\theta_N $ converges toward $\theta \in \mathbb{C} \setminus \mathrm{supp}(\beta )$  when $N$ goes to infinity and that for some $\eta >0$ and all large $N$, $ \hat{A}_{N-r} - \theta I_{N-r}$ has no singular value in $[0,\eta]$.  Finally, assume that either $\dE X_{11}^2 = 0$ or that $\frac{1}{N-r} \Tr \left\{ (  \theta  I_{N-r} - \hat{A}_{N-r}  )^{-1} (  \theta I_{N-r} -\hat{A}_{N-r}^{\top}  )^{-1}  \right\} $ converges to $\psi \in \dC$ (in the first case, we set $\psi = 0$).  We set $\varphi = \int \lambda^{-1} d \nu_{\theta} (\lambda)$.

Then, for any  $0 < \delta< \eta$, almost surely  for all large $N$  there are exactly $r$ eigenvalues
$\lambda_i$, $i=1, \ldots,r$ of $M_N$ in $B(\theta,\delta)$. Moreover, the point process of $\left( \sqrt{N} ( \lambda_1 -\theta_N), \ldots, \sqrt{N} ( \lambda_r -\theta_N) \right) $ converges in distribution towards the point process of the eigenvalues of a $r\times r $ matrix $V $ defined as 
\begin{equation}\label{defV} V =\sigma X_{r}+  \sigma^2  G ,\end{equation}
where $X_r$ is independent of $G$, a $r \times r$ Ginibre matrix whose entries are independent copies of a centered complex Gaussian variable $Z$ whose covariance is characterized by, 
$$
\dE |Z|^2 =    \frac{\varphi}{1 - \sigma^2 \varphi} \quad \hbox{ and } \quad \dE  Z^2   =  \frac{ ( \dE X_{11}^2)^2 \psi}{1 - \sigma^2 \dE X_{11}^2 \psi} . 
$$
\end{theoreme}

Theorem \ref{th:TCL} shows that the fluctuation of stable outliers are not universal (they may depend on the law of entries). There is a similar phenomenon for deformed Wigner matrices, see notably Capitaine, Donati-Martin and F\'eral \cite{MR2489158,MR2919200}.  The left plot of Figure \ref{fig:nummatTCL} illustrates Theorem \ref{th:TCL}.

\begin{figure}[htb]
\begin{center}
\includegraphics[width = 8.1cm]{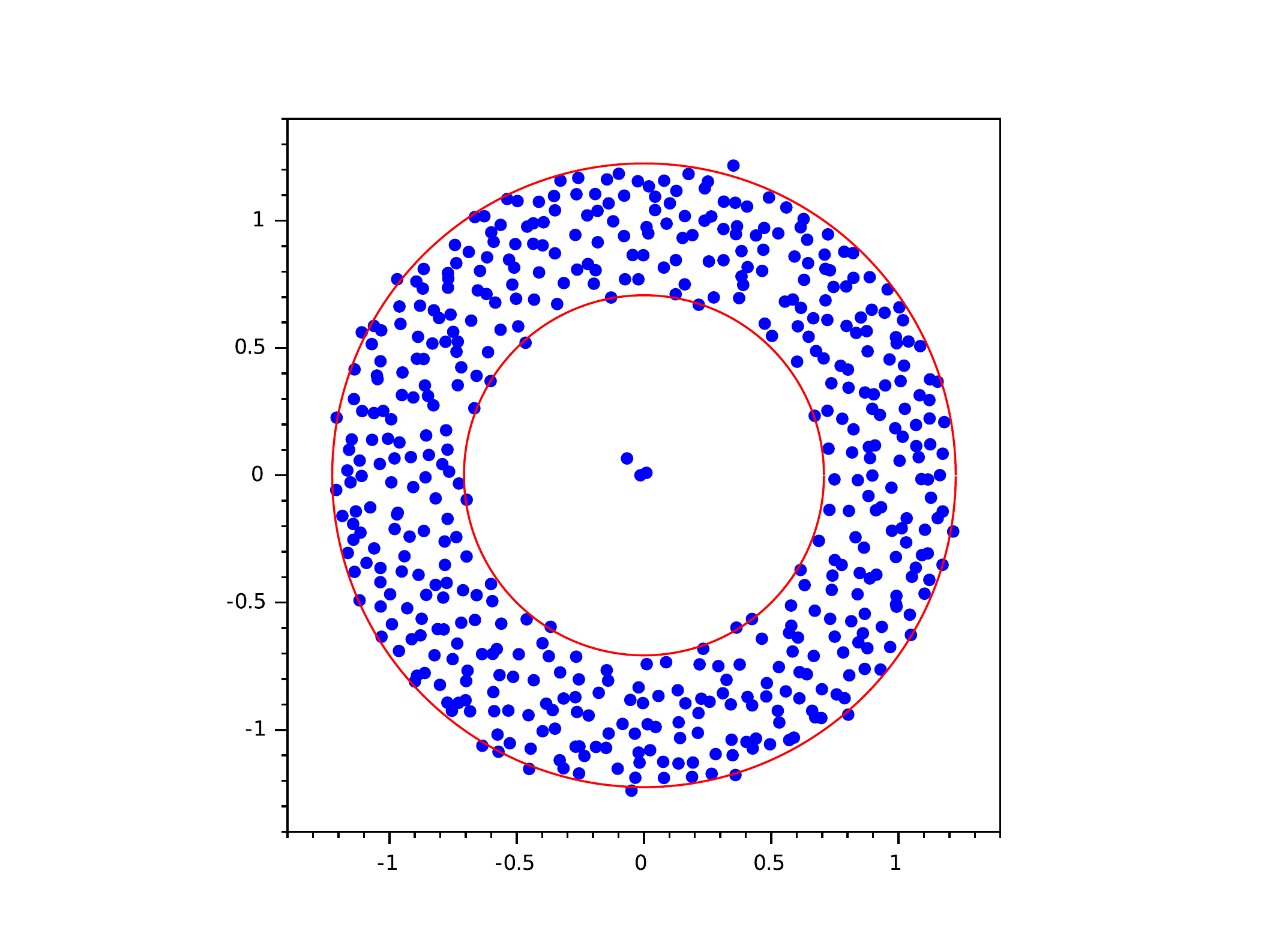}
\includegraphics[width = 8.1cm]{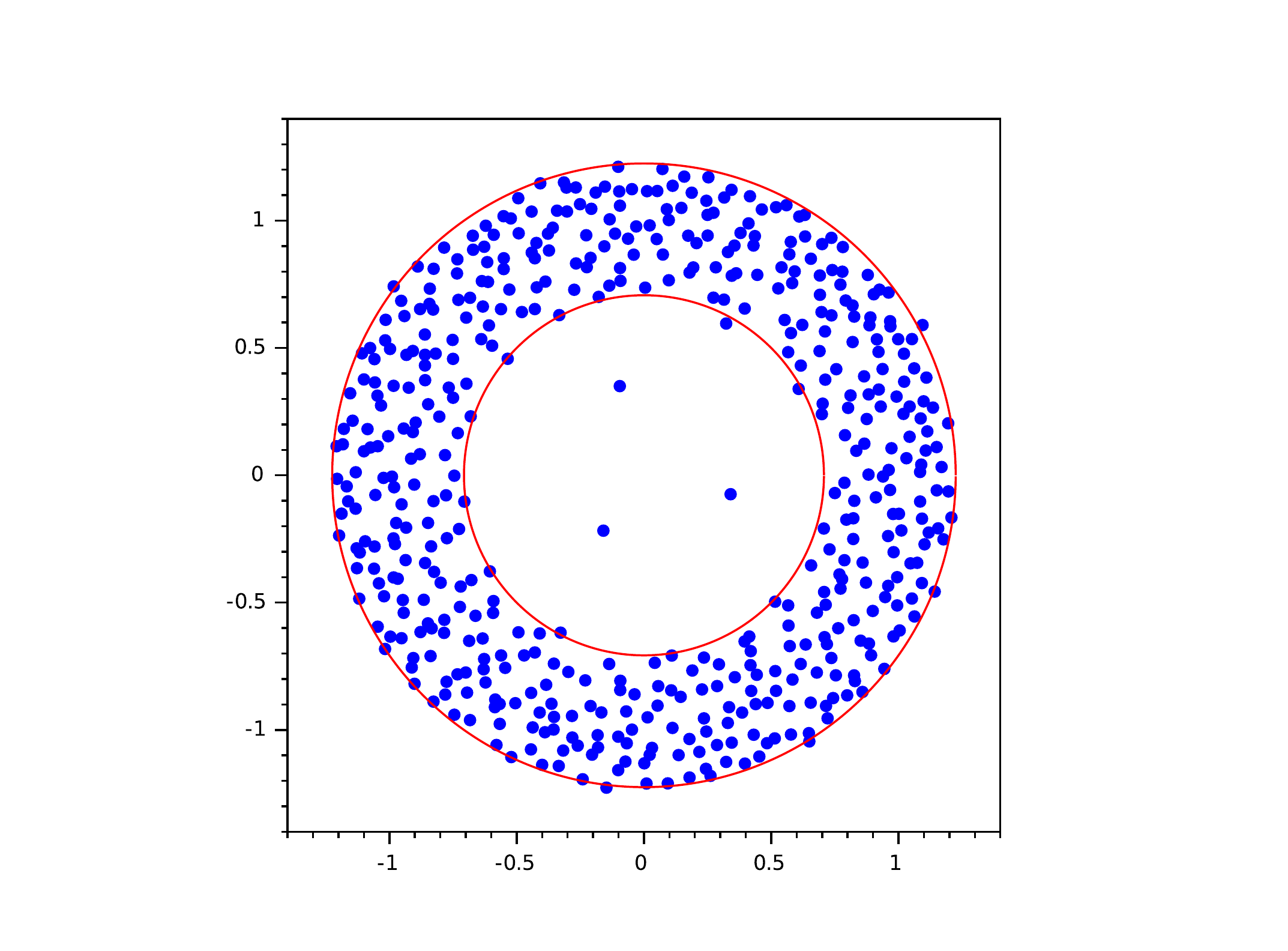}
\end{center}
\caption{Eigenvalues of $M_N$ for $A_N$ given \eqref{eq:ex1A} with $\sigma^2 = 1/2$, $X_{N}$ has complex Gaussian entries, $N = 500$ and $r = 3$. On the left: $B = 0$. On the right: $B  = e_{1} e_2^ * + e_{2} e_3 ^*$, we have $N^{1/(2r)}  = N^{1/6} \simeq 0.35$ ! 
}
\label{fig:nummatTCL}\end{figure}

It was recently discovered by  Benaych-Georges and Rochet \cite{BGR} in the related study of the outliers in the single ring theorem  (see discussion and related results below) that the fluctuations can be larger than $1/ \sqrt N$ when the Jordan decomposition of an eigenvalue is not a diagonal matrix. Inspired by their work, we have also studied the fluctuations when for some integer $r \geq 2$ and complex number $\theta_N$, 
\begin{equation}\label{eq:decompAN2}
A_N =   \left( \begin{array}{c|c}
\hat A_r & 0 \\
\hline
0 & \hat{A}_{N-r}  
 \end{array} \right), \hbox{ with } \quad \quad \hat A_r =  P_N J_N  P_N^{-1}, \quad P_N \in \mathrm{GL}_r(\dC),
\end{equation}
and  $J_N \in M_r (\dC)$ is the Jordan matrix
$$
J_N =  \begin{pmatrix}
  \theta_N & 1 &   &   &   \\
    & \theta_N &  1 &  &   \\
   &  &  \ddots  & \ddots&  
 \end{pmatrix}.
$$

\begin{theoreme}\label{th:TCL2}
Suppose that assumptions (X1-X2) and assumptions (A1-A3) hold with $A_N$ given by \eqref{eq:decompAN2}. We suppose further that $\theta_N $ converges toward $\theta \in \mathbb{C} \setminus \mathrm{supp}(\beta )$  when $N$ goes to infinity and that for some $\eta >0$ and all large $N$, $ \hat{A}_{N-r} - \theta I_{N-r}$ has no singular value in $[0,\eta]$.  Assume finally that $\| P_N - P \| \to 0$ for some $P \in \mathrm{GL}_r(\dC)$, and that either $\dE X_{11}^2 = 0$ or that $\frac{1}{N-r} \Tr \left\{ (  \theta I_{N-r} -   \hat{A}_{N-r}  )^{-1} (  \theta I_{N-r} -\hat{A}_{N-r}^{\top}  )^{-1}  \right\} $  converges to $\psi \in \dC$ (in the first case, we set $\psi = 0$).

Then, for any  $0 < \delta< \eta$, almost surely  for all large $N$  there are exactly $r$ eigenvalues
$\lambda_i$, $i=1, \ldots,r$ of $M_N$ in $B(\theta,\delta)$. Moreover, the point process of  $\left(N^{1/(2r)} ( \lambda_1 -\theta_N), \ldots,N^{1/(2r)} ( \lambda_r -\theta_N) \right) $ converges in distribution towards the point process of the roots of the random polynomial 
$$
z^r - e_r^*  P^{-1} V P e_1,
$$
where $V$  is defined by \eqref{defV}.
\end{theoreme}

When $\hat A_ {N -r} = 0$ and $X_N$ is a complex Ginibre matrix, the above result is contained in \cite[Theorem 2.6]{BGR}.  This result shows the strong correlation of the outlier eigenvalues in the setting of Theorem \ref{th:TCL2}: properly rescaled they are asymptotically the $r$-th roots of the same random complex number. The right plot of Figure \ref{fig:nummatTCL} illustrates Theorem \ref{th:TCL2}.

\subsection{Unstable outliers}

\label{subsec:unsou}
Lemma \ref{le:noouter} does not rule out the possibility that \eqref{eq:ratioAA'} fails to hold in a bounded component of $\dC \backslash S$. Let us give a typical situation where \eqref{eq:ratioAA'} does not hold. Consider the nilpotent matrix $A_N$ given by \eqref{eq:Anil}. We have $A_N = A'_N + A''_N$ with 
$$
A'_N = \sum_{i=1}^{N-1} e_{i} e_{i+1}^*   + e_N e_1^*
\quad \hbox{ and } \quad A''_N = - e_N e_1^*. 
$$
$A'_N$ is a permutation matrix whose eigenvalues are  for $1 \leq \ell \leq N$, $\omega_\ell =  e^{\frac{ 2 i \pi \ell}{ N}}$ with associated normalized eigenvector $f_\ell  = (\omega_\ell ^k  / \sqrt N)_{1 \leq k \leq N}$. In particular, if $|z| \leq 1 - \veps$, 
$$
\left| \frac{ \det ( A_N - z ) }{\det ( A'_N - z) }\right| = \frac{ |z| ^N }{ |1 - z^N |}  \leq  \frac{ ( 1 - \veps)^N}{ 1 -  (1-\veps)^N}
$$
decreases exponentially fast. Hence Assumption \eqref{eq:ratioAA'} does not hold for $\Gamma \subset B ( 0, 1 - \veps)$. In Figure \ref{fig:unstable1}, we see numerically that the conclusion of Theorem \ref{th:main} does not seem to hold. Also, \eqref{eq:Anil} can be interpreted as a limit case of \eqref{eq:decompAN2} when $r =N$. Then, Theorem \ref{th:TCL2} hints at macroscopic fluctuations of outlier eigenvalues.  
\begin{figure}[htb]
\begin{center}
\includegraphics[width = 10cm]{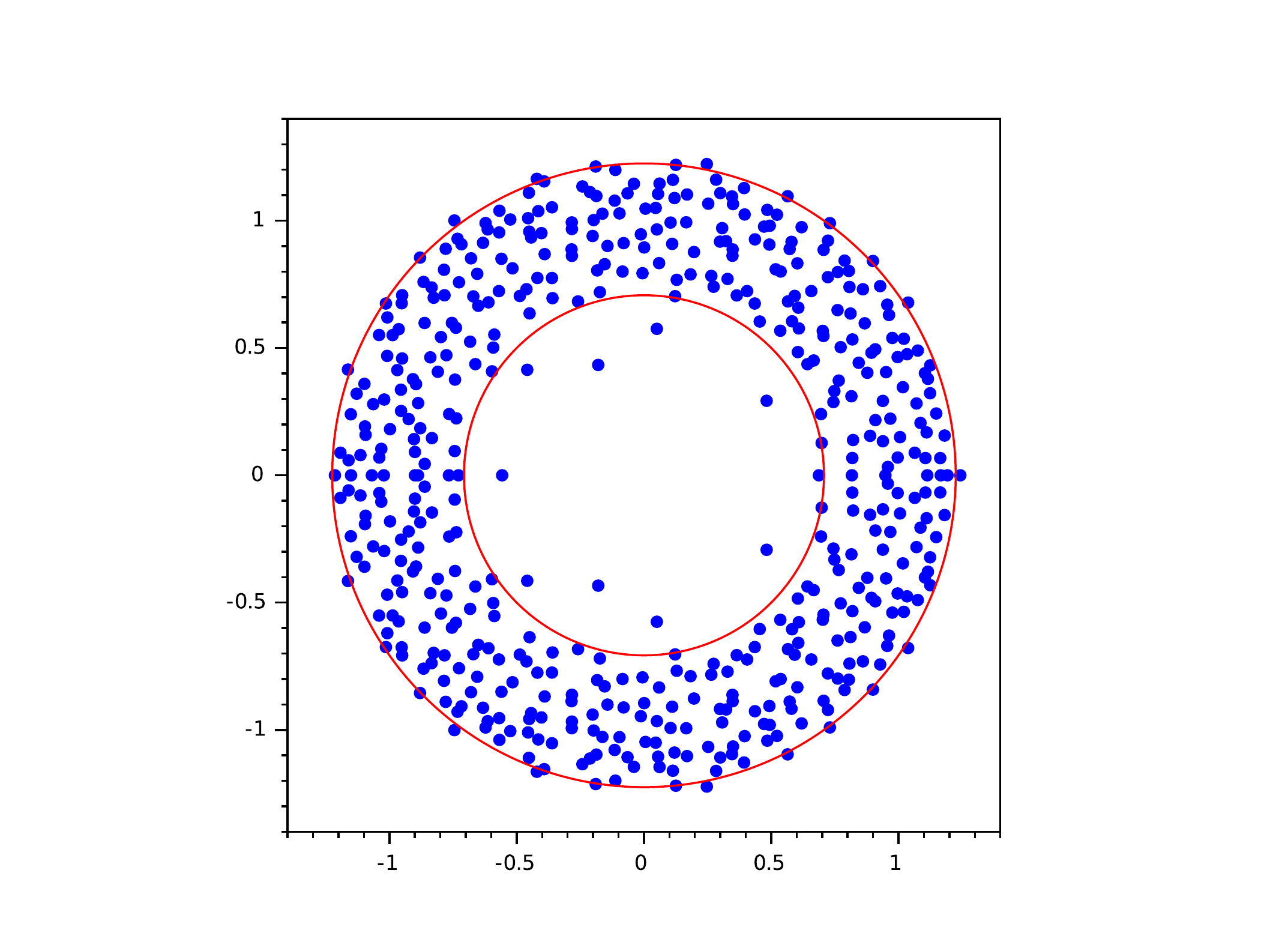}
\end{center}
\caption{Eigenvalues of $M_N$ where $A_N$ given by \eqref{eq:Anil} with $N = 500$,  $\sigma^2 = 1/2$ and $X_N$ has real Gaussian entries.  There are outlier eigenvalues in the bounded component of the complement of $\supp(\beta) = \{z \in \mathbb{C} : 1 / \sqrt 2 \leq |  z |  \leq \sqrt{3/2}\}$.}
\label{fig:unstable1}\end{figure}
To have a better picture, we may rewrite $A_N$ in the orthonormal basis of eigenvectors of $A'_N$. We have $A'_N = U_N B'_N U_N ^*$ with 
\begin{equation}\label{eq:A'Nnil}
B'_N = \mathrm{diag}  ( \omega_1, \cdots , \omega_N) 
\end{equation} 
and $A''_N = U_N B''_N U_N^*$ with 
\begin{equation}\label{eq:A''Nnil}
B''_N = - f_N f_1^\top.
\end{equation}
Hence, an important difference with the setting of Theorem \ref{th:TCL} is that $A''_N$ has a strongly delocalized decomposition in the orthonormal basis of eigenvectors of $A'_N$. This motivates the results of this paragraph.

For simplicity, we will reinforce the assumption (A4) by assuming that $A'_N$ is diagonal.

\begin{enumerate}
\item[(A4')] 
$A''_N$ has rank $r$, $A'_N$ is diagonal and for some $\varepsilon  > 0$ and  all $N$ large enough, all eigenvalues of $A'_N$ lie in $\dC \backslash B(\Gamma, \veps)$. 
\end{enumerate}
If $A'_N$ is a normal matrix and $X_{ij}$ are standard complex Gaussian variables, then, by unitary invariance, we can always assume that (A4') holds if (A4) holds.

We endow the set of analytic functions on a bounded connected open subset $U$ of $\dC$ with the distance
\begin{equation}\label{eq:distHU}
d(f,g)  = \sum_{j \geq 1} 2^{-j} \frac{ \| f -g\|_{K_j} }{ 1 + \| f - g \|_{K_j}},
\end{equation}
where $(K_j)_{j \geq 1}$ is an exhaustion by compact sets of $U$ and $\|f-g\|_{K_j}$ denotes the infinity norm of $f-g$ on $K_j$. We recall that it is complete separable metric space.   The interior of a set $K$ is denoted by $\mathring K$. 

%In order to circumvent the multiple possible choices to order the eigenvalues, we will consider the finite point set $\sum_{i=1} ^n \delta_{x_i}$ of a vector %$(x_1, \cdots, x_n)$.  Let us first recall some basic facts on finite point processes (we refer to Daley and Vere-Jones \cite[Appendix A.2.5]{MR1950431} for %details and terminology). If $\cS$ is a complete separable metric space, we denote by $\mathcal {N}(\mathcal S)$, the set of finite point (integer valued) %measures on $\mathcal S$, equipped with the usual weak topology.  Recall that a point process is random variable on $\mathcal N (\mathcal S)$. The set %$\cP( \mathcal N (\mathcal S)) $ of probability measures on $\cN(\cS)$ is a complete separable metric space and the L\'evy-Prohorov distance is a metric %for the weak convergence of measures in $\cP ( \cN(\cS))$ (or with a slight abuse of language, for weak convergence of point processes on $\cS$). Finally, %recall that the set of analytic functions on a compact set endowed with the uniform norm is complete separable metric space. 

\begin{theoreme}\label{th:nilpotent}
Assume that assumptions (X1-X3) and assumptions (A1'-A4') hold with $r=1$ and $\Gamma \subset \dC \backslash \supp (\beta)$ a compact set with continuous boundary. Assume further that  
$
A''_N =  v_{N} u_{N}^*
$
where $\| u_{N} \|_{\infty},  \| v_{N} \|_{\infty}$ are of order $O ( 1 / \sqrt N)  $ and
\begin{equation}
\max_{z \in \Gamma} \left| \frac { \det ( A_N - z) }{\det ( A'_N -z ) } \right| = o \left( \frac 1 {\sqrt N} \right) \label{eq:ratiodets}.
\end{equation}

Consider the centered Gaussian process $(g_N(z))_{z \in \Gamma}$ with covariance given by, for $z , w \in \Gamma$, 
\begin{eqnarray*}
\dE g_N(z) \bar g_N(w) &=& \frac{  u_N^* R'_N(z) (R'_N(w) )^* u_N v_N^* (R'_N(w) )^*  R'_N(z)  v_N }{ 1 -\sigma^2  \varphi_N(z,w)} \\
\dE g_N(z) g_N(w) &=& \frac{  u_N^* R'_N(z) R'_N(w)  \bar u_N v_N^T (R'_N(w)) ^T  R'_N(z)  v_N  \dE X_{11}^2}{ 1 - \dE X_{11}^2 \sigma^ 2 \psi_N(z,w)} 
\end{eqnarray*}
where $R'_N (z) = ( zI_N - A'_N  )^{-1}$, $\varphi_N(z,w) = \frac 1 N \Tr R'_N(z) (R'_N(w))^*$ and $\psi_N(z,w) = \frac 1 N \Tr R'_N(z) R'_N(w)$. 

Then, $g_N$ is a tight sequence of random analytic functions in $\mathring \Gamma$. Moreover, the L\'evy-Prohorov distance between the point process of eigenvalues of $M_N$ in $\mathring \Gamma$ and the point process of zeros of $g_N$ in $\mathring \Gamma$ goes to $0$ as $N$ goes to $\infty$.  
\end{theoreme}

The intensity of zeros of $g_N$ can be computed explicitly thanks the Edelman-Kostlan's formula, see \cite[Theorem 3.1]{MR1290398}. With the material of this paper, it is possible to generalize Theorem \ref{th:nilpotent} for $r \geq 1$. The analog of condition \eqref{eq:ratiodets} will however be more complicated and the analog of $g_N$ will be the determinant of $r \times r$ random Gaussian matrix (see forthcoming Remark \ref{re:rgeneral}).

In \S \ref{subsec:recipe}, we will give  a general method to find perturbations $A''_N$ such that \eqref{eq:ratiodets} holds. In the specific case of the nilpotent matrix \eqref{eq:Anil}, formulas are simpler. Following the terminology of Hough, Krishnapur, Peres, Vir{\'a}g \cite{MR2552864}, we say that $g$ is a Gaussian analytic function on a domain $\Gamma \subset \dC$, if $g$ is a random analytic function on $\Gamma$, $(g(z), z \in \Gamma)$ is a Gaussian process and  for $z , w \in \Gamma$, $\dE g(z)  g(w) =0.$
% $z_1, \cdots, z_n$ in $\Gamma$,  $(g( z_1 ), . . . , g( z_n ))$ is a mean zero complex Gaussian
%distribution of the form $a x$ with $a \in M_n (\dC)$ and $(x_1, \cdots, x_n)$ i.i.d. complex Gaussian variables with $\dE |x_i |^2 = 1$, $\dE x_i ^2 = 0$. The distribution of a Gaussian analytic function is characterized by its kernel $K(z,w) = \dE g(z) \bar g(w)$ (see \cite[Chapter 2]{MR2552864}). The  Edelman-Kostlan's formula takes a particularly simple form for Gaussian analytic functions: the intensity of zeros of $g$ is given by $\frac 1 {4 \pi} \Delta \log K(z,z)$, where $\Delta$ is the usual Laplacian. 

The next corollary deals with the phenomenon illustrated by Figure \ref{fig:unstable1}.

\begin{corollaire}
\label{cor:nilpotent}
Let $A_N = B'_N + B''_N$ with $B'_N$ and $B''_N$ given by \eqref{eq:A'Nnil} and \eqref{eq:A''Nnil}. Suppose that  assumptions (X1-X3) hold. We set 
\begin{equation}\label{eq:defvarphi}
\varphi(z,w) = \frac 1 { 1 - z \bar w }. 
\end{equation}
  The support of $\beta$ is  $ \{ z \in \dC :  \sqrt{ (1 - \sigma^2)_+ } \leq  |z| \leq \sqrt {1 + \sigma^2} \}$. If $\sigma < 1$, the point process of eigenvalues of $M_N$ in $\mathring B(0,\sqrt {1 - \sigma^2} )$ converges weakly to the zeros of the Gaussian analytic function $g(z)$ on $\mathring B(0,\sqrt {1 - \sigma^2})$ with kernel given by, for $z , w \in \Gamma$, 
\begin{eqnarray*}
K(z,w) = \frac{   \varphi(z,w)^2  }{ 1 -\sigma^2  \varphi(z,w)}.  
\end{eqnarray*}
\end{corollaire}

We may notice the following surprising fact. As $\sigma \to 0$, the kernel $K(z,w)$ appearing in Corollary \ref{cor:nilpotent} does not vanish, it converges pointwise to the kernel $K_0(z,w) = \varphi(z,w)^2$ on the unit complex disc. The kernel $K_0$ is the kernel of the Gaussian analytic function
$$
g(z) = \sum_{k=0}^\infty   z^k \gamma_k\sqrt { k+1}  , 
$$
where $\gamma_k$ are iid complex Gaussian variables with $\dE \gamma_k ^2 = 0$, $\dE |\gamma_k|^2 = 1$. This Gaussian analytic function may thus be related to the numerical phenomenon illustrated by the right plot of Figure \ref{fig:nummat}.

Finally, as in Rajan and Abbott \cite{PhysRevLett.97.188104} and Tao \cite{tao-outliers}, interesting outliers may appear when $\|A''_N\|$ is of order $\sqrt N$.

\begin{theoreme}\label{th:largenorm}
Assume that assumptions (X1-X3) and assumptions (A2-A4') hold with $r=1$ and $\Gamma \subset \dC \backslash \supp (\beta)$ a compact set with continuous boundary. We set $R'_N (z) =  ( z I_N - A'_N)^{-1}$, and assume further that $\|A'_N \|  = O(1)$ and $
A''_N =  \sqrt N v_{N} u_{N}^*
$ where $\| u_{N} \|_{\infty},  \| v_{N} \|_{\infty}$ and $u_N^* R'_N(z) v_N $  are of order $O ( 1 / \sqrt N)  $. 

Consider the centered Gaussian process $(g_N(z))_{z \in \Gamma}$ of Theorem \ref{th:nilpotent}. Then, the L\'evy-Prohorov distance between the point process of eigenvalues of $M_N$ in $\mathring \Gamma$ and the point process of zeros of $1 - \sqrt{N} u_N^* R'_N(z) v_N + \sigma g_N(z)$ in $\mathring \Gamma$ goes to $0$ as $N$ goes to $\infty$.  
\end{theoreme}

\begin{figure}[htb]
\begin{center}
\includegraphics[width = 10cm]{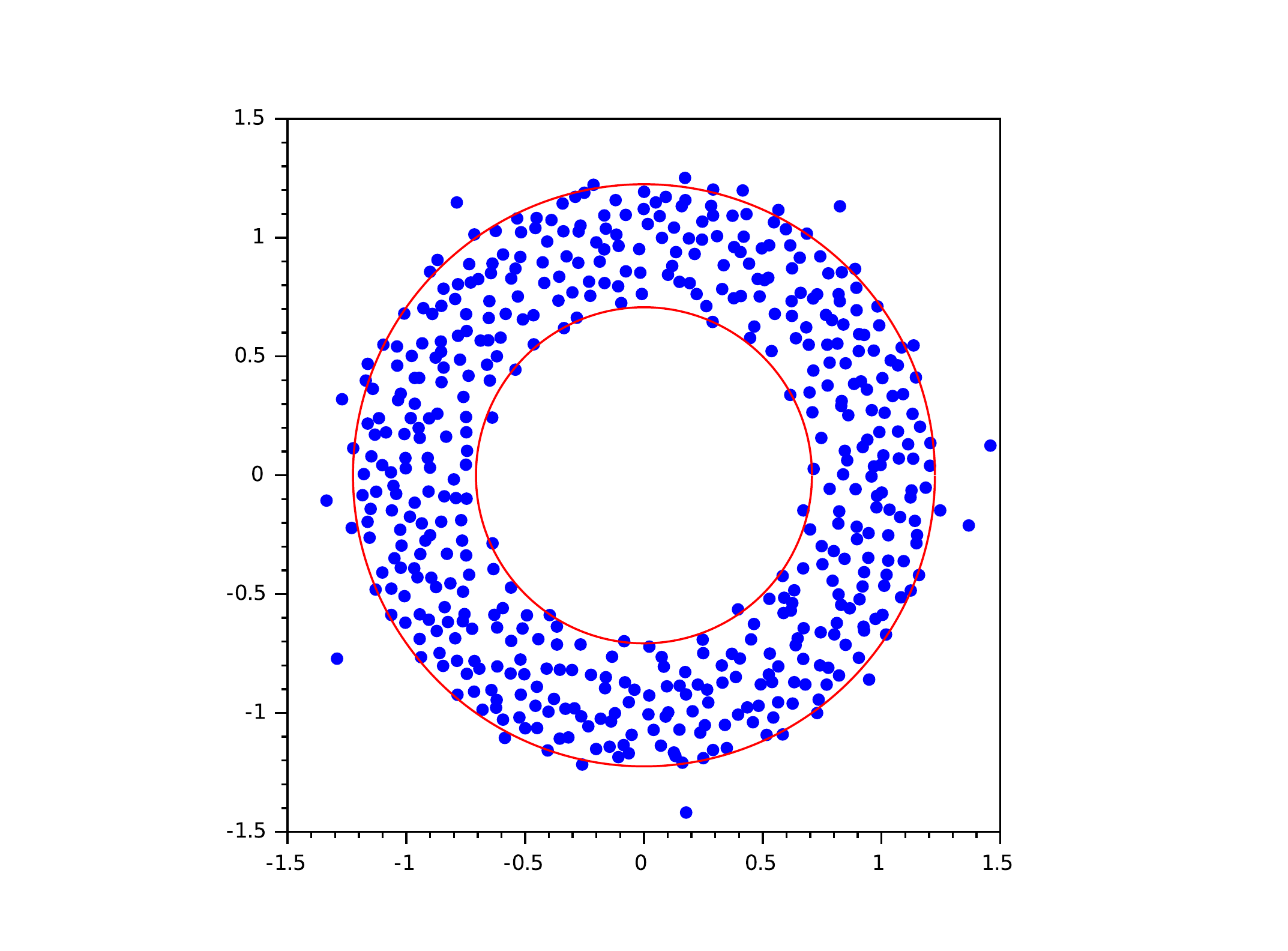}
\end{center}
\caption{Eigenvalues of $M_N$ where $A_N =B'_N - \sqrt N B''_N $ given by \eqref{eq:A'Nnil} and \eqref{eq:A''Nnil}, $N = 500$, $\sigma^2 = 1/2$ and $X_N$ has complex Gaussian entries.  For $|z| > 1$, we have $f_N^T  R'_N(z) f_1 = 1 / ( z^N - 1) = o ( 1 / \sqrt N)$. The outlier eigenvalues in the unbounded component of the complement of $\supp(\beta) = \{z \in \mathbb{C} : 1 / \sqrt 2 \leq |  z |  \leq \sqrt{3/2}\}$ converge in distribution to the zeros of $1 + \sigma g$ where $g$ is the Gaussian analytic function with kernel $H(z,w) =     \varphi(z,w)^2  / ( 1 + \sigma^2  \varphi(z,w)) $ and $\varphi$ given by \eqref{eq:defvarphi}.}
\label{fig:unstable2}\end{figure}

 This result is illustrated with Figure \ref{fig:unstable2}. As an application, we have for example the following corollary which is related to \cite[Theorem 1.11]{tao-outliers}.
\begin{corollaire}\label{cor:largenorm}
Assume that assumptions (X1-X3) hold, that $\sigma = 1$, $A'_N = 0$ and $A_N = A''_N = \theta_N v_N u_N ^T$ with $\sqrt N /  \theta_N  \to  \kappa \in \dC$, $u_N , v_N \in \dR^N$, $\|u_N \| = \|v_N \| = 1$, $ \theta_N u_N^T  v_N  \to \lambda \in \dC$ and $\| u_{N} \|_{\infty},  \| v_{N} \|_{\infty}$ are of order $O ( 1 / \sqrt N) $.  Fix $\veps >0$ and set $\Gamma = \dC \backslash B(0, 1 + \veps)$.  We consider $g(z) = \sum_{k \geq 0} \gamma_k z^{-k}$ with $\gamma_k$ independent complex Gaussian variables with variance given by $
\dE |\gamma_k |^2  = 1$ and $\dE \gamma_k ^2 = ( \dE X_{11}^2)^{k+1}$.

Then,  as $N$ goes to $\infty$,  the point process of eigenvalues of $M_N$ in $\Gamma$ converges vaguely to the point process of zeros of $\kappa  z ( z  - \lambda )  +  g (z)$ in $\Gamma$.  
\end{corollaire}

Observe that $ \theta_N  u_N^T  v_N$ is the outlier eigenvalue of $A_N$. Hence, in the limit $|\kappa| \to \infty$, we find a result consistent with Theorem \ref{th:main}.   
Moreover, when $\dE X_{11} ^2 = 0$, $g(z)$ is a Gaussian analytic function, and, if  $\kappa = 0$, we are interested by the level set zero of this Gaussian analytic function. From Peres and Vir\'ag \cite{MR2231337}, it is known that this level set forms a determinantal point process.

\subsection{Discussion and related results}

In a recent work \cite{BGR}, Benaych-Georges and Rochet consider matrices of the type $M_N=X_N + A_N$ where the rank of $A_N$ stays bounded as the dimension goes to infinity and $X_N$ is a random non Hermitian matrix whose distribution is invariant under the left and right actions of the unitary group. The limiting empirical eigenvalues distribution of such a model is described by the so-called single ring theorem, see \cite{MR2831116,MR3000558,MR3164983}, and its  support is of the form
$\{z, a \leq  \vert z \vert \leq b\}$. Benaych-Georges and Rochet  prove that if $A_N$ has some eigenvalues out of the maximal circle of the single ring, then $M_N$ has outliers in the neighborhood of these eigenvalues of $A_N$. Nevertheless, when $a>0$, the eigenvalues of $A_N$ which may be in the inner disk of the 
complement of the limiting support do not generate outliers in the spectrum of 
$M_N$.

Now, in the framework of the present paper dealing with full rank perturbations of iid matrices, there can be outlier eigenvalues in bounded components of the complement of $\supp (\beta)$, see the example given by \eqref{eq:ex1A} and Figure \ref{fig:stable}.

Actually, the nature of the bounded connected component of the complement of the support of the limiting empirical eigenvalues distributions considering above is different: the first (in \cite{BGR}) comes from the limiting support of the non-deformed model whereas the second one (in the framework of our paper) is created by the deformation. Subordination-like properties of the Stieltjes transform $g_{\mu} (z) = \int d \mu ( \lambda)/ (z -  \lambda )$ of limiting spectral measures may help to understand these phenomena as explained below. 

In the case of \cite{BGR}, since  the limiting  empirical eigenvalues distribution $\mu$  is radial, we have  $g_\mu(z)=\frac{1}{z} $ if $\vert z \vert >b$ and 
$g_\mu(z)=0 $ if $\vert z \vert <a$ so that roughly speaking $$g_\mu(z)= g_{\delta_0}(\omega(z))~~
\mbox{~~where~~} \left\{ \begin{array}{ll} \omega(z) =z \mbox{~if~} \vert z \vert >b\\ \mbox{``}\omega(z) =\infty\mbox{"~if~} \vert z \vert <a. \end{array} \right.$$

In our case, dealing for instance with diagonal perturbations of a Ginibre matrix,  the limiting  empirical eigenvalues distribution $\beta$ is the Brown measure of $c+a$ where $c$ is a circular element which is free with $a$ whose Brown measure is $\alpha$ (see {\'S}niady \cite{MR1929504}). We have the following subordination property.
$$\forall z \in \mathbb{C} \setminus \supp(\beta), \; g_{c+a}(z)= g_a(\omega(z)) \mbox{~~where~~} \omega(z)=z.$$
It can be deduced from \cite[Proposition 4.3]{bordenave-caputo-chafai-cirgen}, see also \cite{MR1744647,MR1876844}.

In both cases, the intuition is that $$g_{\mu_{M_N}}(z)\approx  g_{\mu_{A_N}}(\omega(z))$$ and that therefore  they will be eigenvalues $\rho$  of $M_N$ that separate from the bulk
whenever some of the equations
$\omega(\rho) =\theta$ admits a solution $\rho$ outside the limiting  support, when  $\theta$ describes the spectrum of $A_N$. 

Therefore, we understand in one hand  that in the framework of \cite{BGR}, there is no solution inside the inner  disk of such an equation since there $\omega\equiv\infty$ and in the other hand that  the outliers of the deformed model stay in the neighborhood of the eigenvalues of the perturbation 
which are located where $\omega$ is the identity function.

In \cite{FPZ}, Feldheim, Paquette and Zeitouni have recently studied the model \eqref{modele} when $\sigma$ decays polynomially of $N$ and $A_N$ is a block diagonal matrix with blocks of size $\log N$. 

Note that  motivated by the seminal article of Baik, Ben Arous and P\'ech\'e \cite{MR2165575}, previous works are devoted to the study of the outlier eigenvalues of  deformed random Hermitian  matrix models,  see notably \cite{BFF,MR2887686,BBCF,MR2835249,MR2782201,MR2944410,MC,MR2489158,MR2919200,MR2835253,MR2291807,
MR3103909,KY,
MR2851051,MR2336602,MR2221787,MR3060148,SR,MR3039820}. 
Actually, the papers  \cite{BBCF,MC,MR2835253} already  show that the  results on existence and location of outliers of
deformed Wigner matrices/ deformed unitarily invariant matrices/ sample covariance matrices/ information-plus-noise type matrices can be completely described in terms of 
subordination functions involved in free additive/multiplicative/rectangular convolutions.
Thus, free subordination properties definitely seem to provide a global explanation for the problem of existence and location of outliers of deformed random matrix models.

In the Hermitian deformed models, the outliers of the deformed model are not located in a neighborhood of the spikes of the deformation. It contrasts with Corollary \ref{cor:main}. It is rather non-intuitive that additive perturbation of $A_N$ by a  Hermitian random matrix has more effect on outlier eigenvalues than additive perturbation by a non-Hermitian random matrix.

The remainder of the papers is organized as follows. In Section \ref{sec:prelimS}, we review some properties of the limiting spectral measures $\beta$ and $\mu_z$ and recall basic matrix identities, we notably prove Proposition \ref{prop:supportbeta} and Lemma \ref{le:noouter}. In Section \ref{sec:nooutlier} and Section \ref{sec:main}, we prove the first order results, Theorem \ref{inclusion} and Theorem \ref{th:main} respectively. In Section \ref{sec:TCL}, we prove the central limit theorem for stable outliers, Theorem \ref{th:TCL}. In Section \ref{sec:unstable}, we prove all results concerning unstable outliers, Theorem \ref{th:nilpotent}, Corollary \ref{cor:nilpotent}, Theorem \ref{th:nilpotent} and Corollary \ref{cor:nilpotent}. Finally, an appendix contains a central limit theorem for a random bilinear form.

\section{Preliminaries on limiting spectral measures and useful matrix relations}
\label{sec:prelimS}

\subsection{Useful matrix identities and perturbation inequalities}

If $A \in M_N ( \dC)$ we will denote by $s_N (A) \leq \ldots \leq s_1 ( A)$ the singular values of $A$. We have $s_1(A) = \|A\|$ and $s_N (A) ^{-1} = \|A ^{-1} \|$ where $\|A \|$ denotes the operator norm. In this work, we will repeatedly use the following classical perturbation inequalities. If $A,B$ in $M_n(\dC)$ then
\begin{equation}\label{eq:CF}
| s_i ( A) - s_i (B) | \leq \|A - B \|.  
\end{equation}
It follows immediately from Courant-Fisher variational formulas for the singular values, see e.g. \cite[Theorem A.46]{MR2567175}. We will often apply it in the following context, if $A - z I_N$ has no singular value in the interval $[0,\eta]$ and $\| B \| \leq \eta / 3$ then 
\begin{equation}\label{eq:CFres}
\sup_{ w : |z - w | \leq \eta /2 } \| ( A  - w I_N )^{-1} \| \leq 2 \eta^{-1} \quad \hbox{ and } \quad  \sup_{ w : |z - w | \leq \eta /3 } \| ( A + B  - w I_N )^{-1} \| \leq 3 \eta^{-1}.
\end{equation}
Similarly, if $K\subset \dC$ is compact and for all $z \in K$, there exists $\eta_z >0$ such that $A-zI_N$ has no singular value in $[0,\eta_z]$ then, there exists $C_K$ such that  
\begin{equation}\label{eq:CFresunif}
\sup_{ z  \in K } \| ( A - z I_N )^{-1} \| \leq C_K,
\end{equation}
($K$ is covered by $\cup_{z \in K} B(z, \eta_z / 2)$ and use compactness). 

As in previous works such as \cite{MR2782201,MR2944410}, we will use the identity, if $P, Q^{\top} \in M_{N,r} (\dC)$, 
$$
\det ( I_N + PQ ) = \det ( I_r + QP ).
$$
It will imply notably that for any  $M \in M_N( \dC)$ and $\lambda \in \dC$ such that $M_N - \lambda I_N$ is invertible,  
\begin{equation}\label{eq:idcle}
\det ( \lambda I_N - M - PQ ) = \det ( \lambda I_N - M ) \det ( I_r -  Q ( \lambda I_N - M)^{-1} P). 
\end{equation}
In particular the eigenvalues of $M + PQ$ which are not eigenvalues of $M$ are the zeros of an $r \times r$ determinant. With $M = A'_N +\sigma Y_N$ and $PQ = A''_N$, the above identity will be our starting point to study the outlier eigenvalues. 

\subsection{Characterization of the limit measure $\beta$}
\label{subses:characbeta}

For a probability measure $\tau$ on $\C$ such that $\int \log (1 + |\lambda|) d\tau (\lambda) < \infty$, we denote by $h_\tau$ its logarithmic potential defined for $z \in \C$,  by
$$
h_\tau (z) = - \int_{\C} \log | \lambda - z| d \tau (\lambda). 
$$
%We have $h_{\tau} ( z) \in[-\infty ; +\infty)$. However, Fubini's Theorem and the local integrability of $z \mapsto \log |z|$ with respect to the Lebesgue measure imply that $h_{\tau}$ is finite for allmost all $z \in \C$. 

There are various possible characterizations of the limit measure $\beta$, the usual relies on its logarithmic potential.  It is expressed in terms of Cauchy-Stieltjes transform of the limit measures of shifted singular values of $M_N$. More precisely, for a probability measure $\tau $ on $\R$, 
 denote by $g_\tau $ its Stieltjes transform defined for $z \in \C\setminus \R$ by 
$$g_\tau (z) = \int_\R \frac{d\tau (x)}{z-x}.$$

For any $z \in \mathbb{C}$,  denote by 
$$M_N^z= \sigma Y_N+A_N-zI_N.$$
According to Dozier and Silverstein \cite{DozierSilver},  almost surely the empirical spectral measure 
$\mu _{M_N^z {M_N^z}^*}$ of ${M_N^z {M_N^z}^*}$ converges weakly  towards a nonrandom distribution $\mu_z$  which is characterized in terms of its Stieltjes transform which satisfies the following equation:
 for any $w \in \mathbb{C}^+$,
\begin{equation}\label{eqgmuz} g_{\mu_z}(w)=\int \frac{1}{(1-\sigma^2g_{ \mu_z}(w))w- \frac{ t}{1- \sigma^2 g_{ \mu_z}(w)} }d\nu_z(t).\end{equation}

According to \cite{tao-vu-cirlaw-bis,MR1929504}, see also \cite{MR2908617}, almost surely the empirical spectral measure of $\mu_{M_N}$ converges weakly  to a probability measure $\beta$ on $\C$ which is characterized by its logarithmic potential
$$
h_{\beta} (z) = - \frac 1 2 \int \log (t ) d\mu_z(t).  
$$
The probability measure $\beta$ has a natural interpretation within the framework of operator algebra and free probability, see \cite{MR1929504,MR1876844,MR1784419}.

Explicit computation of $\beta$ are rare, see Biane and Lehner \cite{MR1876844}. There is  also an alternative characterization based on the limit of a quaternionic resolvent of $M_N$, for details and references we refer to the survey \cite[Section 4.6]{MR2908617} and Rogers \cite{rogers2010}. Using this characterization, when for all $z \in \dC$, $\nu_z$ is the law of $|L-z|^2$, the density of $\beta$ has a tractable expression.  Let $L$ be a random variable with law $\alpha$.   In this case, almost surely, $\beta$ is the limit spectral distribution of the sequence of matrices $\sigma X_N/\sqrt N + D_N$ if $D_N$ is a diagonal matrix whose empirical spectral measure converges weakly to $\alpha$ and $\| D_N \| \leq M$. We set 
\begin{equation}\label{eq:defSigma}
\Sigma = \{ z \in \C : \E |L - z |^{-2} > \sigma^{-2} \}. 
\end{equation}
Observe that, from Fatou's lemma, $\Sigma$ is a an open set. There exists a unique function $f : \Sigma \to (0,+\infty)$ such that
  for all $z\in \Sigma$,
  \[
  \E \frac{1}{|L-z|^2 + f(z)^2} =\sigma^ {-2}.
  \]
The map $f$ is $C^\infty$ on $\Sigma$ (see \cite[Section 4.4]{bordenave-caputo-chafai-cirgen}). On $\Sigma^c$, we set $f(z) = 0$. We introduce the $\C^2 \to \R_+$ function
$$
\Phi(w,z) = \left\{ \begin{array} {ll}
(|w-z|^2 + f(z)^2)^{-2} & \hbox{if $z\in \Sigma$ } \\
0 & \hbox{if $z \notin \Sigma$.}
\end{array} \right. 
$$
It is shown in \cite{bordenave-caputo-chafai-cirgen} that $\beta$ admits a density on $\C$ with respect to Lebesgue measure given by 
\begin{equation}\label{eq:densitybeta}
\rho(z) =     \frac{1}{\pi} f(z)^2 \E \Phi(L,z) %
    + \frac{1}{\pi}
    \frac{\left|  \E  (L-z)\Phi(L,z) ^2  \right| }{\E \Phi(L,z)}.
\end{equation}
Note that $\Sigma$ is the set of $z \in \C$ such that $\rho(z) >0$.  In particular, the support of $\beta$ is $\overline \Sigma$.

\subsection{Properties of the support and proof of Proposition \ref{prop:supportbeta}}

\label{sec:prelim}

Proposition \ref{prop:supportbeta} is direct consequence of assumption (A3) and Proposition \ref{zeroinclus} below established in Capitaine \cite[Theorem 1.3 A)]{MC}.

\begin{proposition}\label{zeroinclus} 
Suppose that assumptions (X1) and (A1-A2) hold. Then $0 \notin \mathrm{supp}(\mu_z)$ if and only if $0\notin \mathrm{supp}( \nu_z)$ and  $g_{\nu_z}(0) >-\sigma^{-2}$, or equivalently, if and only if $z \notin S$ and $\int \lambda^{-1} d \nu_z (\lambda) < \sigma^{-2}$. 
\end{proposition}
%\begin{remarque}\label{boundedsupport}
%\textcolor{red}{VOIR CE DONT ON A REELLEMENT BESOIN COMME HYPOTHESE SUR $\nu_z$}
%\end{remarque}

The characterization of the  complement in $\mathbb{R}\setminus \{0\}$ of the support of $\mu_z$ established in \cite{DozierSilver2} and  Proposition \ref{zeroinclus} allow the author in \cite{MC}  to put forward the following complete characterization of the  complement of the support of $\mu_z$  in $\mathbb{R}$. 

\begin{proposition}\label{support1t} Suppose that assumptions (X1) and (A1-A2) hold. Then,

\begin{equation}\label{support1}\mathbb{R}\setminus\supp(\mu_z) = \Phi_{\nu_z}\left\{ u \in   \mathbb{R}\setminus\supp(\nu_z), \Phi'_{\nu_z} (u) >0, g_{\nu_z}(u) >-\sigma^{-2}\right\}, \end{equation}
where $$\Phi_{\nu_z} :\begin{array}{ll} \mathbb{R}\setminus\supp(\nu_z) \rightarrow \mathbb{R}\\
x \mapsto 
x (1+\sigma^2g_{ \nu_z}(x))^2. \end{array}$$
More precisely, 
 $\Phi_{\nu_z} $ is a homeomorphism from $${\cal E}_{\nu_z}=\left\{ u \in   \mathbb{R}\setminus\supp(\nu_z), \Phi'_{\nu_z} (u) >0, g_{\nu_z}(u) >-\sigma^{-2}\right\}$$ onto $\mathbb{R}\setminus\supp(\mu_z)$ with inverse $\omega_{\nu_z}$,
$$\omega_{\nu_z} :\begin{array}{ll} \mathbb{R}\setminus\supp(\mu_z) \rightarrow \mathbb{R}\\
x \mapsto 
x (1- \sigma^2 g_{ \mu_z}(x))^2. \end{array}$$ 
Moreover, for any $y> x  $ in ${\cal E}_{\nu_z}$, we have $\Phi_{\nu_z}(y) > \Phi_{\nu_z}(x)$,  and respectively  for any $y> x  $ in  $\mathbb{R}\setminus\supp(\mu_z)$, $\omega_{\nu_z}(y) > \omega_{\nu_z}(x)$.

\end{proposition}

The following corollary readily follows.
\begin{corollaire}
\label{cor:support1}Suppose that assumptions (X1) and (A1-A2) hold. For any  $z $ be in $ \mathbb{C}$, we have $$\DIST(0, \supp(\nu_z)) \geq \DIST(0, \supp(\mu_z)).$$
\end{corollaire}
\begin{proof} From Proposition \ref{zeroinclus}, if $\mbox{dist}(0, \supp(\nu_z)) = 0$ then $\mbox{dist}(0, \supp(\mu_z)) = 0$. Moreove, suppose that $\mbox{dist}(0, \supp(\mu_z)) > \veps >0$. Then, according to Proposition \ref{support1t}, $\omega_{\nu_z}([0,\veps]) \subset {\cal E}_{\nu_z}$ and 
\begin{eqnarray*} \omega_{\nu_z}([0,\veps])=[\omega_{\nu_z}(0),\omega_{\nu_z}(\veps)]=[0, \omega_{\nu_z}(\veps)].
\end{eqnarray*}
Now,  $\omega_{\nu_z}(\veps)= \veps (1-  g_{ \mu_z}(\veps))^2 \geq \veps$
since $g_{ \mu_z}(\veps) \leq 0 $ . \end{proof}

%\begin{lemme}\label{le:Sclosed}
%Under assumption (A1)-(A2), the set $S$ is closed. 
%\end{lemme}
%
%\begin{proof}
%Let $z_k \in S$ and $z_k \to z$, we want to prove that $0 \in \supp (\nu_z)$, or equivalently, that for any $\veps >0$, $\nu_z ( [0,\veps] ) > 0$. We fix $\veps >0$.  By assumption (A1), there exists $c >1$ such that $2 \|A_N\| + |z_k| + |z| \leq c$, for all $k,N$. We may find an integer $k$ such that $|z - z_k | \leq \veps / (2c)$. Since $z_k \in S$, there exists $\delta >0$ such that $\nu_{z_k} ( [0,\veps/2)) \geq 2  \delta$. By assumption   (A2) and the portemanteau theorem, $\mu_{(A_N - z_k I_N) ( A_N - z_k I_N) ^* } ( [0,\veps/2)) \geq \delta$ for all $N$ large enough. 
%
%By \eqref{eq:CF}, for any $B \in M_N( \dC)$, if the eigenvalues of $\lambda_i( B - z) ( B-z ) ^*)$ are ordered non-increasingly, \begin{eqnarray*}
%| \lambda_i ( (B- z I)(B-zI)^*)   - \lambda_i ( ( B - z' I ) ( B - z' I)^* ) |  &=& |s_i^2 ( B- z I ) - s_i ^2 ( B - z' I) | \nonumber\\
%&  \leq& |z- z'|  ( 2 \| B \| + |z| + |z'|).\label{eq:CF2} 
%\end{eqnarray*}
%In particular, from our choice of $c$, we deduce that 
%$
%\mu_{(A_N - z I_N) ( A_N - z I_N) ^* } ( [0,\veps)) \geq \delta. 
%$
%A final use of the portemanteau theorem gives that $
%\nu_{z} ( [0,\veps]) \geq \delta. 
%$
%\end{proof}

\begin{lemme}\label{continuitephi}
Under assumption (A2)-(A4), the function $ \varphi  : z \mapsto \int \lambda^{-1} d\nu_z (\lambda)$ is continuous and subharmonic in $\Gamma$. 
\end{lemme}

\begin{proof}
Let $z_0 \in \Gamma$, by assumption (A4), there is no singular value of $A'_N - z_0 I_N$ in $[0,\eta]$. By \eqref{eq:CFres}, for all $z \in B(z_0, \eta/2)$,   $\mu_{(A'_N -z I_N)(A'_N -z I_N)^* }( [0,\delta)) = 0$ with $\delta=(\eta/2)^2$. By assumption (A2) and using the equicontinuity of the $\varphi_N$ on the compact set $ B(z_0, \eta/2)$, the function 
$$\varphi_N (z) = \frac 1 N \Tr  ( (A'_N -zI_N)^{-1} (A'^*_N -\bar zI_N)^{-1} ) =  \int \lambda^{-1} d\mu_{(A'_N -zI_N)(A'_N -zI_N)^* } (\lambda)$$
converges uniformly to $\varphi(z)$ on $B(z_0,\eta/2)$ and $\varphi$ is continuous on $ B(z_0, \eta/2)$. Also, on $B(z_0, \eta/2)$, by \eqref{eq:CFres}
\begin{eqnarray}
| \varphi_N (z)  - \varphi_N (z_0)  |  =  \frac 1 N  \left| \sum_{i=1} ^N  s^{-2} _i ( A'_N- z I_N) -  s^{-2} _i ( A'_N- z_0 I_N)  \right| \leq  C | z - z_0 |,\label{eq:contphi}
\end{eqnarray}
where $C  = 16 \eta^{-3}$ and we have used  $| x^{-2} - y^{-2} | \leq 2 |x - y | / ( |x| \wedge |y| )^3$.  
%The continuity  of $\varphi$ follows from Arzel\`a-Ascoli Theorem. 
Moreover, since $\partial (A'_N - z)^{-1} = (A'_{N} - z I )^{-2}$, we find 
$$\Delta \varphi_N (z)  = 4 \bar \partial \partial \varphi_N (z) = \frac 4 N \Tr ( (A'_N -zI_N)^{-2} (A'^*_N -\bar zI_N)^{-2} )  \geq 0. $$
Consequently, $\varphi_N$ is subharmonic on $B(z_0, \eta/2)$, and, since the convergence is uniform, $\varphi$ is also subharmonic on $B(z_0, \eta/2)$. \end{proof}

\subsection{Case $\nu_z$ law of $|L-z|^2$}
In the subsection, we prove the following lemma. 

\begin{lemme}\label{le:A3normalcase}
Suppose that assumptions (X1), (A1-A2) hold and that for all $z\in \dC$, $\nu_z$ is the law of $|L-z|^2$ for some complex random variable $L$.
Denote by $\alpha$ the distribution of $L$ and assume moreover that for any $z_0$ in  $\mathbb{C} \setminus \mathrm{supp}(\alpha)$, $z \mapsto \int \frac{d\alpha(s)}{\vert s-z\vert^2}$ is not constant on any neighborhood of $z_0$.
 Then assumption (A3) is satisfied. 
\end{lemme}
\begin{proof}
 Observe that $S = \supp (\alpha)$. We set $\varphi (z ) = \int \lambda^{-1} d\nu_z (\lambda) = \dE |L - z |^{-2}$. From \eqref{eq:densitybeta}, the support of $\beta$ is given by $\overline \Sigma$ where $\Sigma$ is defined by \eqref{eq:defSigma}. From Proposition \ref{support1t}, it is sufficient to prove that 
$$
 \left(  \mathrm{supp} (\beta)  \right)^c     = \{ z  \in \C :  z \notin  S \hbox{ and } \varphi(z) < \sigma^{-2}\}.
$$ 
 Assume first that $ z \notin  S \hbox{ and } \varphi(z) < \sigma^{-2}$. Since $S = \supp( \alpha)$, it is closed and there exists an open ball $B_2 = B(z,2 r_0)$ with center $z$ and radius $r_0 >0$ such that $B_2 \cap S  =\emptyset $. In particular, with $B_1 = B (z, r_0)$, the $B_1 \to \R_+$  map $ \varphi$ is bounded and continuous.. Since $\varphi(z) <\sigma^{-2}$, there exists $B_0  = B(z,r_0) \subset B_1$ such that $\varphi (u) <\sigma^{-2}$ on $B_0$.  Hence the density of $\beta$ is $0$ on $B_0$ and $z \notin \mathrm{supp} (\beta)$.

The other way around. Assume that $z \notin \mathrm{supp} (\beta)$ or equivalently $z \notin \overline \Sigma$. Then there exists $d>0$ such that the open ball $ B = B(z,d)$ satisfies $B \cap \overline \Sigma = \emptyset$. In particular, $B\cap \Sigma = \emptyset$, and, for all $u \in B$, $\varphi ( u) \leq \sigma^{-2}$.  Let us first check that $z \notin S =  \mathrm{supp}(\alpha) $. By contradiction : if $z \in \supp(\alpha)$, then for any $\epsilon >0$, there exist $u \in B(z,\epsilon)$, $\tau$ and $\delta >0$ such that for all $0 \leq t \leq \tau$, $\alpha (B(u,t) ) \geq \delta \pi t^2$. It follows that 
$$
\varphi(u) = \int | s - u |^{-2}  d \alpha (s) \geq 2 \pi \delta \int_0^\tau t^{-2}  t dt = + \infty.    
$$
Applied to $\epsilon = d$, it leads to a contradiction. Hence $ z \notin \supp(\alpha)$. Arguing as above, for some, $r_0 >0$, $B (z, 2r_0) \cap \supp(\alpha) = \emptyset$. It follows that on $B_0 = B(z, r_0)$, the map $\varphi$ is bounded and continuous. Moreover, it is subharmonic on $B_0$. We can assume without loss on generality that $r_0 < d$. We may now finish the proof: it remains to check that  $\varphi (z) < \sigma^{-2}$. We know a priori that for all $u \in B$, $\varphi ( u) \leq \sigma^{-2}$.   Assume by contradiction that $\varphi(z) = \sigma^{-2}$. Then the maximum principle  implies that $\varphi = \sigma^{-2}$ on $B_0$. We get  a contradiction.     
\end{proof}

\subsection{Proof of Lemma \ref{le:noouter}}

Let $D$ be as in Lemma \ref{le:noouter}. We may write $A''_N = P_N Q_N$ with $P_N, Q_N^\top \in M_{N,r} ( \dC)$, and by assumption (A1'), $\|P_N\| = 1$, $\|Q_N\| \leq \sqrt r M$. 
From \eqref{eq:idcle}, we find if $z \in D$, 
\begin{equation}
f_N (z) = \frac{ \det ( A_N - z I_N) }{\det (A'_N - z I_N) }  =  \det ( I_r + Q_N (A'_N - z I_N )^{-1} P_N ). 
\end{equation}

By assumption (A4),  $f_N(z)$ is a uniformly bounded analytic function in $D$. In particular, from Montel's theorem, $f_N$ is a precompact and any accumulation point $f$ of $f_N$ is a bounded analytic function on $D$.

Observe moreover that for any $\delta$, for all $z \in \dC$ with  $|z|$ large enough,  $\| Q_N (A'_N - z I_N )^{-1}  P_N\| < \delta$. We use  the crude inequality, for any $B,C \in M_{r} ( \dC)$ 
\begin{equation}\label{eq:contdet}
|\det ( B + C ) - \det (B) | \leq r \| C\|  \left( \| B \| \vee \|B+C\|\right) ^{r-1}.
\end{equation} 
We deduce that  $| f_N(z) | \geq \det ( I_r) - 1/2 = 1/2$ for all $z \in D$ with $|z|$ large enough. 

It  follows that any accumulation point $f$ of $f_N$ is a non-zero bounded analytic function on $D$. In particular, $f$ has a finite number of zeros on any compact subset of $D$. Lemma \ref{le:noouter} follows easily.

\section{No outlier: proof of Theorem \ref{inclusion}}
\label{sec:nooutlier}

Theorem \ref{inclusion} is a direct consequence of the following proposition.

\begin{proposition}\label{nonunif} Suppose that assumptions (X1-X2), (A1), (A2) and (A4) hold.
Let $z$ be in $\Gamma$ such that $0 \notin \supp (\mu_z)$. There exists $\gamma_z>0$ such that almost surely for all large $N$, there is no singular value 
of $\sigma Y_N + A'_N-z I_N$  in $[0, \gamma_z]$. Consequently, for any compact $K \subset \Gamma \cap \{z, 0 \notin \supp (\mu_z)\}$, there exists $\gamma_K >0$ such that a.s. for all large $N$,
$$
\inf_{ z \in K} s_N (\sigma Y_N + A'_N-z I_N) \geq \gamma_K.
$$
\end{proposition}

The second statement of Proposition \ref{nonunif} is a consequence of the first statement and \eqref{eq:CFresunif}. We begin with introducing some notation: 
$$\nu_{N,z} =\mu_{(A'_N-zI_N)( A'_N-zI_N)^*},$$
and $\mu_{N,z}$ denotes the distribution whose Stieltjes transform satisfies the equation 
\begin{equation}\label{eqgmuz} g_{\mu_{N,z}}(w)=\int \frac{1}{(1-\sigma^2g_{ \mu_{N,z}}(w))w- \frac{ t}{1- \sigma^2 g_{ \mu_{N,z}}(w)} }d\nu_{N,z}(t).\end{equation}

\begin{proposition}\label{decolle} Suppose that assumptions (X1), (A1), (A2) and (A4) hold.
Let $z$ be in $\Gamma$ such that $0 \notin \supp (\mu_z)$; then
there exists $\epsilon_z>0$, such that $[0, \epsilon_z] \subset \mathbb{R}\setminus  \mathrm{supp}(\mu_{z})$  and, for all large $N$,  $[0, \epsilon_z] \subset \mathbb{R}\setminus  \mathrm{supp}(\mu_{N,z})$.
\end{proposition}
\begin{proof}
According to the assumption (A4), there exists some $\eta>0$, such that for all large $N$, the spectrum of  $(A'_N-zI_N) (A'_N-zI_N)^*$ is included in $]\eta^2 , +\infty[$.
Also,  there exists $\delta >0$ such that $[0 , \delta] \subset \mathbb{R}\setminus \mathrm{supp}(\mu_z)$. We may choose $\delta$ small enough so that  $\omega_{\nu_z} (\delta ) \leq \eta^2 /2$. We find that  for $N$ large enough, 
 \begin{equation}\label{distance} \DIST\left(  [0 , \omega_{\nu_z} (\delta) ],\supp(\nu_{N,z}) \right) \geq \eta^2 /2.
\end{equation}
Also, according to   Proposition  \ref{support1t}, there exists $\tau>0$ such that \begin{equation}\label{strict} \mbox{for any $x$ in $ [0, \omega_{\nu_z}(\delta)]$,}~~  g_{\nu_z}(x) >-\sigma^{-2}+ \tau  \mbox{    and~~} \Phi'_{\nu_z} (x) > \tau.\end{equation}

From \eqref{distance}, assumption (A2) and Montel's theorem, $g_{\nu_{N,z}}$, $g'_{\nu_{N,z}}$ and  $ \Phi^{'}_{\nu_{N,z}}$ 
converge to $g_{\nu_z}$,  $g_{\nu_z} '$ and $\Phi_{\nu_z}^{'}$ respectively uniformly on $ [0,\omega_{\nu_z}(\delta)]$. Hence, using (\ref{strict}), we can claim that for all large $N$, 
$$ [0,\omega_{\nu_z}(\delta)]\subset \left\{ u \in   \mathbb{R}\setminus\supp({\nu_{N,z}}),  \Phi'_{\nu_{N,z}} (u) >0, g_{{\nu_{N,z}}}(u) >-\sigma^{-2}\right\}.$$

According to Proposition \ref{support1t}, we can deduce that
$$ \Phi_{\nu_{N,z}}\left( [0,\omega_{\nu_z}(\delta)]\right) =\left[ 0,\Phi_{\nu_{N,z}}(\omega_{\nu_z}(\delta))\right]
\subset   \mathbb{R}\setminus \supp(\mu_{N,z}).$$

Finally, since $\Phi_{\nu_{N,z}}(\omega_{\nu_z}(\delta))$ converges  towards $\Phi_{\nu_z}(\omega_{\nu_z}(\delta))=\delta$, we have 
for all large $N$, $$\Phi_{\nu_{N,z}}(\omega_{\nu_z}(\delta) ) \geq \delta / 2,$$  and then $[0,\delta/2] \subset      \mathbb{R}\setminus \supp(\mu_{N,z})$. \end{proof}

We are now ready to prove Proposition \ref{nonunif}.

\begin{proof}[Proof of Proposition \ref{nonunif}] Let $\gamma_z  >0$ be such that $\gamma_z < \epsilon_z$ and $\omega_{\nu_z}(\gamma_z) < \eta_z ^22$ where $ \epsilon_z$ is defined  in Proposition \ref{decolle} and $\eta_z$ is defined in (A4). By definition
$$\omega_{\nu_{N,z}}(\gamma_z 
)= \gamma_z  (1- \sigma^2 g_{\mu_{N,z }}(\gamma_z ))^2.
$$

Since $\mu_{{N,z}}$ converges weakly towards $\mu_z$, by Proposition \ref{decolle}, for all large $N$, $$\omega_{\nu_{N,z}}(\gamma_z)< \eta_z2.$$

Now, for $z \in \dC$ and $i \in \{1, \cdots, N\}$, let ${A ^{(i,z)}_N}$  be the $N \times (N-1)$ matrix obtained from $A'_N - z I_N$ be removing the $i$-th column. The interlacing inequalities (see e.g. \cite[Lemma A.1]{tao-vu-cirlaw-bis}) imply that 
$$
s_{N} ( A'_N -z I_N) \leq s_{N-1} \left( A^ {(i,z)}_N \right).
$$
It follows that  $A^{(i,z)}_N$ has no singular value in $[0, \eta_z]$. The condition (1.10) of Bai and Silverstein in \cite{MR2930382} is thus fulfilled on $[0, \frac{\gamma_z}{2}]$. We may thus apply \cite[Proposition 3.3]{MC}, we get that almost surely for all large $N$, there is no eigenvalue of $(\sigma Y_N + A'_N-z I_N) (\sigma Y_N +A'_N-z I_N) ^*$ in $[0 , \frac{\gamma_z}{2}] $. \end{proof}

\section{Stable outliers: proof of Theorem \ref{th:main}}
\label{sec:main}

The strategy of proof is to use Theorem \ref{inclusion} in conjunction with \eqref{eq:idcle}.  

\subsection{Convergence of bilinear forms of random matrix polynomial}

We start the proof of Theorem \ref{th:main} with a result of independent interest. We denote by $ \dC \langle X_1, \cdots, X_k  \rangle$ the set of non-commutative polynomials in the non-commutative variables $\{X_1,\cdots, X_k\}$ ($\dC$-linear combinations of words in the $X_i$'s with the empty word identified as $1 \in \dC$).

\begin{proposition}\label{prop:bilinearP}
Let $k \geq 1$ be an integer and $P \in \dC \langle X_1, \cdots, X_k  \rangle$ 
such that the exponent of $X_k$ in each monomial of $P$ is nonzero. We consider a sequence $(B_N^{(1)}, \cdots, B_{N}^{(k-1)}) \in M_{N} (\dC)^{k-1}$ of matrices with operator norm uniformly bounded in $N$ and $u_N$, $v_N$ in $ \dC^N$  with unit norm. Then, if $X_N$ satisfies assumptions (X1-X2), a.s. 
$$
u_N ^* P \left(  B_N ^ {(1)}, \cdots, B_N^{(k-1)} , Y_N \right) v_N \to 0. 
$$
\end{proposition}

 \begin{proof}
\noindent{\bf Step 1 : truncation / reduction.}
We set $B_N = (B_N^{(1)}, \cdots, B_N^{(k-1)})$. Without loss of generality, we can assume (up to changing $k$, $B_N^{(\ell)}$, $1 \leq \ell \leq k-1$,  and $u_N$, $v_N$) that, for any $1 \leq \ell \leq k-1$, $\|B^{(\ell)} _N \| \leq 1$ and that $P$ is of the form 
\begin{equation}\label{eq:formofP}
P \left(  B_N  , Y_N \right)  = Y_N \prod_{\ell = 1} ^{k-1} \left(  B_N ^{(\ell)} Y_N \right).
\end{equation}
We shall skip the index $N$ for ease of notation. Observe also that for any $S, T \in M_N ( \dC)$,  
\begin{align}
 \|  P \left(  B , S \right)   -  P \left( B  , T \right) \|    \leq k ( \|S\| \vee \|T\| )^{k-1} \|S - T \|.  \label{eq:diffnormP}
\end{align}

Moreover, for some fixed $K > 0$, consider the matrix $X^{(1)} _{ij} = X_{ij} \IND ( |X_{ij} | \leq K) - \dE  X_{ij} \IND ( |X_{ij} | \leq K)  $  and $X ^{(2)} _{ij} = X_{ij} \IND ( |X_{ij} | > K) - \dE X_{ij} \IND ( |X_{ij} | > K) $.  For $i = 1,2$, we set $Y^{(i)} = X^{(i)} / \sqrt N$, we have $Y = Y^{(1)} + Y^{(2)}$. From Bai-Yin theorem \cite{MR1235416} (Theorem 5.8 in \cite{MR2567175}) , there exists  $\veps (K) \to 0$ as $K \to \infty$, such that, a.s.  $\limsup_N \|Y^ {(2)}\| \leq \veps (K)$, $\limsup_N \|Y^ {(1)}\| \leq  2 + \veps(K)$ and $\lim_N \|Y\| =  2$.  In particular, from \eqref{eq:diffnormP}, we deduce that a.s. 
\begin{align*}
 \limsup_N  \|  P \left(  B, Y \right)   -  P \left(  B , Y^{(1)} \right) \| \leq  k  ( 2 + \veps(K) )^{k-1}  \veps (K), 
\end{align*}

Hence, in summary, it is sufficient to prove the statement of Proposition \ref{prop:bilinearP} with, for some $K>0$, $X_{ij}$ with bounded support in the ball of radius $K$, $P$ of the form \eqref{eq:formofP} and $\|B^{(i)}\| \leq 1$.

We now consider $G = (G_{ij}) \in M_N (\dC)$ a random matrix with i.i.d. $N(0,1/n)$ Gaussian entries independent of $X$. The above argument shows that for any $\theta>0$, a.s. 
\begin{align*}
& \limsup_N  \|  P \left(  B , Y \right)   -  P \left(  B, Y + \theta G  \right) \| \leq  k  \left(  2^2  ( 1 + \theta^2 ) \right)^{\frac{k-1}{2}}   2 \theta , 
\end{align*}

In  particular, without loss of generality we may assume that there exist $\theta,K >0$, such  that the law of $X_{ij}$ are a convolution of the Gaussian distribution $N(0,\theta^2)$ with a law of bounded support in the ball of radius $K$. It implies notably they satisfy a log-Sobolev inequality with a common constant $\delta >0$ (see \cite{MR3067796,wangwang}). This will be our final assumption of the laws of $X_{ij}$. 

\vline

\noindent{\bf Step 2 : concentration.} Our aim is now to check that a.s., as $N \to \infty$, 
\begin{equation}\label{eq:concbilinearP}
u^* P ( B, Y ) v - \dE u^* P ( B, Y ) v  \to 0. 
\end{equation}

This will follow from a general concentration argument. We identify $M_N ( \dC)$ with $\dR^{2 N^2}$: the Frobenius norm of a matrix, $\|x \|_2 = \sqrt{ \sum_{ij} \Re ( x_{ij} ) ^2  + \Im (x_{ij})^2 }$ is then its Euclidean norm. We consider the $\dR^{2 N^2} \to \dC$ function, $F(x) = u^* P(B,x) v$ and define $K \subset M_N ( \dC)$ as the convex subset of matrices with operator norm bounded by $4$. If $x,y \in K$ then by \eqref{eq:diffnormP}, 
$$
|F(x) - F(y)| \leq k 4^{k-1} \|x - y \| \leq k 4^{k-1} \|x - y \|_2. 
$$
It follows that if $\Pi$ is the Euclidean projection of a matrix on $K$, the function $G(x) = F (\Pi (x))$ is Lipschitz with constant $ k 4^{k-1}$.  From Herbst's argument (e.g. \cite[Lemma 2.3.3]{MR2760897}), we deduce that 
$$
\dP (  | G( Y ) - \dE G(Y) | \geq t ) \leq 4 \exp ( -  c N t ^2),
$$
 where $c$ is related to the Lipschitz constant of $G$ and the constant of the Log-Sobolev inequality satisfied by the laws of the $X_{ij}$'s. In particular, a.s.  as $N \to \infty$, 
$$
G(Y) -  \dE G(Y)  \to 0. 
$$
From Bai-Yin theorem \cite{MR1235416}, see also (\cite[Theorem 5.8]{MR2567175}), a.s. $\|Y \| \leq 2 + o(1) < 4$. Hence, a.s. $G(Y) = F(Y)$ for all $N$ large enough, and a.s. as $N \to \infty$, 
$$
F(Y) - \dE G(Y) \to 0. 
$$

Also, the same reasoning applied to the $1$-Lipschitz function $x \mapsto \|x\|$ gives that 
$$
\dP (   \|Y\| - \dE \|  Y\|   \geq t ) \leq  \exp ( -  c N t ^2).  
$$
Using again that a.s. $\|Y \| \leq  2 + o(1) < 4$ for all $N$ large enough, we deduce that as $N \to \infty$, 
$$
| \dE    G (Y) - \dE F(Y) | \leq 4^{k} \dP (  \| Y \| \geq 4) + \dE \|Y\|^k  \IND (  \| Y \| \geq 4) \to 0. 
$$
We thus have proved that \eqref{eq:concbilinearP} holds. 

\vline

\noindent{\bf Step 3 : graph counting.} The proof of Proposition \ref{prop:bilinearP} will be complete if we manage to check that 
\begin{equation}\label{eq:bilinearPaim}
\dE u^* P(B,Y) v  = O ( N^{-1/2} ). 
\end{equation}
For integers $0 \leq \ell  < k $, we set $\bl k \br = \{1, \cdots, k\}$ and $\bl \ell, k  \br = \{\ell, \cdots, k \}$. We have 
$$
\dE u^* P(B,Y) v   = N^{- \frac k 2} \sum \bar u_{i_1} v_{i_{2k }} \dE X_{i_1 i_2} \prod_{\ell = 1}^{k-1} B^{(\ell)}_{i_{2\ell}  i_{2 \ell+1}} X_{i_{2\ell+1}i_{2 \ell +2}},
$$
where the sum is over all $1 \leq i_s  \leq N$, $1 \leq s  \leq  2k$. The variables $X_{ij}$ are centered, independent and have uniformly bounded moments first $k$
 moments. It follows that the above expectation will be non-zero only if the pairs of index $(i_{2 \ell-1}, i_{2\ell})$, $1 \leq \ell \leq k$ appears at least twice.  Hence, there are $1 \leq q \leq \lfloor k/2 \rfloor$ distinct such pairs and $p \leq 2q$ distinct indices in $(i_1, \cdots , i_{2k})$.  We may thus bound our expectation as a finite sum (depending on $k$) of terms of the type $c N^{- \frac k 2}S$ with $c >0$ and 
\begin{equation}\label{eq:graphcount}
S =   \sum | u_{i_1} | | v_{i_{\pi(2k)}} | \prod_{\ell = 1}^{k-1} | B^{(\ell)}_ {i_{\pi(2\ell)}  i_{\pi(2 \ell+1)}}|,
\end{equation}
where the sum is over all $1 \leq i_1 ,  \cdots , i _{p} \leq N$ and $\pi :\bl 2k \br \to \bl  p \br$ is a fixed surjective map such that $\pi(1) = 1$ and for any $(u,v) \in \bl p \br ^2$, $n(u,v) = \sum_{\ell =1}^{k} \IND \{  ( \pi (2\ell-1) , \pi (2 \ell )) = (u,v) \} \ne 1$. We may further assume that if $\pi(2k) \ne \pi(1)$, $\pi(2k) = p$. 

Since $q \leq k/2$, the bound \eqref{eq:bilinearPaim} would follow if we manage to prove the bound 
\begin{equation}\label{eq:Saim}
S \leq N^{   \frac{2 q - 1}{2}}. 
\end{equation}

%The case $p=1$ is simple :  from the assumption $\|B^{(\ell)} \| \leq 1$, we have  $| B^{(\ell)}_ {i j}| \leq 1$. From Cauchy-Schwartz inequality, we readily get $S \leq 1$. 

To this end,  we introduce a natural graph associated to the map $\pi$. For $(u,v) \in \bl p \br ^2$, we set $m(u,v) = \sum_{\ell =1}^{k-1} \IND \{  ( \pi (2\ell) , \pi (2 \ell +1 )) = (u,v) \} $.  We consider the graph $G = (V,E)$ (with loops and multiple edges) on the vertex set $V = \bl p \br$ and $$M(\{u,v\}) = m(u,v) + m(v,u) \IND( v \ne u)$$ is the multiplicity of the edge $\{u,v\}$ ($E$ is a multiset and $\{u,v\}$ appears $M(u,v)$ times in $E$).

We will prove that \eqref{eq:Saim} holds when $\pi(2k) \ne \pi(1)$. The case $\pi(2k) = \pi(1)$ is analog and simpler. Then, the key observation is that the condition $n(u,v) \ne 1$ implies that any $u \in \bl 2, p -1\br$ has degree at least $2$: $\deg(u) = \sum_{v \in V} M(u,v) \geq 2$. We also have $\deg(1) \geq 1$ and $\deg(p) \geq 1$.

Let $\Lambda \subset V$ be the vertices with a loop, i.e. the set of $u \in V$ such that $m(u,u) \geq 1$. We note that
\begin{eqnarray}\label{eq:Stemp}
|\Lambda| \leq  2 q -  p.
\end{eqnarray}
Indeed, if $v = \pi( 2 \ell) = \pi( 2\ell +1)$ then the oriented egdes $(\pi(2 \ell-1),\pi(2\ell))$ and $(\pi(2\ell+1), \pi(2\ell+2))$ share at least one adjacent vertex. They are distinct (due to orientation) unless $\pi(2\ell -1) = \pi ( 2\ell +2) = v$. It follows that \eqref{eq:Stemp} can be proved easily by recursion on $|\Lambda|$. As a consequence \eqref{eq:Saim} is implied by the stronger result : 
\begin{equation}\label{eq:Saim2}
S \leq N^{   \frac{p  + |\Lambda | - 1}{2}}. 
\end{equation}

We now start the proof of this last equation \eqref{eq:Saim2}. We can certainly decompose \eqref{eq:graphcount} as a product over the connected components of $G$. Let $H'= (V',E')$ be a connected component of $G$. Assume first that the vertex set $V'$ of $H'$ contains neither $1$ nor $p$. Then, if $L' =  \{ \ell \in \bl k-1 \br :  \{\pi(2\ell) , \pi(2 \ell+1)\} \in E'\}$, we claim that 
\begin{equation}\label{eq:boundS'}
S' =  \sum_{i_v : v \in V'}   \prod_{\ell \in L' } | B^{(\ell)}_ {i_{\pi(2\ell)}  i_{\pi(2 \ell+1)}}| \leq N^{\frac{|V'| + \veps'}{2}},
\end{equation}
where $\veps' \in \{0,1\}$ is equal to $1$ if $H'$ contains a vertex in $\Lambda$ and is $0$ otherwise.
First, from the key observation, $H'$ is not a tree. In particular, there exists a spanning subgraph $H''\subset H'$ where $H'' = (V', E'')$ is a cycle of length $c$ (if $c =1$, $\veps' = 1$, the cycle is a loop and if $c=2$, it is a multiple edge) with attached pending trees. Recall that  $| B^{(\ell)}_ {i j}| \leq 1$. It follows that, if $L'' =  \{ \ell \in \bl k-1 \br :  \{\pi(2\ell) , \pi(2 \ell+1)\} \in E''  \}$, we find
$$
S' \leq S'' =  \sum_{i_v : v \in V'}   \prod_{\ell \in L'' } | B^{(\ell)}_ {i_{\pi(2\ell)}  i_{\pi(2 \ell+1)}}|.  
$$
$H''$ has $c$ vertices on its cycle and $|V'| - c$ vertices on the pending subtrees. Consider $v \in V'$ a leaf of one these pending subtrees, i.e. $\deg_{H''} (v) = 1$, then it appears only once in the above product. Since $\| B ^{(\ell)} e_ j \| \leq \|B^{(\ell)} \| \leq 1$, we have for any $j$, 
$$
\sum_{i_v} | B^{(\ell)}_ {i_v j} | \leq \sqrt N \sqrt { \sum_{i_v}| B^{(\ell)}_ {i_v j} |^2 } \leq \sqrt N,  
$$
and similarly for $ \sum_{i_v}| B^{(\ell)}_ {j i_v} | $. We may repeat iteratively this bound for all vertices in the pending subtrees, we deduce that, if $C = ( V_C, E_C)$ is the cycle of $H''$ and $L_C = \{ \ell :  \{\pi(2\ell) , \pi(2 \ell+1)\} \in E_C\}$, 
$$
S'' \leq N^{\frac{ |V'| - c}{2}} \sum_{i_v : v \in V_C}   \prod_{\ell \in L_C } | B^{(\ell)}_ {i_{\pi(2\ell)}  i_{\pi(2 \ell+1)}}|.  
$$
Now, if $c =1$ then it remains a unique loop vertex $i_v$ and a product of elements of form  $B^{(\ell)}_ {i_{v}  i_{v}}$, $\ell \in L_C$. From Cauchy-Schwartz inequality, we find in this case, 
$$
S'' \leq N^{\frac{ |V'| - 1}{2}} \sum_{i_v }   | B^{(\ell)}_ {i_{v}  i_{v}}| \leq N^{\frac{ |V'| -1 }{2}} \sqrt N \sqrt{ \sum_{i , j}   | B^{(\ell)}_ {i j}|^2 } \leq N^{\frac{ |V'| + 1 }{2}} =  N^{\frac{ |V'| + \veps' }{2}}.  
$$

Similarly, if $c >1$, take any $v \in V_C$, then it appears twice in the above product. From Cauchy-Schwartz inequality, we get for any $\ell, \ell'$ and $j,j'$, 
$$
\sum_{i_v} | B^{(\ell)}_ {i_v j} B^{(\ell')}_{i_v j'}| \leq 1,  
$$
and similarly for $\sum_{i_v} | B^{(\ell)}_ { j i_v } B^{(\ell')}_{i_v j'}|$ and $\sum_{i_v} | B^{(\ell)}_ {j i_v } B^{(\ell')}_{j' i_v }|$. Hence, we sum over $i_v$ and it remains a line-tree with $c-1$ vertices. Arguing as above, we may sum over each vertex: each will add extra factor $\sqrt N$ but the last one, which will give factor $N$.  So finally, 
$$
S'' \leq N^{\frac{ |V'| - c}{2}} N^{ \frac{ c-2}{2}} N = N^{\frac{ |V'| }{2}} \leq N^{\frac{ |V'| + \veps' }{2}} .
$$
It proves \eqref{eq:boundS'}.

Let us now turn to a connected component  $H' = (V',E')$ of $G$ such that  $ 1 \in V'$ and $p \notin V'$. We claim that  
\begin{equation}\label{eq:boundS'2}
S' =  \sum_{i_v : v \in V'}  |u_{i_1} |  \prod_{\ell \in L' } | B^{(\ell)}_ {i_{\pi(2\ell)}  i_{\pi(2 \ell+1)}}| \leq N^{\frac{|V'| + \veps'- 1}{2}}. 
\end{equation}
%Assume first that $|V'| = 1$. Then $V' =  \{ 1\}$ and Cauchy-Schwartz inequality gives
%$$
%S' =  \sum_{i_1}  |u_{i_1} |  \prod_{\ell \in L' } | B^{(\ell)}_ {i_{1}  i_1}| \leq  \left(  \sum_{i_1}  |u_{i_1} |^2 \sum_{i_1} | B^{(\ell_1)}_ {i_{1}  i_1}|^2 \right)^{\frac 1 2} \leq 1, 
%$$
%where $\ell_1$ is any element in $L'$ (which is necessarily non-empty since $\deg(1) \geq 1$). 

%Assume now that $|V'| \geq 2$. 
The argument is as above. There is a spanning subgraph $H'' \subset H'$ and $H''$ is a cycle with attached pending subtrees (indeed a connected graph with at least $2$ vertices and at most one vertex of degree $1$ cannot be a tree). We repeat the above pruning procedure of the pending trees and of the cycle. The only difference comes when this is the turn of $i_1$. Using Cauchy-Schwartz inequality, we improve by a factor $\sqrt N$ our previous bounds 
\begin{align*}
&\sum_{i_1} |u_{i_1}|  | B^{(\ell)}_ {i_{1}  j }| \leq \sqrt{ \sum_{i_1} |u_{i_1}|^2 } \sqrt{ \sum_{i_1}   | B^{(\ell)}_ {i_{1}  j }|^2} \leq 1 \quad   \hbox{ and } \\
&\quad  \sum_{i_1} |u_{i_1}|  | B^{(\ell)}_ {i_{1}  i_{1 }}| \leq \sqrt{\sum_{i_1} |u_{i_1}|^2} \sqrt{ \sum_{i_1}   | B^{(\ell)}_ {i_{1}  i_{1 }}| ^2} \leq \sqrt N.
\end{align*}
It gives \eqref{eq:boundS'2}. 

The same bound obviously holds if the connected component  $H' = (V',E')$ of $G$ is such that  $ 1 \notin V'$ and $p \in V'$. It remains to deal with the case $1 \in V$ and $p \in V$. In this case, we also have the bound
\begin{equation}\label{eq:boundS'3}
S' =  \sum_{i_v : v \in V'}  |u_{i_1} | |v_{i_{\pi(2k)}}| \prod_{\ell \in L' } | B^{(\ell)}_ {i_{\pi(2\ell)}  i_{\pi(2 \ell+1)}}| \leq N^{\frac{|V'| + \veps' - 1}{2}}. 
\end{equation}
The argument goes as follows: $H'$ contains $T = (V',E_T)$ a spanning subtree. If $L_T = \{ \ell :  \{\pi(2\ell) , \pi(2 \ell+1)\} \in E_T\}$, we get
$$
S' \leq  \sum_{i_v : v \in V'} |u_{i_1} | |v_{i_{\pi(2k)}}|  \prod_{\ell \in L_T } | B^{(\ell)}_ {i_{\pi(2\ell)}  i_{\pi(2 \ell+1)}}|.  
$$
We perform the above pruning of the tree starting from the leaves. Again, using Cauchy-Schwartz inequality, each vertex will contribute by a factor $N^{\gamma_v}$ where $\gamma_v = 1/2 + \delta_v/2 - \IND( v = 1)/2 - \IND( v = p)/2$ and $\delta_v =1$ is $v$ is the last vertex removed and $0$ otherwise. We obtain \eqref{eq:boundS'3}.

Summarizing, \eqref{eq:graphcount} can be written as a product over each connected component of $G$ of expressions of the form \eqref{eq:boundS'}, \eqref{eq:boundS'2} (possibly with $v_{i_{\pi(2k)}}$ replacing $u_{i_1}$) or \eqref{eq:boundS'3}. Observe that the sum over all connected components $H'$ of $\veps' = \veps'(H') $ is at most $|\Lambda|$. Two cases are possible, either $1$ and $p$ are in the same connected component and we obtain from \eqref{eq:boundS'}-\eqref{eq:boundS'3}, 
$$
S \leq N^{\frac{p + |\Lambda|-1}{2}},
$$
or $1$ and $p$ are in distinct connected components and, by \eqref{eq:boundS'2}-\eqref{eq:boundS'3}, 
$$
S \leq N^{\frac{p+ |\Lambda|-2}{2}}.
$$
In either case, \eqref{eq:Saim2} holds and it concludes the proof of \eqref{eq:bilinearPaim}. \end{proof}

\subsection{Convergence of resolvent outside the limit support}

Let $\Gamma $ be as in Theorem \ref{th:main}. From the singular value decomposition of $A''_N$, we write $A''_N = P_N Q_N$ with $P_N, Q_N^\top \in M_{N,r} ( \dC)$ with uniformly bounded norms. We introduce the resolvent matrices 
$$
R_N ( z) = ( z I_N - \sigma Y_N - A'_N )^{-1} \quad \hbox{ and } \quad R'_N ( z) = ( z I_N - A'_N)^{-1}.
$$
The objective of this section is to prove that $R_N(z)$ is close to $R'_N(z)$ outside $\supp(\beta)$. More precisely,  we shall prove the following proposition.

\begin{proposition}\label{convunif}
Suppose that assumptions (X1-X2) and assumptions (A1'-A4) hold with $\Gamma \subset \dC \backslash \supp( \beta) $ compact. Almost surely
$$\sup_{z \in \Gamma}\left\|Q_N R_N(z) P_N -Q_N R'_{N}(z) P_N \right\|$$
converges towards zero when $N$ goes to infinity.
\end{proposition}

The main step in the proof of Proposition \ref{convunif} will be the following proposition. 

\begin{proposition}\label{borne}Suppose that assumptions (X1-X2) and assumptions (A1'-A4) hold with $\Gamma \subset \dC \backslash \supp( \beta) $ compact. 
There exists $0<\epsilon_0<1$ and $C>0$ such that almost surely for all large $N$, for any $k \geq 1$,
$$\sup_{ z \in \Gamma} \left\|  \left(R'_{N}(z) \sigma Y_N \right)^k\right\| \leq C(1-\epsilon_0)^k.$$
\end{proposition}

We first establish the following lemmas.
\begin{lemme}\label{rayonspectralprelim}
Suppose that assumptions (X1-X2) and assumptions (A1'-A4) hold with $\Gamma \subset \dC \backslash \supp( \beta) $ compact. There exists $ \rho  > 1$ and $\eta >0$ such that almost surely for all large $N$, for all $w $ in $ \mathbb{C}$ such that $\vert w \vert \leq \rho \sigma$ and for all $z$ in $\Gamma$,
there is no eigenvalue of $(w Y_N + A'_N-z I_N) (w Y_N  +A'_N-z I_N) ^*$ in $[0,\eta]$.

\end{lemme}
\begin{proof}
According to Proposition \ref{prop:supportbeta}, for any $z$ in $\mathbb{C} \setminus \text{supp}(\beta)$,
$z \notin S$ and $\varphi(z) <\sigma^{-2}.$ 
Since according to  Lemma \ref{continuitephi}, the    function $\varphi$ is continuous  on $\mathbb{C} \setminus \text{supp}(\beta)$, it attains its lowest upper bound on
the compact set $\Gamma$, so that there exists $0< \gamma<1$ such that 
 for any $z$ in $\Gamma$, $\varphi(z)<(1-\gamma)\sigma^{-2}$.

Let $w \in \mathbb{C}\backslash\{0\} $. 
Since $w Y_N +A'_N-zI_N=\vert w \vert \exp (i \arg(w)) Y_N  +{A'_N}-{z}I_N$, by Theorem \ref{convbeta},  the spectral measure of $(w Y_N+A'_N-zI_N)(w Y_N +A'_N-zI_N)^*$ converges weakly towards a probability measure  $\mu_{w,z}$ and, by Proposition \ref{zeroinclus}, we have 
\begin{align*}
 \{ z \in \C :  0 \in \mathrm{supp}(\mu_{w, z}) \}  = \{ z  \in \C :  z \in  S \hbox{ or } \varphi(z) \geq  \vert w \vert^{-2}\}.
 \end{align*}

For $w=0$, we define $\mu_{0, z}=\nu_z$. Therefore, using also (A3) and  Corollary \ref{cor:support1} for $w=0$, setting $\rho = 1/ \sqrt{ 1 - \gamma}$, it follows that for any $z \in \Gamma $ and any $w$ such that $\vert w \vert \leq \rho \sigma,$  we have  $0\notin \mathrm{supp}( \mu_{w,z}).$
Define the compact set $$\tilde{\Gamma}=\{(w,z)\in \mathbb{C}^2, \vert w \vert \leq \rho \sigma,~z \in \Gamma\} .$$

According to Proposition \ref{nonunif},  for any $(w,z)$ in $\tilde \Gamma $, there exists $\gamma_{(w,z)}>0$ such that 
 almost surely for all large $N$,  
there is   no eigenvalue of 
$(w Y_N  + A'_N-z I_N) (w Y_N  +A'_N-z I_N) ^*$ in
$[0 , \gamma_{(w,z)}]$. Also from Bai-Yin theorem \cite{MR1235416}, almost surely, $\| Y_N \| \leq 2 + o(1)$. Then, using  \eqref{eq:CFres} and the  same compactness argument leading to  \eqref{eq:CFresunif}, it proves that there exists $\eta >0$ such that almost surely for all large $N$, for  any $(w,z) \in \tilde \Gamma$, there is no eigenvalue of $(w Y_N  + A'_N-z I_N) (w Y_N +A'_N-z I_N) ^*$ in $[0,\eta]$.
\end{proof}

\begin{lemme}\label{rayonspectral}
Suppose that assumptions (X1-X2) and assumptions (A1'-A4) hold with $\Gamma \subset \dC \backslash \supp( \beta) $ compact. There exists $0<\epsilon_0<1$ such that almost surely for all large $N$, we have for any $z$ in $\Gamma$,
$$\rho \left(R'_{N}(z)\sigma   Y_N \right) \leq 1-\epsilon_0,$$
where $\rho(M) $ denotes the spectral radius of a matrix $M$.
\end{lemme}

\begin{proof}

Now, assume that $\lambda\neq 0$ is an eigenvalue of $R'_{N}(z) Y_N$.  Then there exists $v \in \mathbb{C}^N$, $v\neq 0$ such that 
$(zI_N -A'_N)^{-1} Y_Nv =\lambda v$ and thus $(\lambda^{-1}  Y_N  +A'_N -z I_N) v=0.$
It follows that $z$ is an eigenvalue of $\lambda^{-1} Y_N  +A'_N$. By Lemma \ref{rayonspectralprelim}, we can deduce that almost surely for all large $N$, the non nul eigenvalues of $R'_{N}(z)Y_N$ must satisfy $1 / | \lambda| >  \rho  \sigma$. The result follows.
\end{proof}

We are now ready to prove Proposition \ref{borne}.
\begin{proof}[Proof of Proposition \ref{borne}] For $z \in \Gamma$, we set 
$T_N=R'_{N}(z) \sigma Y_N .$ Let $\epsilon_0$ be as defined by Lemma \ref{rayonspectral}. Thanks to the Cauchy formula, for all $x \in \mathbb{C}$ such that $\vert x \vert < 1-\epsilon_0/2$, for any $k \geq 0$, $x^k =\frac{1}{2i\pi} \int_{\vert w\vert = 1-\epsilon_0/2} \frac{w^k}{w-x} dw.$
Therefore, according to Lemma \ref{rayonspectral} and using the holomorphic functional calculus, we have almost surely for all large $N$,  for any $z$ in $\Gamma$,
$$\forall k \geq 0~~, T_N^k =\frac{1}{2i\pi} \int_{\vert w\vert = 1-\epsilon_0/2} {w^k}{(w-T_N)^{-1}} dw,$$
and therefore $$\forall k \geq 0~~, \Vert  T_N^k \Vert \leq  \sup_{\vert w\vert = 1-\epsilon_0/2 } \Vert {(w-T_N)^{-1}}\Vert {( 1-\epsilon_0/2 ) }^{k+1}.$$

Now, since $$(wI_N-T_N)=  - w R'_N   \left(\frac{\sigma}{w} Y_N  +A'_N-zI_N\right),$$
Lemma \ref{rayonspectralprelim} readily implies that for $\epsilon_0$ small enough, there exists $C>0$ such that  we have almost surely for all large $N$,  for any $z$ in $\Gamma$,
$$ \sup_{\vert w\vert = 1-\epsilon_0/2 } \Vert {(wI_N-T_N)^{-1}}\Vert \leq C.$$
  Proposition \ref{borne} follows.\end{proof}

\begin{lemme}\label{serie}
Suppose that assumptions (X1-X2) and assumptions (A1'-A4) hold with $\Gamma \subset \dC \backslash \supp( \beta) $ compact.  For any $z$ in  $\mathbb{C}\setminus \mathrm{supp}(\beta )$, almost surely the series
$\sum_{k\geq 1} Q_N \left(R'_{N}(z) \sigma Y_N \right)^k R'_{N}(z)P_N$ converges in norm towards zero as $N$ goes to infinity.

\end{lemme}
\begin{proof}
The singular value decomposition of $A''_N$ gives that for any $i,j \in \{1,\ldots,r\}$,
$$(Q_N \left(R'_{N}(z) \sigma Y_N \right)^k  R'_{N}(z) P_N)_{ij}= s_i v_i^* \left(R'_{N}(z) \sigma Y_N \right)^k R'_{N}(z)u_j$$
where $u_j$ and $v_j$ are unit vectors and $s_i$ is a singular value of $A''_N$.
 According to (A1'), the $s_i$'s are uniformly bounded. By  (A4), for any $z$ in $\Gamma$, there exists $\eta_z$ such that 
 for all large $N$, \begin{equation}\label{majresA} \Vert R'_{N}(z)\Vert \leq 1/\eta_z.\end{equation}
Therefore, Proposition \ref{prop:bilinearP} yields that $v_i^* (R'_{N}(z) \sigma  Y_N )^k R'_{N}(z)u_j$ converges almost surely towards zero.
The result follows by applying the dominated convergence theorem thanks to Proposition \ref{borne}.
\end{proof}

All ingredients are gathered to prove Proposition \ref{convunif}.

\begin{proof}[Proof of Proposition \ref{convunif}] 
We start by proving that for any $z$ in $\Gamma$, almost surely, as $N \to \infty$, 
\begin{equation}\label{eq:convRz}
\Vert Q_N R_N(z) P_N -Q_N R'_{N}(z) P_N \Vert \to 0. 
\end{equation}
 Let $C'>0$ such that $\left\| P_N \right\| \left\| Q_N \right\| \leq C'$.  According to Proposition \ref{nonunif}, for any $z \in \Gamma$, there exists $\gamma_z>0$ such that   almost surely for all large $N$
\begin{equation}\label{majresA2}
\Vert R_N(z) \Vert \leq 1/\gamma_z.\end{equation} Then 
using also Proposition \ref{borne} and \eqref{majresA}, for any $k\geq 1$, we have
$$\left\| Q_N  \left(R'_{N}(z) \sigma Y_N \right)^kR'_{N}(z) P_N\right\| \leq \frac{CC'}{\eta_z}(1-\epsilon_0)^k,$$ 
$$\left\| Q_N  \left(R'_{N}(z) \sigma Y_N \right)^kR_N(z) P_N\right\| \leq \frac{CC'}{\gamma_z}(1-\epsilon_0)^k.$$ 
Let $\eta>0$. Choose $K \geq 1$ such that $\frac{CC'}{\gamma_z}(1-\epsilon_0)^K< \eta/2$
and $\sum_{k\geq K}  \frac{CC'}{\eta_z}(1-\epsilon_0)^k <\eta/2$.\\
Now, using repeatedly the resolvent identity,
$$R_N(z)= R'_{N}(z) +  R'_N(z) \sigma Y_N R_N(z),$$
we find
\begin{align*} & Q_NR_N(z) P_N  -  Q_N R'_N(z) P_N  \\
& \quad \quad =  \sum_{k=1}^{K-1} Q_N  \left(R'_N(z) \sigma Y_N \right)^kR'_N(z) P_N   + Q_N  \left(R'_N(z) \sigma Y_N \right)^{ K }R_N(z) P_N \end{align*}
Thus for any $\eta>0$, 
$$ \left\|Q_NR_N(z) P_N - Q_N R'_N(z) P_N - \sum_{k\geq 1} Q_N  \left(R'_N(z) \sigma  Y_N \right)^kR'_N(z) P_N \right\| < \eta$$
and letting $\eta$ going to zero,
we have \begin{equation}\label{eq:devtaylor}
Q_N R_N(z) P_N  -  Q_N R'_N(z) P_N  = \sum_{k\geq 1} Q_N  \left(R'_N(z) \sigma Y_N\right)^kR'_N(z) P_N.
\end{equation}
Applying Lemma \ref{serie}, we obtain \eqref{eq:convRz}.

To conclude the proof of the proposition, it sufficient to check that for any $\delta >0$, a.s., for all large $N$, 
\begin{equation}\label{eq:convRzunif}
\sup_{ z \in \Gamma} \left\| Q_N R_N(z) P_N -Q_N R'_{N}(z) P_N \right\| \leq 3 \delta. 
\end{equation}
We set $\zeta_z = \eta_z \wedge \gamma_z $ and $r_z =  ( \zeta_z /2 ) \wedge  ( \delta ( \zeta_z / 2 C' )^{2} )$. Using the resolvent identity, \eqref{eq:CFres} and  \eqref{majresA}-\eqref{majresA2}, if $|z - w| \leq r_z$,
\begin{eqnarray*}
\| Q_N R_N (z) P_N  - Q_N R_N (w) P_N \| &\leq &\left(\frac{ 2 C'}{\zeta_z} \right)^2   |z - w| \leq \delta\\
\| Q_N R'_N (z) P_N  - Q_N R'_N (w) P_N \|& \leq &\left(\frac{ 2 C'}{\zeta_z} \right)^2   |z - w| \leq \delta
\end{eqnarray*}
Since $\Gamma \subset \cup_{z \in \Gamma} B( z, r_z)$ and $\Gamma$ compact, there is a finite covering and \eqref{eq:convRzunif} follows from \eqref{eq:convRz}. 
\end{proof}

\subsection{Proof of Theorem \ref{th:main}}

According to Theorem \ref{inclusion}, almost surely  for all large $N$,  for any $z \in \Gamma$, 
the matrix $z I_N- \sigma Y_N  -A_N'$ is invertible. By \eqref{eq:idcle}, a.s. for all large $N$, the eigenvalues of $M_N$ in $\Gamma$
are precisely the zeros of the random analytic function 
$$\det (I_r-Q_N R_N(z )P_N )$$in that  set. On the other end, by assumption (A4), \eqref{eq:idcle} implies also that for all $z \in \Gamma$, 
$$
\det (I_r-Q_N R'_N(z)P_N  ) = \frac{\det(zI_N-A_N)}{\det (zI_N- A'_N)}.  
$$
From \eqref{eq:contdet}, we deduce from Proposition \ref{nonunif}, Proposition \ref{convunif} and assumption (A4) and \ref{eq:CFresunif} that $\det (I_r-Q_N R_N(z )P_N )-\det (I_r-Q_N R'_N(z)P_N  )$  converges to zero uniformly on $\Gamma$.  Using (\ref{eq:ratioAA'}), the result follows
by Rouch\'e's Theorem.

\subsection{Proof of Corollary \ref{cor:main}}

By assumption (A1) and Bai-Yin Theorem, a.s. for all $N$ large enough, all eigenvalues of $M_N$ are included in $K = B ( 0 , M + 4)$. 
From (A1), up to extract a converging subsequence, we can assume that   for $k \in J$,  $\lambda_k(A_N) $ converges to $\lambda_k$.
Let $0 < \delta < \veps$ and for $k \in J$, let $\Gamma_k = B ( \lambda_k, \delta)$, and $\Gamma_0 $ the closure of $K \backslash \cup_{ k \in J} \Gamma_k$. 
%For $\delta$ small enough, there is at most one distinct eigenvalue of $A_N$ in each $\Gamma_k$. 
Then, from \eqref{eq:detA'A"}, we may apply Theorem \ref{th:main} to each of the $\Gamma_k$. Since $\delta$ can be arbitrarily small, the conclusion follows.

\section{Fluctuations of stable outlier eigenvalues}
\label{sec:TCL}

\subsection{Normalized trace of some random matrix polynomials}

\label{subsec:TCLcond}

For further needs, we start this section with a proposition on trace of powers of random matrices.

\begin{proposition}
\label{prop:ABX}
Let $k,k' \geq 0$ be integers. We consider a sequence of matrices $ \left(B^{(\ell)}_N\right)_{0\leq \ell \leq k} , \left(C^{(\ell)}_N\right)_{1\leq \ell \leq k'}$ in $M_{N} (\dC)$, with operator norm uniformly bounded in $N$ such that, for some $w, v_{\ell} \in \dC$, $1 \leq \ell \leq k \wedge k'$,
$$
\frac 1 N \Tr B^{(\ell)}_N C^{(\ell)}_N \to v_\ell \; \hbox{ and }  \; \frac 1 N \Tr B^{(0)}_N \to w.
$$
Then, if $X_N$ satisfies assumptions (X1-X2), a.s., as $N \to \infty$,
\begin{eqnarray*}
\frac 1 N \Tr \left\{  B^{(0)}_N  \prod_{\ell =1}^k (Y_N B^{(\ell)}_N) \prod_{\ell =1}^{k'} (C^{(k'-\ell+1)}_N Y_N ^ * )  \right\} & \to & w \prod_{\ell=1} ^k v_\ell  \IND_{\{ k = k'\}} \\
\frac 1 N \Tr \left\{   B^{(0)}_N \prod_{\ell =1}^k (Y_N B^{(\ell)}_N) \prod_{\ell =1}^{k'} (C^{(k'-\ell+1)}_N Y_N ^ \top ) \right\} & \to & \dE \left(X_{11}^2\right)^k w \prod_{\ell=1} ^k v_\ell  \IND_{\{k = k'\}}.
\end{eqnarray*}
\end{proposition}

\begin{proof}
The proof of the two statements is identical. We will only prove the first statement.  We start as in the proof of Proposition \ref{prop:bilinearP}. For ease of notation, we drop the subscript $N$, in $Y_N$, $B^{(\ell)}_N,C^{(\ell)}_N$. We can assume without loss of generality that $\|B^{(\ell)}\|, \|C^{(\ell)}\| \leq 1$. For integers $0 \leq \ell  < k $, we set $\bl k \br = \{1, \cdots, k\}$ and $\bl \ell, k  \br = \{\ell, \cdots, k \}$. First, we may repeat steps 1 and 2 of the proof of Proposition \ref{prop:bilinearP}. We find that it is sufficient to prove that 
\begin{equation}\label{eq:ABXaim}
\frac 1 N \dE \Tr  P ( Y , Y^* ) = \frac 1 N \dE \Tr \left\{  B^{(0)} \prod_{\ell =1}^k (Y B^{(\ell)}) \prod_{\ell =1}^{k'} (C^{(k'-\ell+1)} Y ^ * )  \right\} = w \prod_{\ell=1} ^k v_\ell  \IND_{\{ k = k'\}} + o\left(1\right),
\end{equation}
when $X_{11}$ has finite moments of any order.

\noindent \paragraph{Step three : replacement principle.}
We  prove that if $\tilde Y = \tilde X / \sqrt N$ has iid centered entries with $\dE | \tilde X_{11} |^2  = \dE | X_{11} |^2 = 1$, $\dE \tilde X_{11} ^2= \dE X_{11}^2 $ and $\tilde X_{11}$ has finite moment of any order then
\begin{equation}\label{eq:ABXRP}
\frac 1 N \dE \Tr  P ( Y , Y^* )  -  \frac 1 N \dE \Tr  P ( \tilde Y , \tilde Y^* )   =  O\left( \frac  1 {\sqrt N}\right).
\end{equation}
To prove \eqref{eq:ABXRP}, we set, for $1 \leq \ell < k$, 
$
D^{(\ell)} = B^{(\ell)}$ , $ D^{(k)} = B^{(k)} C^{(k')} $, for $k +1 \leq \ell < k + k'$, $D^{(\ell)} =  C^{(k'-\ell + k )}$ and $D^{(k+k')} = B^{(0)}$. With this alternative notation, $\| D^{(\ell)} \| \leq 1$ and 
$$
  \frac 1 N   \Tr  P (  Y ,  Y^* )   = N^{- \frac k 2 - \frac {k'} 2 -1} \Tr X D^{(1)} \ldots X D ^{(k)} X^* D^{(k+1)} \ldots X^* D^{(k+k')}.
$$
We get
\begin{equation}\label{eq:trPYY}
\frac 1 N  \dE \Tr  P ( Y , Y^* )   = N^{- \frac k 2 - \frac {k'} 2 -1} \sum  \prod_{\ell = 1}^{k+k'} D^{(\ell)}_{i_{2 \ell} i_{2 \ell +1}} \dE \prod_{\ell = 1}^{k} X_{i_{2\ell -1} i_{2\ell} } \prod_{\ell = k+1}^{k+k'}  \bar X_{i_{2\ell} i_{2\ell-1 } },
\end{equation}
where the sum is over all $1 \leq i_s  \leq N$, $1 \leq s  \leq  2(k+k')$ and $i_{2 (k + k ') +1 } = i_1$. 
The summand in the above expectation will be non-zero only if the pairs of index $(i_{2 \ell-1}, i_{2\ell})$, $1 \leq \ell \leq k$ and $(i_{2 \ell}, i_{2\ell-1})$, $k +1 \leq \ell \leq k+k'$ appears at least twice.  Hence, there are $1 \leq q \leq \lfloor (k+k' ) /2 \rfloor$ distinct such pairs and $p \leq 2q$ distinct index in $(i_1, \ldots , i_{2(k+k')})$. We may thus decompose the above expectation as a finite sum (depending on $k,k'$) of terms of the type $N^{- \frac k 2 - \frac {k' }2 - 1}S(\pi)$ with
\begin{equation}\label{eq:graphcount'}
S (\pi) =    c(\pi) \sum    \prod_{\ell = 1}^{k+k'} D^{(\ell)}_{i_{\pi(2 \ell)} i_{\pi(2 \ell +1)}},
\end{equation}
where the sum is over all $1 \leq i_1 ,  \cdots , i _{p} \leq N$, $\pi :\bl 1, 2( k +k')\br \to \bl  p \br$ is a fixed surjective map such that for all $(u,v) \in \bl p \br ^2$, $n(u,v) = n_1 (u,v) + n_2(u,v) = \sum_{\ell =1}^{k} \IND_{\{  ( \pi (2\ell-1) , \pi (2 \ell )) = (u,v) \}} +\sum_{\ell =k+1}^{k+k'} \IND_{ \{  ( \pi (2\ell) , \pi (2 \ell -1 )) = (u,v) \} } \ne 1$. In \eqref{eq:graphcount'}, we have used the convention that $\pi(2(k+k') + 1)= \pi(1)$. Finally,
$$
c(\pi) = \prod_{1 \leq u , v \leq p} \dE X_{11}^{n_1(u,v)} \bar X_{11}^{n_2(u,v)} = O(1).
$$
We may restrict further ourselves to mapping $\pi$ such that $\pi(1) = 1$ and if $\pi(2k) \ne 1$, $\pi(2k)= p$.  For $(u,v) \in \bl p \br ^2$, we set $m(u,v) = \sum_{\ell =0}^{k} \IND_{ \{  ( \pi (2\ell) , \pi (2 \ell +1 )) = (u,v) \}} +\sum_{\ell =k+1}^{k+k'} \IND_{ \{  ( \pi (2\ell-1) , \pi (2 \ell )) = (u,v) \}}  $.  We consider the graph $G = (V,E)$ (with loops and multiple edges) on the vertex set $V = \bl p \br$ and edge multiplicities $M(\{u,v\}) = m(u,v) + m(v,u) \IND( v \ne u)$. Similarly to the proof of Proposition  \ref{prop:bilinearP}, the condition $n(u,v) \ne 1$ implies that $\deg(u) =  \sum_{v \in V} M(u,v)  \geq 2$ unless of the two following symmetric cases occur
\begin{enumerate}[(i)]
\item 
$u= 1$,  $\pi(1) = \pi (2(k+k')) = 1$ and $\{1\}$ is a connected component of $G$,
\item 
$u = p$, $\pi(k) = \pi (k+1) = p$ and $\{p\}$ is a connected component of $G$.
\end{enumerate}
Finally, we denote by $\Lambda \subset V$ the set of vertices with a loop, i.e. the set of $u \in V$ such that $m(u,u) \geq 1$. Arguing as in \eqref{eq:Stemp},  we find easily by recursion that $
|\Lambda| \leq  2 q -  p +2 $.

 Consider a connected component of $G$ say, $H= (V',E')$. We set $\veps' = 1$ if $H'$ contains a vertex in $\Lambda$ and $0$ otherwise. If  $V' \ne \{1\}$ and $V' \ne \{p\}$ as in $(i)$-$(ii)$, then all vertices of $H$ have degree at least $2$. Hence, $H$ contains a cycle and the argument leading to \eqref{eq:boundS'}  gives 
$$\sum_{i_v , v \in V'  }     \prod_{ \ell \in L'}   | D^{(\ell)}_{i_{\pi(2 \ell)} i_{\pi(2 \ell +1)}}| \leq N^{\frac{|V'| + \veps'}{2}},
$$ 
where $L'$ is the set of $\ell$ such that $\pi(2\ell) \in V'$ and $\veps' \in \{0,1\}$ is equal to $1$ if $H'$ contains a vertex in $\Lambda$ and $0$ otherwise.  Similarly, if $V' = \{1\}$ then we are in case $(i)$ and the contribution of this connected component is bounded by
$$
\sum_{i_1  }      | D^{(k+k')}_{i_{1} i_{1}}| \leq N = N^{\frac{|V'| + \veps'}{2}}.
$$
The same bounds apply if $V' = \{p\}$ and we are in case $(ii)$.  Taking the product over all connected components of $G$, we deduce that 
$$
S(\pi) \leq C N^{\frac{p + |\Lambda|}{2}}  \leq C N^{q+1},
$$

Now if $Y$ and $\tilde Y$ are as above, $S(\pi)$ are equal unless there is at least one of the $q$ distinct edges which appears more than twice. In particular, in such case $1 \leq q <   (k+k' ) /2 $. Hence for such $\pi$, we have $N^{- \frac k 2 - \frac {k' }2 - 1}S(\pi) =  O ( N^ {-1/2} )$. It proves \eqref{eq:ABXRP}.

\noindent \paragraph{Step four : Gaussian case.}
It remains to prove \eqref{eq:ABXaim} when $Y$ is complex Gaussian with $\dE |X_{11}|^2 =1$.  We will adapt an argument of \cite{MR2060040}. We will prove the following statement by recursion: for  all $p$,  $k_i, k'_i\geq 0$, $i=1,\ldots,p$ and matrices $B_i^{(\ell)}$, $C_i^{(\ell)}$, 
\begin{equation}\label{eq:recPYY}
 \dE \left[ \prod_{i=1}^p\frac 1 N \Tr \left\{  B_i^{(0)} \prod_{\ell =1}^{k_i} (Y B_i^{(\ell)}) \prod_{\ell =1}^{k'_i} (C_i^{(k'_i-\ell +1)} Y ^ * )  \right\} \right]=  \prod_{i=1}^p \frac 1 N \Tr B_i^{(0)}  \prod_{\ell=1} ^{k_i}  \left( \frac 1 N  \Tr B_i^{(\ell)} C_i^{(\ell)} \right) \IND_{\{ k_i = k_i'\}} + o\left( 1 \right),
\end{equation}
where the $o(1)$ is uniform over all $B_i^{(\ell)}$, $C_i^{(\ell)}$ of norm at most $1$ but depends on $k_i,k_i'$.  
Define $k=\sum_{i=1}^p k_i$ and  $k'=\sum_{i=1}^p k_i'$. The case $k+k'=0$ is obvious with the convention that $\prod_{l=1}^0 \equiv 1$. The case $k+k'=1$ is obvious since $X_{11}$ is centered.  
Now, it is easy to check that 
$$\dE \left( \frac 1 N  \Tr   B_1^{(0)} Y B_1^{(1)}  \frac 1 N  \Tr   B_2^{(0)} C_2^{(1)} Y^* \right)=  \frac{1}{N^3}  \Tr   B_1^{(1)}  B_1^{(0)}  B_2^{(0)} C_2^{(1)}  = o\left( 1 \right).$$
$$\dE \left( \frac 1 N  \Tr   B_1^{(0)} Y B_1^{(1)}   C_1^{(1)} Y^* \right)=  \frac{1}{N}  \Tr   B_1^{(0)}   \frac{1}{N}  \Tr  B_1^{(1)} C_1^{(1)}.$$
$$\dE \left( \frac 1 N  \Tr   B_1^{(0)} Y B_1^{(1)} Y  B_1^{(2)} \right)=  \frac{\dE(X_{11}^2)}{N^2}  \Tr   B_1^{(2)}  B_1^{(0)} { B_1^{(1)}}^{\top}  = o\left( 1 \right).$$
$$\dE \left( \frac 1 N  \Tr   B_1^{(0)} Y B_1^{(1)}   \frac 1 N  \Tr   B_2^{(0)} Y B_2^{(1)} \right)=  \frac{\dE(X_{11}^2)}{N^3}  \Tr   B_1^{(1)}  B_1^{(0)}  {B_2^{(0)}}^{\top} {B_2^{(1)}}^{\top}  = o\left( 1 \right).$$
Therefore \eqref{eq:recPYY} holds for $k + k ' =2$.
 We thus assume that \eqref{eq:recPYY} holds for all $k + k ' \leq n$ for some $n \geq 2$. We take $k + k ' = n +1$.   We use the identity in law 
$$
Y \stackrel{d}{=}\frac{ Y_1 + Y_2 }{\sqrt 2},
$$
where $Y_1$ and $Y_2$ are two independent copies of $Y$. We develop the right hand side of \eqref{eq:recPYY} in $Y_1$ and $Y_2$
\begin{eqnarray*}
&\hspace*{-8cm} \dE \left[ \prod_{i=1}^p \frac 1 N  \Tr \left\{  B_i^{(0)} \prod_{\ell =1}^{k_i} (Y B_i^{(\ell)}) \prod_{\ell =1}^{k_i'} (C_i^{(k'_i-\ell+1)} Y ^ * )  \right\}\right] 
\\  \quad&=  2^{-k/2 - k'/2} \sum_{\veps_i \in \{1,2\}^{k+k'},i=1\ldots p }\hspace{-1pt}  \dE \left[ \prod_{i=1}^p\frac 1 N  \Tr \left\{  B_i^{(0)} \prod_{\ell =1}^{k_i} (Y_{\veps_i^{(\ell)}} B_i^{(\ell)}) \prod_{\ell =1}^{k'_i} (C_i^{(k'_i-\ell+1)} Y_{\veps_i^{(k_i+ \ell)}}^ * )  \right\} \right].
\end{eqnarray*}
The two terms in the summand with for all $i$, $\veps_i = ( 1, \ldots, 1)$ and for all $i$,  $\veps_i = (2,\cdots , 2)$ are  equal to the left hand side of \eqref{eq:recPYY}. For the $2^{k+k'}-2$ other terms, we may condition on $Y_2$. For such vector $\veps$, we have $\sum_i\sum_{\ell} \IND (\veps_i^{(\ell)} = 1) \leq k + k ' - 1 \leq n$. We can thus use the recursion hypothesis by integrating over $Y_1$ and conditioning on $Y_2$. We then use $\dE \|Y_2 \|^{n} < C_n$ and the recursion hypothesis by now integrating over $Y_2$.  We find their contribution is $o(1)$ unless for each $i$, $k_i = k_i'$ and $\veps_{2k_i - \ell +1}  = \veps_{\ell}$ in which case 
\begin{eqnarray*}
&\hspace*{-5cm}  \dE \left[ \prod_{i=1}^p\frac 1 N  \Tr \left\{  B_i^{(0)} \prod_{\ell =1}^{k_i} (Y_{\veps_i^{(\ell)}} B_i^{(\ell)}) \prod_{\ell =1}^{k'_i} (C_i^{(k'_i-\ell+1)} Y_{\veps_i^{(k_i+ \ell)}}^ * )  \right\} \right]\\& \quad \quad \quad \quad \quad \quad=\prod_{i=1}^p \frac 1 N \Tr B_i^{(0)}  \prod_{\ell=1} ^{k_i}  \left( \frac 1 N  \Tr B_i^{(\ell)} C_i^{(\ell)} \right) \IND_{\{ k_i = k_i'\}} + o\left( 1 \right).
\end{eqnarray*}
For $k = k'$ (which implies $k \geq 2$), there are $2^{k} - 2$ such vectors $\veps$. It follows that 
\begin{align*}
&\hspace*{-1cm} \dE \left[ \prod_{i=1}^p\frac 1 N \Tr \left\{  B_i^{(0)} \prod_{\ell =1}^{k_i} (Y B_i^{(\ell)}) \prod_{\ell =1}^{k'_i} (C_i^{(k'_i-\ell +1)} Y ^ * )  \right\} \right]  \\
& \quad\quad\quad\quad\quad = 2 . 2^{-k} \dE \left[ \prod_{i=1}^p\frac 1 N \Tr \left\{  B_i^{(0)} \prod_{\ell =1}^{k_i} (Y B_i^{(\ell)}) \prod_{\ell =1}^{k'_i} (C_i^{(k'_i-\ell +1)} Y ^ * )  \right\} \right]  \\
& \quad\quad\quad\quad\quad \quad  + (2^{k} - 2). 2^{-k} \prod_{i=1}^p \frac 1 N \Tr B_i^{(0)}  \prod_{\ell=1} ^{k_i}  \left( \frac 1 N  \Tr B_i^{(\ell)} C_i^{(\ell)} \right) \IND_{\{ k_i = k_i'\}} + o\left( 1 \right).
\end{align*}
Since $k \geq 2$, we obtain \eqref{eq:recPYY}. 
\end{proof}

\subsection{Proof of Theorem \ref{th:TCL}}

We set 
$$ 
X_N =\begin{pmatrix} X_{r\times r}   &  X_{r \times (N-r)} \\  X_{( N-r) \times r }  &  X_{(N-r) \times (N-r)}\end{pmatrix} , \;  \quad Y_N = \frac {X_N }{\sqrt N} =\begin{pmatrix} Y_{r\times r}   &  Y_{r \times (N-r)} \\  Y_{( N-r) \times r }  &  Y_{(N-r) \times (N-r)}\end{pmatrix},
$$
and $$
\hat M_N = \sigma  Y_{(N-r) \times (N-r)} + \hat A_{N - r}, \; \quad  \hat R_N (z) = ( z I_{N-r} - \hat M_N )^{-1}.
$$
We fix any $a \in S$. The matrix $A_N$ satisfies assumptions (A1'-A4) with,
$$A'_N =   \left( \begin{array}{c|c}
a I_r& 0 \\
\hline
0 & \hat{A}_{N-r}  
 \end{array} \right), \quad \quad A''_N =   \left( \begin{array}{c|c}
(\theta_N-a) I_r& 0 \\
\hline
0 & 0 
 \end{array} \right). $$ 

Thus, by Theorem \ref{th:main}, for  $\delta < \eta$ small enough,
 almost surely  for all large $N$, there are exactly $r$ eigenvalue
$\lambda_i$, $i=1,\ldots,r$ of $M_N$ in $B(\theta,\delta).$  Note that almost surely $\lim_{N\rightarrow +\infty} ( \lambda_i -\theta) =0.$ It concludes the proof of the first statement of Theorem \ref{th:TCL}.

For the second statement, we start with a consequence of Proposition \ref{nonunif}. 
 \begin{lemme} \label{le:wellcond}
For any $0 < \delta< \eta$, there exists $C>0$ such that almost surely for all large $N$, 
\begin{equation*}\label{borneG} \sup_{z \in B(\theta,\delta)}
\left\| \hat R_N(z) \right\| \leq C.\end{equation*}
\end{lemme}

\begin{proof}
By \eqref{eq:CFres}, the smallest singular values of $(z I_{N-r} - \hat A_{N-r}) $ is uniformly lower bounded on $B(\theta,\delta)$ by $\eta - \delta >0$. It remains to  apply Proposition \ref{nonunif} to $\hat M_N$ and $\Gamma = B ( \theta, \delta)$.
\end{proof}

By Lemma \ref{le:wellcond}, a.s. for all large $N$ and all $z\in B(\theta, \delta) $, $z I_N - \hat M_N$ is invertible. Thus, since
$$M_N - \lambda I_N =\begin{pmatrix} \sigma Y_{r\times r}+( \theta_N - \lambda) I_r  & \sigma Y_{r \times (N-r)} \\ \sigma Y_{( N-r) \times r}& \hat M_N - \lambda I_{N-r} \end{pmatrix},
$$ and then, from Jacobi's determinant formula, 
\begin{align*}
 \det( M_N - \lambda I_N) =  \det( \hat M_N - \lambda I_{N-r} ) \det \left( \sigma Y_{r\times r} + (\theta_N- \lambda) I_r  +  \sigma^2 Y_{r \times N-r} \hat R_N (\lambda) Y_{N-r \times r} \right).
\end{align*}
 Now, using the resolvent identity
    $$
 \hat R_N (\lambda)-\hat R_N({\theta_N})= - (\lambda- {\theta}_N)
   \hat R_N({\theta_N})\hat R_N (\lambda),$$
 one can replace   $\hat R_N(\lambda)$ by $\hat R_N({\theta_N})+\left[-(\lambda -{\theta_N})\hat R_N ({\theta_N})\left(\hat R_N({\theta}_N) -(\lambda -{\theta_N})\hat R_N ({\theta_N})
\hat R_N (\lambda) \right)\right]$ and get that, a.s. for all $N$ large,  $\lambda \in B(\theta, \delta) $ is an eigenvalue of $M_N$ if and only if
\begin{equation}
\label{eq:deffN}  \det \left( V_N  -\sqrt{N}(\lambda-\theta_N) \left[I_{r} + C_N(\lambda)
\right]\right)=0,
\end{equation}
where $$
V_N = \sigma {X_{r\times r}} 
+ \frac{ \sigma^2 }{\sqrt N} X_{r \times N-r}\hat R_N (\theta_N) {X_{N-r \times r}},
$$
and \begin{equation}\label{CN} C_N(\lambda)= \sigma^2  Y_{r \times N-r}\hat R_N (\theta_N)^2 {Y_{N-r \times r}}-(\lambda-\theta_N) \sigma^2  Y_{r \times N-r}\hat R_N (\theta_N)^2\hat{R}_N(\lambda){Y_{N-r \times r}}.\end{equation}

\begin{lemme}\label{le:boundCN}
If $\delta_N $ converges to $0$ as $N$ goes to infinity, a.s. , as $N\to \infty$, 
$$
\sup_{ \lambda : | \lambda - \theta | \leq \delta_N} \| C_N ( \lambda) \| \to 0. 
$$\end{lemme}
\begin{proof}
The convergence towards zero of the first term in (\ref{CN}) readily follows from Proposition \ref{prop:bilinearP} and Lemma \ref{le:wellcond}.
For the second  term in (\ref{CN}), from Lemma \ref{le:wellcond}, Bai-Yin Theorem \cite{MR1235416}, there exists $C' >0$ such that we have a.s. for all large $N$, $\|  Y_{r \times N-r}\hat R_N (\theta_N)^2\hat{R}_N(\lambda){Y_{N-r \times r}} \| \leq C'$. Since $\delta_N \to 0$, the a.s. convergence toward zero of second  term in (\ref{CN}) follows.
\end{proof}
\begin{lemme}\label{FQ} Under the assumptions of Theorem \ref{th:TCL},
the random matrix $\frac{1}{\sqrt{N}} X_{r \times N-r}\hat R_N (\theta_N) {X_{N-r \times r}}$ converges in distribution towards $G$ the complex $r \times r$ Ginibre matrix defined in \eqref{defV}. 
\end{lemme}
\begin{proof}[Proof of Lemma \ref{FQ}]
Note that $$\frac{1}{\sqrt{N}} X_{r \times N-r}\hat R_N (\theta_N) {X_{N-r \times r}}= \left( \frac{1}{\sqrt{N}} {x^{(p)}}^\top \hat R_N (\theta_N) y^{(q)}  \right)_{1\leq p,q \leq r}$$
where ${x^{(1)}}, \ldots,  {x^{(r)}}$ and $y^{(1)} , \ldots,y^{(r)} $ are the columns of $X^\top _{r \times N-r}$ and $X_{N-r \times r}$ respectively. Now, by Lemma \ref{le:wellcond}, a.s. $\| \hat R_N (\theta)  \|  $ is uniformly bounded  and arguing as in Lemma \ref{continuitephi},  $\frac{1}{N-r} \Tr \hat R_N (\theta) \hat R_N (\theta)^*$ converges towards $\int \lambda^{-1} d\mu_{ \theta}(\lambda)$. Also, from the resolvent identity and Lemma \ref{le:wellcond}, 
$$
\| \hat R_N (\theta_N)  - \hat R_N (\theta ) \| \leq \| \hat R_N (\theta_N) \| \| \hat R_N  ( \theta) \| | \theta - \theta_N| \to 0. 
$$
Thus, $\frac{1}{N-r} \Tr \hat R_N (\theta_N) \hat R_N (\theta_N)^*$ and $\frac{1}{N-r} \Tr \hat R_N (\theta_N) \hat R_N (\theta_N)^\top$ are asymptotically close to  $$\tau_N = \frac{1}{N-r} \Tr \hat R_N (\theta) \hat R_N (\theta)^*\quad \hbox{ and } \quad \zeta_N = \frac{1}{N-r} \Tr \hat R_N (\theta) \hat R_N (\theta)^\top,$$  respectively. We will now use Proposition \ref{prop:ABX} to obtain a nice expression for $\tau_N$ and $\zeta_N$. We set $R'_N (z) = ( \hat A_{N- r} - z I_{N-r} ) ^{-1}$. We first observe that by Proposition \ref{borne},  a.s.  the series expansion 
$$
\hat R_N (\theta) = \sum_{\ell \geq 0} (-\sigma)^\ell R'_N(\theta) \left(  Y_{(N-r) \times (N-r)} R'_N(\theta) \right)^{\ell}
$$
 converges in norm uniformly in $N \geq N_0$. In order to prove that $\zeta_N$ converges, it is thus sufficient to prove the a.s. convergence of normalized traces of the form, for fixed integers $k,\ell \geq 0$,
$$
\zeta_N^{(k,\ell)} = \frac 1 {N- r} \Tr \left\{ R'_N(\theta) \left(  Y_{(N-r) \times (N-r)} R'_N(\theta) \right)^{\ell} \left( R'_N(\theta) ^\top  Y_{(N-r) \times (N-r)} ^\top  \right)^{k} R'_N(\theta) ^\top \right\}
$$
to a complex number $\zeta^{(k,\ell)}$. Applying Proposition \ref{prop:ABX}, we find that $\zeta^{(k,\ell)} = \IND( k = \ell)( \dE X_{11}^2 )^k \psi^{k+1}$. Hence 
$$
\zeta_N = \sum_{k \geq 0} \sigma^{2k} (\dE X_{11}^2)^k  \psi^{k+1} + o(1)= \frac{ \psi }{ 1 - \sigma^2 \dE X_{11}^2 \psi} +o (1). 
$$
Similarly, 
$$
\tau_N = \sum_{k \geq 0} \sigma^{2k}    \varphi^{k+1} + o(1)= \frac{ \varphi }{ 1 - \sigma^2 \varphi} +o (1). 
$$
The lemma is thus a consequence of Proposition \ref{TCLFQ}, in Appendix.
\end{proof}

We start by a classical application of Rouch\'e's Theorem for random analytic functions.  Recall that we endow the set of analytic functions on a open connected set $U$ with the distance defined by \eqref{eq:distHU}. The next lemma is contained in Shirai \cite[Proposition 2.3]{MR3014854}.
\begin{lemme}\label{le:roucheRAF}
Let $U$ be a bounded connected open set and $f_N$ a  tight sequence of random analytic functions on $U$ converging weakly to $f$ for the finite dimensional convergence. Then, if $f$ is a.s. non-zero, the point process of zeros (with multiplicities) of $f_N$ in $U$ converges weakly to the point process of zeros of $f$ in $U$: i.e. for any continuous $\phi: U \rightarrow \mathbb{R}$ with compact support $K$,
denoting respectively  by $x_i^{(N)}$ and $x_i$ the zeroes of $ f_N $ and $f$ in $K$, $\sum_{i } \phi(x_i^{(N)}) $ converges weakly to  $\sum_{i} \phi(x_i) $.   
\end{lemme}
\begin{proof} As already pointed, analytic functions on $U$ endowed with distance \eqref{eq:distHU} is a Polish space. Hence, from Skorokhod's representation theorem, it suffices to check the following deterministic statement: if $f_N$ is a sequence of analytic functions converging to a non-zero analytic function $f$, then the point set of zeros of $f_N$ converges to the point set of zeros of $f$. 
Let $J$ be such that $K\subset \mathring{K}_J$ where $(K_j)_{j \geq 1}$ is an exhaustion by compact sets of $U$. It is clear that $f$ has a finite number $q$ of zeroes in  $K_J$. Let $\epsilon >0$ and $\delta>0$ be such that for any $x $ and $y$ in $K_J$, if $\vert x-y\vert < \delta$ then $\vert \phi(x)-\phi(y) \vert < \epsilon/q$. For any $z \in K$, choose $0<\delta_z < \delta$ such that $B(z,\delta_z)\subset  \mathring{K}_J$  and such that the minimum of $|f(z)|$ on $\partial B(z,\delta_z)$ is strictly positive.  From Rouch\'e's Theorem, for all large $N$, the number of zeros of $f_N$ and  $f$ in $\mathring{B}(z, \delta_z)$ are equal. Extracting a finite subcover of $K$ from $\cup_{z \in K}  \mathring{ B}(z,\delta_z)$, it follows that for all large $N$, $\left| \sum_{i } \phi(x_i^{(N)}) -\sum_{i} \phi(x_i)\right| \leq \epsilon $. 
The conclusion follows. \end{proof}

We may now conclude the proof of Theorem \ref{th:TCL}. We introduce  
\begin{eqnarray*}
h_N(z)&=&\det \left( V_N -z  \left[I_{r} + C_N\left(\frac{z}{\sqrt{N}} +\theta_N\right)
\right]
\right),\\
 h(z)& = &\det(V - z I_r),\end{eqnarray*}
where $V$ is defined in (\ref{defV}).

Let $t >0$. Using the continuity of the determinant, by Lemma \ref{le:boundCN} and Lemma \ref{FQ}, for any $(z_1, \ldots, z_k) \in B(0,t)$, $(h_N(z_1), \ldots, h_N(z_k))$ converges weakly to $(h(z_1), \ldots, h(z_k))$. Moreover, by  Lemma \ref{le:boundCN} and Lemma \ref{FQ}, for any $\epsilon>0$, there exists $C>0$ such that $\mathbb{P} \left( \Vert h_N \Vert_{B(0,t)} \leq C \right) \geq 1-\epsilon.$ By Montel's theorem, the set of holomorphic functions $f$ on  $\mathring{B}(0,t)$ such that $\sup_{z\in  \mathring{B}(0,t) }\vert f(z) \vert \leq C$ is compact in the set of analytic functions on $\mathring{B}(0,t)$. Therefore,
 $h_N$ is a tight sequence of random analytic functions on $\mathring{B}(0,t)$. It follows from Lemma \ref{le:roucheRAF} that the zeros of $h_N$ in $\mathring{B}(0,t)$ converges to the zeros of $h$ in $\mathring{B}(0,t)$, or equivalently, from \eqref{eq:deffN}, the eigenvalues of $\sqrt N ( M_N - \theta_N) $ in $\mathring{B}(0,t) $ converges weakly to the eigenvalues of $V$ in $\mathring{B}(0,t)$. Since, as $t \to \infty$, the probability that  the $r$ eigenvalues of $V$ are in $B(0,t)$ goes to $1$, the theorem follows.

\subsection{Proof of Theorem \ref{th:TCL2}}

The argument leading to \eqref{eq:deffN}  shows that, a.s. for all $N$ large,  $\lambda \in B(\theta, \delta) $ is an eigenvalue of $M_N$ if and only if
\begin{equation}\label{eq:deffN2}
 \det \left( \frac{V_N}{\sqrt{N}} - \left[\lambda I_r - \hat A_r  + ( \lambda- \theta_N) C_N(\lambda)
\right]\right)=0.
\end{equation}

We are going to rewrite conveniently the above equation. Let $D = \DIAG ( \omega_1, \cdots, \omega_r)$ with $\omega_k = \exp ( 2 i \pi k / r)$, and the unit vectors in $\dC ^r$ with coordinates $v_k = 1/ \sqrt r$ and $u_k =\overline{\omega_k} / \sqrt r $.  So that, as in \eqref{eq:A'Nnil}-\eqref{eq:A''Nnil},  for some unitary matrix $U$, 
$$
J_N - \theta_N I_r = U ( D - v u ^* ) U^*  
$$
Set $Q_N = P_N U $, we get $\hat A_r  = Q_N ( \theta_ N I_r + D - v u^* ) Q_N^{-1}$. It follows from \eqref{eq:deffN2} that we are interested by the zeros in $B(\theta, \delta)$ of
\begin{equation*}
 \det \left( \frac{W_N}{\sqrt{N}} - (\lambda  - \theta_N ) ( I_r + C'_N(\lambda) ) + D - v u ^* \right),
\end{equation*}
where 
$$
W_N = Q_N^{-1} V_N Q_N \quad  \hbox{ and } \quad C'_N ( \lambda) = Q_N^{-1} C_N ( \lambda) Q_N. 
$$
Setting
$$
D_N ( z) =  D + \frac{W_N}{\sqrt{N}} - (z  - \theta_N ) ( I_r + C'_N(z) ).
$$

We will need a finer  uniform bound for $u ^* C'_N(\lambda) v$ than the one given by Lemma \ref{le:boundCN}. 

\begin{lemme}\label{le:boundCN2}
If  $\delta_N \leq N^{-\veps}$ for some $\veps >0$, then with probability tending to $1$ as $N  \to \infty$,
$$
\sup_{ \lambda : | \lambda - \theta_N | \leq \delta_N}  \Vert  C'_N ( \lambda) \Vert \leq \frac{ \log N  }{ \sqrt N }. 
$$\end{lemme}

\begin{proof}
We can assume without loss of generality that $\veps = 1/ 2 \ell$ for some integer $\ell \geq 1$. The resolvent formula iterated $\ell$ times gives
\begin{eqnarray*}
\hat R_N(\lambda) & = & \sum_{p=0}^{\ell-1} ( \theta_N - \lambda)^p   \hat R_N (\theta_N)^{p+1}  +   ( \theta_N - \lambda)^\ell\hat R_N(\theta_N)^{\ell}\hat R_N (\lambda).
\end{eqnarray*}
In the second term of \eqref{CN}, we expand the term in $\hat R_N(\lambda)$ as above. Setting $u_N = (Q_N^{*}) ^{-1} u$ and $v_N = Q_N v$, we find that 
$u^* C'_N ( \lambda) v $ is a finite sum of terms of the form, with $0 \leq q \leq \ell$,
$$S = ( \theta_N - \lambda)^q  u_N^*    Y_{r \times N - r} \hat R ( \theta_N)^{q +2}Y_{ N - r \times r} v_N$$ and of a final term of the form   
$$S' =  ( \theta_N - \lambda)^{\ell+1}   u_N^*  Y_{r \times N - r} \hat R_N(\theta_N)^{\ell+2}\hat R_N (\lambda) Y_{N - r \times r}  v_N.$$

We first bound $|S'|$. We use Lemma \ref{le:wellcond} and Bai-Yin Theorem \cite{MR1235416}. We find that there exists $C' >0$ such that a.s. for all large $N$, $\|  Y_{r \times N-r}\hat R (\theta_N)^{\ell +2} \hat{R}(\lambda){Y_{N-r \times r}} \| \leq C'$. In particular, a.s. $|S'|  = O ( \delta_N ^{\ell +1}) = o (1 / \sqrt N)$. 

It remains to upper bound  $|S|$. It suffices to prove that with probability tending to $1$ all entries of the matrix $Z_N = Y_{r \times N - r} \hat R ( \theta_N)^{q +3}Y_{ N - r \times r} $ are bounded by $\log N / \sqrt N$.  It follows from Proposition \ref{TCLFQ} in Appendix that, conditionally on $\| \hat R_N (\theta_N ) \| \leq C$, the random variables $\sqrt N (Z_{N})_{ij}$ are a tight sequence of random variables for each $1\leq i,j \leq r$.   Using again  Lemma \ref{le:wellcond} , the lemma follows. 
\end{proof}

We consider the sequence 
\begin{equation}\label{eq:defdeltaN}
\delta_N  = \frac{ \log N}{ N^{  1 / 2r }}. 
\end{equation}

 According to Lemma  \ref{le:boundCN2} and Lemma \ref{FQ},  the event $\Omega_N $ that for all $z$ in $B(\theta_N,\delta_N)$, $\| C'_N(z)\| \leq \log N / \sqrt N$ and $\| W_N\| \leq \log N $ has probability tending to $1$. If $\Omega_N$ holds then
$$
\sup_{z : |z - \theta_N| \leq \delta_N}  \| D- D_N (z) \| =  o(1). 
$$ 
Since all singular values of $D$ are equal to $1$, on the event $\Omega_N$, for all $z$ in $B(\theta_N,\delta_N)$, $D_N(z)$ is invertible  and $\| D_N (z)\|$ and $\|D_N(z)^{-1} \|$ are bounded by $1 + o (1)$. On this event $\Omega_N$, using \eqref{eq:idcle}, we deduce that $\lambda \in B(\theta_N, \delta_N) $ is an eigenvalue of $M_N$ if and only if
\begin{equation}\label{eq:deffN3}
 1 - u^* D_N ( \lambda ) ^{-1} v =0.
\end{equation}

The next lemma is main deterministic ingredient of the proof of Theorem \ref{th:TCL2}. 
\begin{lemme}\label{le:TCL2}
If $\Omega_N$ holds, then 
$$
u^* D_N (z) ^{-1} v = 1 + (z- \theta_N)^r-  \frac{ u^* D^{-1} W_N D^{-1} v }{\sqrt N} + o \left( \frac{ 1 }{ \sqrt N} \right), 
$$
where the $O( \cdot)$ is uniform on $B(\theta_N,\delta_N)$. 
\end{lemme}

\begin{proof}
As often, we use  the formula, 
$
D_N (z)  ^{-1} = D^{-1} + D^{-1} ( D - D_N (z) ) D_N (z)^{-1}.
$ Iterating, we find
\begin{equation}\label{eq:defDN}
D_N(z) ^{-1}  = \sum_{p=0}^{r} D^{-1} \left( ( D - D_N (z) ) D^{-1} \right)^p +  D^{-1} \left( ( D - D_N (z) ) D^{-1} \right)^{r} ( D - D_N (z) ) D_N (z)^{-1}
\end{equation}
First observe that, by definition, if $\omega = e^{2 i \pi / r}$, we have  
\begin{eqnarray*}
u^* D^{-q} v = \frac 1 r \sum_{k=0}^{r-1}  \omega^{- k (q-1)}  = \left\{ \begin{array}{ll}
1 & \hbox{ if $q = 1 \;  \mathrm{mod}(r)$} \\
0 & \hbox{ otherwise.} 
\end{array} \right. 
\end{eqnarray*}
Consequently, since
$$
D - D_N ( z) =   (z  - \theta_N )  I_r + (z - \theta_N) C'_N(z)  - \frac{W_N}{\sqrt{N}} =  (z  - \theta_N )  I_r +  B_N(z), 
$$
we have for all $2 \leq p \leq r-1$, $p \ne r$,  
$$
 \left| u^* D^{-1} \left( ( D - D_N (z) ) D^{-1} \right)^p v  \right| = O \left(  C_p \delta_N ^{p-1} \frac{\log N}{\sqrt N}\right) = o \left( \frac{ 1} {\sqrt N}\right). 
$$
on the event $\Omega_N$. Indeed, since $u^* D^{-p-1} v = 0$, the non-zero contributions in the expansion of $$\left(  ( (z  - \theta_N )  I_r +  B_N(z) ) D^{-1} \right)^p$$ must have at least one $B_N(z)$ and $\| B_N (z) \| \leq 2 \log N / \sqrt N$ on $\Omega_N$.  

Similarly, since $u^* D^{-r-1} v = 1$,
$$
 \left| u^* D^{-1} \left( ( D - D_N (z) ) D^{-1} \right)^{r}  v  - (z - \theta_N)^{r} \right| \leq C_r \delta_N ^{r-1} \frac{\log N}{\sqrt N} = o \left( \frac{ 1} {\sqrt N}\right). 
$$

Finally, on $\Omega_N$, $\| D - D_N(z) \| = O ( \delta_N)$. Hence, the last term of \eqref{eq:defDN} may bounded as follows
$$
\|  D^{-1} \left( ( D - D_N (z) ) D^{-1} \right)^{r} ( D - D_N (z) ) D_N (z)^{-1} \| = O ( \delta_N ^{r+1}) = o ( 1 / \sqrt N).
$$ 

In summary, on the event $\Omega_N$, from \eqref{eq:defDN}, we have 
\begin{eqnarray*}
u^* D_N(z)  ^{-1} v & = & 1 + u^*D^{-1} B_N (z) D^{-1}v +  (z - \theta_N)^{r} + o ( 1/ \sqrt N) \\
& = & 1 - \frac{  u^*D^{-1} W_N D^{-1}v}{\sqrt N} + (z - \theta_N)^{r} + o (1/ \sqrt N),
\end{eqnarray*}
where we have used that, on $\Omega_N$, $\| B_N (z) +  W_N / \sqrt N \| = O ( \delta_N \log N / \sqrt N)$.  \end{proof}

We are now ready to conclude the proof of Theorem \ref{th:TCL2}. Fix $t > 0$, if $\Omega_N$ holds, then from \eqref{eq:deffN3}, the random function 
$$
f_N (z) = \sqrt N \left( 1 - u^* D _N ( \theta_N + z / N^{1/2r} ) ^{-1} v\right)
$$
is analytic on $B(0,t)$ and its zeros are the eigenvalues of $N^{1/2r} ( M_N - \theta_N)$ in $B(0,t)$. Moreover, from Lemma \ref{FQ} and Lemma \ref{le:TCL2}, $f_N(z)$ converges in distribution to the random polynomial on $B(0,t)$, 
$$
f (z) = u^* D^{-1} W D^{-1} v - z^r, 
$$
where $W = Q^{-1} V Q$, $Q = P U$ and $V$ given by Theorem \ref{th:TCL}. It is straightforward to check that $u^* D^{-1} = v^* $, $D^{-1} v = u$, $U v= e_r$ and $U u = e_1$. We may thus rewrite 
$$
f (z) = e_r^*  P^{-1} V P e_1  - z^r.
$$
It remains to use Lemma \ref{le:roucheRAF} and conclude as in the proof of Theorem \ref{th:TCL}.

\section{Unstable outliers}

\label{sec:unstable}

\subsection{Tightness}

The objective of this subsection is to prove the following proposition. It will be used to obtain the tightness of the point process of eigenvalues. 

\begin{proposition}\label{prop:momentunif}
Let $B \in M_N( \dC)$ be a diagonal matrix, $u, v \in \dC^N$ and $M > 0$. Assume that $  N \| B \|^2  \leq  M  \tr (BB^*)$, and $\|u \|_\infty , \|v\|_{\infty} \leq M / \sqrt N$. If assumptions (X1-X3) hold, there exists a constant $C >0$ depending on $M$ such that for any integer $1 \leq k \leq N^{1/C}$, we  have
$$
\dE \left| \sqrt N  u^* \left( B \frac {X_N}{\sqrt N}  \right)^{k} v  \right|^2  \leq  C  k^4 \left( \frac 1 N \tr ( B B^*)  \right) ^k , 
$$
\end{proposition}

\begin{proof}
We start with the first statement. We have 
\begin{eqnarray}\label{eq:tightm2}
\dE  \left| \sqrt N  u^* \left( B X_N \right)^{k} v  \right|^2 = N  \sum_{\ibf, \jbf} \bar u_{i_0} v_{i_k} u_{j_0} \bar v_{j_k} \dE  \prod_{\ell = 0} ^{k-1} B_{i_\ell i_\ell} X_{i_\ell i_{\ell +1}}  \bar B_{j_\ell j_\ell} \bar X_{j_\ell j_{\ell +1}},
\end{eqnarray}

where the sum is over all $k$-tuples, $\ibf = (i_0, \cdots, i_k)$, $\jbf = (j_0, \cdots, j_k)$. Only pairs $(\ibf, \jbf)$ such that  for each $\ell$, $(i_{\ell} , i_{\ell +1})$  and $(j_{\ell} ,  j_{\ell +1})$  appear at least twice will matter. For such pair $(\ibf,\jbf)$, we will consider the oriented graph  $G(\ibf,\jbf)$ with vertex set $\ibf \cup \jbf$ and edge set $\{ (i_{\ell} , i_{\ell +1}) , 0 \leq \ell \leq k-1 \} \cup \{ (j_{\ell} ,  j_{\ell +1}) , 0 \leq \ell \leq k-1\}$.

We shall first treat the case where $\ibf \cap \jbf \ne \emptyset$. Then, the graph $G(\ibf,\jbf)$ is connected. It has at most $k$ edges and $k+1$ vertices.  For $1 \leq t \leq k+1$, let $W(k,t)$ be the subset of the pairs $(\ibf,\jbf)$ where in addition the number of distinct elements in $(\ibf,\jbf)$ is $t$.

If $(\ibf,\jbf) \in W(k,t)$, there are at least $t-1$ edges in $G(\ibf,\jbf)$. Let $L$ be the number of edges with multiplicity $2$, we have $2 L + 3( t  -1 - L) \leq 2k$ and thus $2 (k-L) \leq 6 ( k-t +1)$. By assumptions (X1)-(X3) and H\"older inequality, we find
\begin{equation}\label{eq:mmmXX}
\dE  \prod_{\ell = 0} ^{k-1} | X_{i_\ell i_{\ell +1}}  \bar X_{j_\ell j_{\ell +1}}| \leq ( \dE |X_{11}|^2 ) ^{L} \dE |X_{11}|^{ 2 k - 2 L } \leq (c k )^{6 c (k-t+1)}. 
\end{equation}
We say that two pairs $(\ibf,\jbf)$, $(\ibf',\jbf')$ are isomorphic if there exists a bijection $\sigma \in S_N$ such that  $\sigma(i_{\ell}) = i'_\ell$ and $\sigma(j_{\ell}) = j'_\ell$.  Let $\cW(k,t)$ be the set of equivalence classes of elements of $W(k,t)$. Each element $w \in \cW(k,t)$ has $N(N-1) \cdots (N-t+1)$ pairs $(\ibf,\jbf)$ in its equivalence class. For any $w \in \cW(k,t)$, since $|B_{ii}| \leq \|B\|$, we get the bound
\begin{eqnarray}
\sum_{(\ibf,\jbf) \in w} \prod_{\ell = 0} ^{k-1} |B_{i_\ell i_\ell}   \bar B_{j_\ell j_\ell}| & \leq & \sum_{(i_1\cdots , i_t) : \hbox{distinct}} \prod_{\ell =1} ^{t-1} |B_{i_\ell i_\ell} |^2 \|B\|^ {2 ( k -t+1)}  \nonumber \\
& \leq & N \tr ( B B^*)^{t-1} \| B \|^ {2 ( k -t+1)} \nonumber  \\
& \leq & N \tr ( B B^*) ^k  ( M / N)^{  k-t +1} \label{eq:mmmBB}, 
\end{eqnarray}
where at the last line we have used the assumption $\|B\| \leq \sqrt{ (M / N) \tr (B B^*) }$. So finally, in \eqref{eq:tightm2}, we get the upper bound for some constant $C >0$,
\begin{eqnarray*}
 I & = &\dE \left|  N  \sum_{\ibf \cap \jbf \ne \emptyset} \bar u_{i_0} v_{i_k} u_{j_0} \bar v_{j_k} \prod_{\ell = 0} ^{k-1} B_{i_\ell i_\ell} X_{i_\ell i_{\ell +1}}  \bar B_{j_\ell j_\ell} \bar X_{j_\ell j_{\ell +1}}\right| \\ 
& \leq & M^2 \tr ( B B^*) ^k \sum_{t=1} ^{k+1} | \cW(k,t) | \left( \frac{ C k^{C}  }{N} \right) ^{k-t+1}. 
\end{eqnarray*}
Using Lemma \ref{le:pathFK} below, we get for some new constant $C >0$,
$$
I \leq  C  k^4 \tr ( B B^*) ^k \sum_{t=1} ^{k+1}   \left( \frac{ C k^{C}  }{N} \right) ^{k-t+1} = \, C  k^4 \tr ( B B^*) ^k \sum_{\ell=0} ^{k}   \left( \frac{ C k^{C}  }{N} \right) ^{2 \ell }. 
$$
If $C k ^C / N \leq 1/2$, we find that 
$$
I \leq 2 C k^4 \tr ( B B^*)^k. 
$$

To complete the proof of Proposition \ref{prop:momentunif}, we also need to take care in \eqref{eq:tightm2} of the pairs $(\ibf,\jbf)$ such that $\ibf \cap \jbf = \emptyset$. The proof is similar. We denote by $W'(k,t)$ the set of $\ibf = (i_0, \cdots, i_k)$ such that each oriented edge of the graph $G(\ibf)$ formed by $\ibf$ is visited at least twice and the vertex set of $G(\ibf)$ has cardinal $t$. Observe that $t \leq k/2$: indeed if $t = k/2 +1$ then $G(\ibf)$ is a tree, however to visit twice an oriented edge, there must be a cycle in $G(\ibf)$. It follows that in  \eqref{eq:tightm2} only pairs $\ibf\cap \jbf = \emptyset$ such that $\ibf$ and $\jbf$ are in $W'(k,s)$ and $W'(k,t)$ for some $1 \leq  s, t \leq k/2$ will contribute. 

Let $\ibf \in W'(k,t)$, since $G(\ibf)$ is not a tree, there are at least $t$ edges in $G(\ibf)$. If $L$ denotes the number of edges with multiplicity $2$ we find that $2 L + 3 ( t - L) \leq k$ or $k - 2 L \leq 3 ( k - 2 t)$.  Arguing as in \eqref{eq:mmmXX}, we find that 
\begin{equation*}
\dE  \prod_{\ell = 0} ^{k-1} | X_{i_\ell i_{\ell +1}}| \leq ( \dE |X_{11}|^2 ) ^{L} \dE |X_{11}|^{ k - 2 L } \leq (c k )^{ 3c (k-2t)}. 
\end{equation*}

Let $\cW'(k,t)$ denote the set of equivalence classes of elements in $W'(k,t)$. Arguing as in \eqref{eq:mmmBB}, for any $w \in \cW'(k,t)$, we get
\begin{eqnarray}
\sum_{\ibf \in w} \prod_{\ell = 0} ^{k-1} |B_{i_\ell i_\ell}  | & \leq & \sum_{(i_1\cdots , i_t) : \hbox{distinct}} \prod_{\ell =1} ^{t} |B_{i_\ell i_\ell} |^2 \|B\|^ { k -2t} \nonumber \\
& \leq & \tr ( B B^*)^{t} \| B \|^ { k -2 t} \nonumber \\
& \leq &\tr ( B B^*) ^{k/2}  ( M / N)^{k/2-t }\label{eq:mmmBB'},
\end{eqnarray}
where at the first line we have used the fact that each vertex in $G(\ibf)$ is adjacent to at least two edges.

Putting the above estimates together, 
\begin{eqnarray*}
 I' & = &\dE \left|  N  \sum_{\ibf \cap \jbf = \emptyset} \bar u_{i_0} v_{i_k} u_{j_0} \bar v_{j_k} \prod_{\ell = 0} ^{k-1} B_{i_\ell i_\ell} X_{i_\ell i_{\ell +1}}  \bar B_{j_\ell j_\ell} \bar X_{j_\ell j_{\ell +1}}\right| \\ 
& \leq & \frac { M^2}{ N} \tr ( B B^*) ^{k} \left( \sum_{t=1} ^{k/2} | \cW'(k,t) | \left( \frac{ C k^{C}  }{N} \right) ^{  k/2-t} \right)^2 . 
\end{eqnarray*}
Using Lemma \ref{le:pathFK} and arguing as above, we find for $C k^{C+2}  / N \leq 1/ 2 $, 
$$
I' \leq 4 \frac{ M^2 }{N} k^{18} \tr ( B B^*)^k. 
$$
Adjusting the value of $C$, we get the first statement of the proposition.

\end{proof}

\begin{lemme}
\label{le:pathFK} For any $1 \leq t \leq k+1$, 
$$
|\cW(k,t)| \leq k^4 ( 3 k )^{10(k-t+1)},
$$
and for any $1 \leq t \leq k/2$, 
$$
|\cW'(k,t)| \leq  k^{5(k-2t) + 9}.
$$
\end{lemme}

\begin{proof}
We follow the strategy developed  by F\"uredi and Koml\'os \cite{MR637828}, especially the exposition of  Vu \cite{MR2384414}. We start with the upper bound on $|\cW (k,t)|$. Let $x = (x_0, \cdots, x_{2k+1}) = (\ibf, \jbf) \in W(k,t)$. We consider the canonical element $ x = (\ibf,\jbf) \in W(k,t)$ defined by $x_0 = 1$ and for all $\ell \geq 0$, $x_{\ell+1} \leq \max_{1 \leq s \leq \ell}(x_{s}) + 1$. In order to upper bound $\cW(k,t)$ we need to find an injective way to encode the canonical vectors $x = (\ibf,\jbf)$.

We mark the oriented edge $(i_{\ell}, i_{\ell+1}) $ (or $(j_{\ell},j_{\ell+1})$) as $+$ if it is the first time that $i_{\ell+1}$ (or $j_{\ell+1}$) is visited. We mark it as $-$ if it is the second time that the oriented edge $(i_{\ell}, i_{\ell+1}) $ (or $(j_{\ell},j_{\ell+1})$) is visited. Otherwise, we say that it is {\em neutral} and mark it as $v = i_{\ell+1}$. The edge $(j_0,j_{1})$ gets also an extra mark in $\{+, j_0\}$ depending on whether $j_0$ has been previously seen: $(j_0,j_{1})$ can be marked as $++$, $+j_1$, $j_0+$, $j_0-$ or $j_0j_1$ (the edge $(j_0,j_1)$ is neutral if its mark is not $++$). We call this encoding the {\em preliminary codeword}.

Now, given the preliminary codeword, it is not always possible to reconstruct the vector $x = (\ibf,\jbf)$. It is due to some edges marked as $-$. Imagine that we have reconstructed $(x_0,\cdots,x_\ell)$. Set $x_\ell = u$, if $(x_{\ell}, x_{\ell +1})$ is marked $-$ and there are more than one index $1 \leq q \leq \ell-1$ such that $x_q = u$ and $(x_{q} ,x_{q +1})$ is marked as $+$ or neutral then there are more than one possibility for the value of $x_{\ell +1}$.

To overcome this issue, we need to add extra information to the preliminary codeword. We now build a {\em redundant codeword} obtained from the preliminary encoding as follows. If $(x_{\ell}, x_{\ell +1})$ is marked $-$ and is as above, we say that it is critical edge and it gets the extra mark $x_{\ell+1}$.  From, the redundant codeword we can now reconstruct the canonical vector $x$ unambiguously. 

As its name suggests, the redundant codeword can be compressed.  The crucial observation is that the configuration $\cdots +- \cdots$ is not possible: if the vertex $i_{\ell}$ is new then the oriented edge $(i_{\ell}, i_{\ell+1})$ cannot be seen for the second time (the orientation of the edges is the fundamental difference with \cite{MR637828,MR2384414}). Hence each $-$ is between two neutral edges in the preliminary codeword and between two neutral edges the preliminary codeword has the form $-\cdots-+\cdots+$ (the sequence of $-$ or $+$ can be empty). Let us call the first $+$ edge between two neutral edges (if it exists), an {\em important} edge. If $N$ is the number of neutral edges and $I$ the number of important edges, we have 
$$
I \leq N-1. 
$$

The {\em final codeword} is the position of the neutral edges, the critical edges and the important edges together with the marks of the neutral and critical edges. From what precedes, the final codeword contains enough information to reconstruct the canonical vector. 

On the other hand, to obtain a critical edge, we need to come back to a vertex which has already been visited. Hence, to any critical edge we can associate in an injective way a previous neutral edge. If $C$ is the number of critical edges, it follows that 
$$
C \leq N.
$$

In $(\ibf, \jbf) \in W(k,t)$, observe that each vertex distinct from $1$ will appear with an oriented edge marked as $+$ and another marked as $-$ (including $j_0$). It follows that there are $t-1$ edges with a mark $+$ and $t-1$ with a mark $-$. In particular, since there are $2k +1$ marks in the preliminary codeword (the edge $(j_0,j_1)$ wears two marks), 
$$N = 2k + 1 - 2 (t-1) = 2 ( k - t +1 ) +1.$$

The final observation is that the number of distinct vertices is at most $t +1$.  We may now prove the first statement of the lemma. The cardinal of $\cW(k,t)$ is upper bounded by the number of ways to place the $N$ neutral edge, $N$ critical edges and $N-1$ important edges over the $(2k+1)$ available slots and put the labels in $\{1, \cdots, t+1\}$ for the neutral and critical edges: 
$$
|\cW (k,t)| \leq (2 k+1)^{3 N -1} (t+1)^{2 N} \leq k^4 ( 3k)^{ 10 ( k -t +1)}.$$

We now turn to the bound on $|\cW'(k,t)|$. The proof is identical. First, a vector $\ibf  \in W'(k,t)$ is canonical if $i_0 =1$ and $i_{\ell+1} \leq \max_{1 \leq s \leq \ell} (i_{s}) + 1$. We consider the above preliminary, redundant and final codewords. Let $N,I,C$ be the number of neutral, critical and important edges in the preliminary codeword.  The bounds $C \leq N$ and $I \leq N-1$ prevail and now $N = k - 2(t-1) = k - 2t +2$. It gives 
$$
|\cW' (k,t)| \leq k^{3 N -1} t^{2 N} \leq  k^{ 5 ( k - 2t) +9} .$$
It concludes the proof.\end{proof}

\subsection{Central limit theorems for bilinear forms of random matrices}

\begin{proposition}\label{prop:TCL} Let $M  >0 $ and $B \in M_N ( \dC)$ be a diagonal  matrix such that $\|B\| \leq M$. We set 
$$
 \frac{1}{N} \tr (B B^*) = \rho_0 \quad \hbox{ and } \quad  \frac{1}{N} \tr (B^2) = \rho_1. 
$$
 Let $u$, $v$ in $\mathbb{C}^N$ be vectors such that  $\|u \|_\infty , \|v\|_{\infty} \leq M / \sqrt N$. 
For $j \geq 1 $, we consider independent  complex Gaussian $G_j \stackrel{d}{\sim} {\cal N}_{\dC}(0, \Sigma_j)$ variables, where, if $\rho^{(j)}_0 = u^* BB^* u v^* v   \rho_0^{j-1}$ and $\rho^{(j)}_1 =  u^* B B^T \bar u  v^T  v   \rho^{j-1}_1 (  \dE X_{11}^ 2) ^ {j}$, 
$$
\Sigma_j =  \begin{pmatrix} \dE \Re (G_j) ^2 & \dE \Re (G_j) \Im ( G_j)   \\
\dE \Re (G_j) \Im ( G_j) & \dE \Im (G_j) ^2    \end{pmatrix} = \frac 1 2 \begin{pmatrix}\rho^{(j)}_0 + \Re ( \rho_1 ^{(j)} ) &  \Im ( \rho_1 ^{(j)} )  \\
\Im ( \rho_1 ^{(j)} ) &\rho^{(j)}_0 - \Re ( \rho_1 ^{(j)} )  \end{pmatrix}.
$$
We set  
$$Z_j=\sqrt{N}  u^*   \left( B \frac{X_N}{\sqrt{N}} \right)^{j}  v.$$
Then for any integer $m \geq 1$, the L\'evy-Prohorov distance between the laws of $(Z_1, \cdots , Z_m)$ and $(G_1, \cdots, G_m)$ is at most $\veps(N)$ where $\lim_{N \to \infty} \veps (N) = 0$ and the function $\veps (\cdot)$ depends only on $M$ and $m$. 
\end{proposition}

\begin{proof}
It is immediate to check that $\Sigma_j$ is non-negative definite. The proof of Propoposition \ref{prop:TCL} follows the lines of the proof of Proposition 4.1 in \cite{tao-outliers}. We use the method of moments: it is sufficient to prove that for any integers $r_i,1 \leq j \leq m$ and any reals $s_j, t_j, 1 \leq i \leq m$,  
\begin{equation}\label{eq:tbpTCL}
\dE \prod_{j=1}^m (s_j Z_j + t_j \bar Z_j )^{r_j}  =   \prod_{j=1}^m \dE (s_j G_j  + t_j \bar G_j )^{r_j}   + o(1),
\end{equation}
where $o(1)$ depends only on $r_j,s_j,t_j, M$. Using Wick's formula, it is not hard to check that if $r$ is even 
$$
\dE (s G_j  + t \bar G_j )^{r} = \frac{ r ! }{ 2 ^{r/2} ( r/ 2) ! } ( 2 s t \rho_0^{(j)} + s^2 \rho_1^{(j)} + t^2\bar  \rho_1^{(j)} )^{r/2}
$$
and it is $0$ if $r$ is odd. 
We set $\alpha = \frac 1 2 \sum_{j=1}^m r_j(j+1)$, $U_i = \sqrt N u_{i}$, $V_i = \sqrt N v_{i}$, $B_{ii} = b_i$. Also, if $\veps   \in \{1 , \bar \cdot \}$, we set $s_j (\veps) =s _j \IND( \veps = 1) +  t_j \IND ( \veps = \bar \cdot )  $.   Expanding the bilinear form, we find
\begin{align}
& \dE \prod_{j=1}^m (s_j Z_j + t_j \bar Z_j )^{r_j} \label{eq:expZj}\\
& = N^{-\alpha} \sum_* \dE \prod_{j=1}^m \prod_{\ell = 1} ^{r_j}  s_{j}(\veps_{j,\ell}) \bar U^{\veps_{j,\ell} }_{k_{j,\ell,0}}  V^{\veps_{j,\ell} }_{k_{j,\ell,j}}   \prod_{i= 0}^{j-1}  b^{\veps_{j,\ell} }_{k_{j,\ell,i}} X^{\veps_{j,\ell} }_{k_{j,\ell,i} k _{j,\ell,i+1}}, \nonumber
\end{align}
where $\veps_{j,\ell}$ and $k_{j,\ell,i}$ range over all tuples of indices with $\veps_{j,\ell} \in \{1 , \bar \cdot \}$ and $\{1,\cdots, N\}$ respectively. A term is non-zero only if each pair $(k_{j,\ell,i} ,  k _{j,\ell,i+1})$ appears at least twice.
% Moreover, observe that the graphs on the indices $k_{j,\ell,i}$ formed by the $\sum_j j r_j$ paths, $(k_{j,\ell,0}, \cdots, k_{j, \ell,j})$ has at most $\sum_j %j r_j$ connected components. Hence, if  each pair $(k_{j,\ell,i} , k _{j,\ell,i+1})$ appears at least twice,  the number of distinct indices is at most $|V| \leq %\sum_j  j r_j / 2 = \alpha$ and the inequality is strict if there is at least one pair $(k_{j,\ell,i} ,  k _{j,\ell,i+1})$ with multiplicity at least $3$ or if the graphs %on the indices $k_{j,\ell,i}$ has less than $\sum_j j r_j$ connected components.  
Recall that by assumption, each term in the summand is $O(1)$ (depending on $r_j,s_j,t_j, M$). It follows that up to $o(1)$ terms, it is sufficient to restrict to indices such that the paths $\pi_{j,\ell} = (k_{j,\ell,0}, \cdots, k_{j, \ell,j})$ are simple paths which occur with multiplicity exactly two and are otherwise disjoint (for a more detailed argument see \cite[Proposition 4.1]{tao-outliers}).

In particular, each such path is matched with another path of the same length. It follows that if  $r_j$ is odd for some $1 \leq j \leq m$ we have $\dE \prod_{j=1}^m (s_j Z_j + t_j \bar Z_j )^{r_j} = o(1)$. Otherwise, a pair of matched paths $(\pi_{j,\ell}, \pi_{j,\ell'})$ can be of three types : 
\begin{enumerate}
\item[(0)]  $\veps_{j,\ell} = 1$, $\veps_{j,\ell'} = \bar \cdot $ or $\veps_{j,\ell} = \bar \cdot $, $\veps_{j,\ell'} = 1$, 
\item[(1)]  $\veps_{j,\ell} = \veps_{j,\ell'} = 1$,
\item[(2)]  $\veps_{j,\ell} = \veps_{j,\ell'} = \bar \cdot$.
\end{enumerate} 
If $\omega = \dE X_{11}^2$, the term $\dE \prod_{i = 0}^{j-1} X_{k_{j,\ell,i} k_{j,\ell,i+1} } \prod_{i = 0}^{j-1} X_{k_{j,\ell',i} k_{j,\ell',i+1}}$ will be equal to $1$, $\omega^j$ or $\bar \omega^j$ if the pair $(\pi_{j,\ell}, \pi_{j,\ell'})$ is of type $0$, $1$ or $2$ respectively.

For each $j$, we order the $r_j/2$ matched pairs $(\pi_{j,\ell},\pi_{j,\ell'})$ in lexicographic order. For $1 \leq a \leq r_j /2$, we set $\omega_{j,a} = 1, \omega^j, \bar \omega^j $, depending on whether the $a$-th pair is of type $0,1,2$. Similarly, we also set $s_{j,a} = s_j t_j, s_j^2,t_j^2$ and, for $x \in \dC$, $x(j,a) = |x|^2, x ^2 , \bar x ^2 $  depending on the type of the $a$-th pair.  We arrive at
\begin{align*}
& \dE \prod_{j=1}^m (s_j Z_j + t_j \bar Z_j )^{r_j} \\
& = N^{-\alpha} \sum_{**} \prod_{j=1}^m \prod_{a = 1} ^{r_j/2}  s_{j,a} \omega_{j,a} \bar U_{k_{j,a,0}}(j,a)  V_{k_{j,a,j}} (j,a)  \prod_{i= 0}^{j-1}  b_{k_{j,a,i}}(j,a) + o(1),
\end{align*}
where the sum is over all simple paths $(k_{j,a,0}, \ldots, k_{j,a,j})$ pairwise disjoint, all collections of matchings of  $\{ (j , \ell) : 1 \leq \ell \leq r_j  \} $, and  $(\veps_{j,\ell})$ ranges over all tuples of indices with $\veps_{j,\ell} \in \{1 , \bar \cdot \}$. Again, each term of summand is $O(1)$, hence, arguing as above, the difference of the above sum with the sum over all paths $(k_{j,a,0}, \ldots, k_{j,a,j})$ is $o( N ^\alpha)$. Hence 
\begin{align*}
& \dE \prod_{j=1}^m (s_j Z_j + t_j \bar Z_j )^{r_j} \\
& =  \prod_{j=1}^m  \left(  N^{-\frac{(j+1)r_j}{2}}  \sum_{***} \prod_{a = 1} ^{r_j/2}  s_{j,a} \omega_{j,a} \bar U_{k_{j,a,0}}(j,a)  V_{k_{j,a,j}} (j,a)  \prod_{i= 0}^{j-1}  b_{k_{j,a,i}}(j,a) \right) + o(1) \\
& = \prod_{j=1} ^m P_j + o(1),
\end{align*}
where, for each $j$, the sum is over paths $(k_{j,a,0}, \ldots, k_{j,a,j})$, matchings of  $\{ (j , \ell) : 1 \leq \ell \leq r_j  \} $, and  $(\veps_{j,\ell})$ ranges over all tuples of indices with $\veps_{j,\ell} \in \{1 , \bar \cdot \}$. There are $(r_j) ! / ( 2^{r_j/2} ( r_j/2)!)$ matchings of the elements of the set $\{ (j , \ell) : 1 \leq \ell \leq r_j  \} $. Observe that the summand in $P_j$ does not depend on the choice of the matching, only on the number of matched paths of type $0$, $1$ and $2$. It follows that 
\begin{align*}
 P_j = & \;   \frac{r_j ! }{  2^{r_j/2} ( r_j/2)!} \sum_{p_0 + p_1 + p_2 =r_j /2}  { r_j \choose ( p_0 , p_1 , p_2) } 2^{p_0}  (s_j t_j)^{p_0}  (\omega^j s_j ^2 )^{p_1}   (\bar \omega^j  t_j ^2)^{p_2} \\
&
\quad \times 
 \left(\frac 1 {N^2} \sum_{i, j =1}^N |U_i|^2 |b_i|^2  |V_j|^2 \right)^{p_0}  \left(\frac 1 {N^2} \sum_{i, j =1}^N \bar U_i^2 b_i ^2 V_j^2 \right)^{p_1} \left(\frac 1 {N^2} \sum_{i, j =1}^N U_i^2 \bar b_i ^2 \bar V_j^2 \right)^{p_2}  \\
& \quad  \times   \rho_{0} ^{p_0 (j-1)}  \rho_{1} ^{p_1 (j-1)}  \bar \rho_{1} ^{p_2 (j-1)} \\
  =  & \;   \frac{r_j ! }{  2^{r_j/2} ( r_j/2)!} \left( 2 s_j t_j   \rho_{0} ^{(j)}  + s_j ^2 \rho_{1}^{(j)}+ t_j ^2   \bar \rho_{1}^{(j)}  \right)^{r_j/2}.&
\end{align*}
We obtain the claim \eqref{eq:tbpTCL}. 
\end{proof}

There is naturally a functional version of Proposition of \ref{prop:TCL}. 

\begin{proposition}\label{prop:fTCL} Let $M , m  >0 $ be integers and $B_1 \cdots, B_m \in M_N ( \dC)$ be diagonal  matrices such that $\|B_i \| \leq M$. For $1 \leq i , i'\leq m$, we set 
$$
 \frac{1}{N} \tr (B_i B_{i'}^*) = \rho_{0ii'} \quad \hbox{ and } \quad  \frac{1}{N} \tr (B_iB_{i'}) = \rho_{1ii'}. 
$$
 Let $u$, $v$ in $\mathbb{C}^N$ be vectors such that  $\|u \|_\infty , \|v\|_{\infty} \leq M / \sqrt N$.  For $i,j \geq 1 $, we consider complex centered Gaussian variables $G_{ij}$ where, $G_{ij}$ and $G_{i'j'}$ are independent for $ j \ne j'$ and whose covariance for $j = j'$ is given by 
\begin{eqnarray*}
\dE G_{ij} \bar G_{i'j} & = & u^* B_iB_{i'}^* u   v^* B^*_{i'} B_{i} v  \rho_{0ii'}^{j-1} ,  \\
\dE G_{ij} G_{i'j} & = & u^* B_iB^T_{i'} \bar u  v^T B_{i'}^T B_i v   \rho^{j-1}_{1ii'} (  \dE X_{11}^ 2) ^ {j}.
\end{eqnarray*}
We set  
$$Z_{ij}=\sqrt{N}  u^*   \left( B_i \frac{X_N}{\sqrt{N}} \right)^{j} B_i v.$$
Then for any integer $m \geq 1$, the L\'evy-Prohorov distance between the laws of $(Z_{ij})_{1 \leq i, j \leq m}$ and $(G_{ij})_{1 \leq i , j \leq m}$ is at most $\veps(N)$ where $\lim_{N \to \infty} \veps (N) = 0$ and the function $\veps (\cdot)$ depends only on $M$ and $m$. 
\end{proposition}

\begin{proof}
We set with $\rho^{(j)}_{0ii'} =  u^* B_iB_{i'}^* u   \rho_{0ii'}^{j-1} $ and $\rho^{(j)}_{1ii'} =  u^* B_iB_{i'} \bar u  v^T v   \rho^{j-1}_{1ii'} (  \dE X_{11}^ 2) ^ {j}$.  We need to extend the proof of Proposition \ref{prop:TCL}. We use again the method of moments. From Cram\'er-Wold Theorem,  it is sufficient to prove that for any $s_{ij}, t_{ij}$ real numbers and $r_j$ integers,
$$
\dE \prod_{j=1}^m \left( \sum_{i=1}^m s_{ij} Z_{ij} + t_{ij} \bar Z_{ij} \right)^{r_j} = \prod_{j=1}^m \dE  \left( \sum_{i=1}^m s_{ij} G_{ij} + t_{ij} \bar G_{ij} \right)^{r_j} + o(1). 
$$
Using Wick's formula, we get if $r$ is even,
\begin{align*}
 \dE  \left( \sum_{i=1}^m s_{ij} G_{ij} + t_{ij} \bar G_{ij} \right)^{r}  &=  \frac{r!}{2^{r/2} (r/2)! } \left( \dE  \left( \sum_{i=1}^m s_{ij} G_{ij} + t_{ij} \bar G_{ij} \right)^{2}  \right)^ {r/2} \\
& =  \frac{r!}{2^{r/2} (r/2)! } \left(   \sum_{1 \leq i, i' \leq m}     s_{ij} t_{i'j} \rho_{0ii'}^{(j)}    + s_{i'j}t_{ij}  \rho_{0i'i}^{(j)} +  s_{ij} s_{i'j} \rho_{1ii'}^{(j)}  + t_{ij} t_{i'j}\bar  \rho_{1ii'}^{(j)}  \right)^ {r/2},
\end{align*}
and the above expression is $0$ if $r$ is odd. We may now repeat the proof of Proposition \ref{prop:TCL}: we set $\alpha = \frac 1 2 \sum_{j=1}^m r_j (j+1)$, $U_i = \sqrt N u_{i}$, $V_i = \sqrt N v_{i}$, $B_{ii} = b_i$. Also, if $\veps   \in \{1 , \bar \cdot \}\times \{1,\cdots, m\}$, we set $s_j (\veps) = \sum_{i=1}^m s _{ij} \IND( \veps = (1,i)) +  t_{ij} \IND ( \veps = (\bar \cdot, i) )  $.   Expanding the bilinear form, we find that \eqref{eq:expZj} holds still true where $\veps_{j,\ell}$ and $k_{j,\ell,i}$ range over all tuples of indices with $\veps_{j,\ell} \in \{1 , \bar \cdot \} \times \{1, \cdots , m\}$ and $\{1,\cdots, N\}$ respectively. 

As in  the proof of Proposition \ref{prop:TCL}, the expression \eqref{eq:expZj} reduces up to $o(1)$ as a sum over matched pairs of paths $(\pi_{j,\ell}, \pi_{j,\ell'})$. The difference is that now a matched pair can be of type $(0ii')$, $(1ii')$ or $(2ii')$ for any $1 \leq i, i' \leq m$, depending on the respective values of $(\veps_{j,\ell},\veps_{j,\ell'})$. A pair will be of type $(0ii')$ if $( \veps_{j,\ell},\veps_{j,\ell'}) = ((1,i),(\bar \cdot, i'))$ or $ = ((\bar \cdot ,i),(1, i))$ when $i = i'$, of type $(1ii')$ if $( \veps_{j,\ell},\veps_{j,\ell'}) = ((1,i),(1, i'))$, or of type $(2ii')$ if $( \veps_{j,\ell},\veps_{j,\ell'}) = ((\bar \cdot,i),(\bar \cdot,  i'))$.

The rest of the proof is identical up to obvious changes due to the modification of the types. \end{proof}

\subsection{Proof of Theorem \ref{th:nilpotent}}

We will assume without loss of generality that 
$$
\sigma = 1. 
$$
From the key identity \eqref{eq:idcle}, the eigenvalues of $M_N$ in $\Gamma$ are given by the zeros of  
\begin{equation}\label{eq:identitecle0}
\det (\lambda I_N- Y_N-A_N)=\det \left(\lambda I_N-(A_N'+ Y_N )\right)(1-u_N^* R_N(\lambda )v_N).
\end{equation}
From Theorem \ref{inclusion}, we find that a.s. for all $N$ large enough, the eigenvalues of $M_N$ in $\Gamma$  are the zeros in $\Gamma$ of 
$$
h_N(z) = \sqrt N ( 1-u_N^* R_N(z )v_N). 
$$
From \eqref{eq:devtaylor} and Lemma \ref{serie}, a.s. for all $N$ large enough, $h_N$ is analytic in $\Gamma$ and can be rewritten as 
\begin{equation}\label{eq:identitecle000}
\frac{ h_N(z) }{ \sqrt N}  =  1 - u_N^* R'_{N}(z) v_N -\sum_{k\geq 1} u^*_N  \left(R'_{N}(z) Y_N \right)^k R'_{N}(z) v_N  .
\end{equation}
We apply \eqref{eq:identitecle0} to $Y_N= 0$, we find
\begin{eqnarray*}
h_N(z)  & = & \sqrt N \frac{ \det ( A_N - z I_N) }{ \det (A'_N - z I_N)} -\sum_{k\geq 1} \sqrt N u^*_N  \left(R'_{N}(z) Y_N \right)^k R'_{N}(z) v_N  \\
& = & \sum_{k\geq 1} a_k (z) + \veps_N(z). 
\end{eqnarray*}
where, from assumption \eqref{eq:ratiodets}, $\veps_N(z)$ is a vanishing bounded analytic function, and for ease of notation, we have set $$a_k (z) =  -  \sqrt N u^*_N  \left(R'_{N}(z)  Y_N  \right)^k R'_{N}(z) v_N$$ (it depends implicitly on $N$). At this stage, it is clear that we need to study the asymptotic normality of the above series.

First, recall that assumption (A4) and \eqref{eq:CFresunif} imply that there exists $C >0$ such that for all $N$ large enough, and all $z \in \Gamma$,  
\begin{equation}\label{eq:boundR'}
\| R_N'(z) \|  \leq C.
\end{equation}
In the sequel, we will always assume that $N$ is large enough, so that the above inequality holds. 

We start with a tightness criterion essentially due to Shirai \cite{MR3014854} of sequence of random analytic functions. 
\begin{lemme}\label{criterion}
Let $D \subset \dC$ be an open connected set in the complex plane.
Let $K \subset  D$ be a compact set  and $U\subset K$ be an open set. Let $(f_N)$ be a sequence of random analytic functions  on $D$. If there exists $p>0$ and $C>0$ such that for all large $N$, $\sup_{z\in K} \dE \vert f_N(z)\vert^p < C $ then $(f_N)$ is a tight sequence of random analytic functions  on $U$.
\end{lemme}
\begin{proof}
Let $K' \subset U$ be a compact set. There exists $\delta>0$ such that the closure of the $\delta$-neighborhood of $K'$ is included in $U$. 
According to Lemma 2.6 \cite{MR3014854}, denoting by $m$ the Lebesgue measure on $\dC$, we have
$$\dE \| f_N \|_{K'} \leq C (\pi \delta^2)^{-1} m(K).$$ Then, the result follows by  Markov inequality and Proposition 2.5 \cite{MR3014854}.

\end{proof}
\begin{lemme}\label{le:hnanalytic}
The sequence  $(h_N)$ defined on $ \mathring{\Gamma}$ is a tight sequence of random analytic functions.   
\end{lemme}

\begin{proof}
From Lemma \ref{criterion}, it is sufficient to check that for some event $\Omega_N$ of probability tending to $1$ and $C >0$, 
\begin{equation}\label{eq:tightcriterion}
\sup_{z \in \Gamma}\dE [ | h_N (z) | \IND_{\Omega_N}] \leq C.
\end{equation}
We define $\Omega_N$ as being the event that for all $z \in \Gamma$ and $k \geq 1$, $ \| \left(R'_{N}(z)  Y_N  \right)^k \| \leq C_0 ( 1- \delta)^k$. From Proposition \ref{borne}, for some $C_0 > 0$ and $0 < \delta <1$, this event has probability tending to $1$.   

 %Recall that $\| A^k - B^k \| \leq k \|A - B \| ( \|A\| \vee \|B \| )^{k-1}$ and $\|R'_N (z) - R'_N(w) \| \leq |z - w| \| R'_N (z) \|\| R'_N (w) \|$ (from the %resolvent identity). 
We get on $\Omega_N$ that, for some $C >0$,
\begin{equation}\label{maja}
| a_k (z) | \leq \sqrt N C ( 1- \delta)^k. 
\end{equation}
For some $\veps >0$ to be chosen later on, we set 
$$
k_N = \lfloor N^{\veps} \rfloor . 
$$
Summing over all $k \geq k_N$, we find that 
\begin{align*}
& \sum_{k > k_N} | a_k (z)| \leq  C' \sqrt N ( 1 - \delta)^{k_N}  = o(1) 
\end{align*}
Hence, in order to prove \eqref{eq:tightcriterion}, we may restrict our attention to 
$$
\tilde h_N (z) = \sum_{k=1}^{k_N} a_k (z). 
$$
By Proposition \ref{prop:supportbeta} and Lemma \ref{continuitephi}, there exists $\delta >0$ such that for all for $z \in \Gamma$, 
$$
\varphi ( z) = \int \lambda^{-1} d \nu_z (\lambda) < 1 - 2\delta. 
$$
Using Assumption (A2) and \eqref{eq:boundR'}, the proof of Lemma \ref{continuitephi} proves that $ \varphi_N(z,z) = \int \lambda^{-1} d \nu_{N,z}$, where $\nu_{N,z}$ is the empirical distribution of the eigenvalues of $(A'_N - z ) (A'_N - z)^*$ converges  to $\varphi(z)$ uniformly on $\Gamma$. It follows that for all $N \geq N_0$ large enough, for all $z \in \Gamma$,
\begin{equation}\label{eq:boundphiN}
\varphi_N (z) < 1 - \delta.
\end{equation}
Then, the claims $
\dE [ | \tilde h_N (z) |] \leq C $   follows directly from Proposition \ref{prop:momentunif} provided that $\veps$ is chosen equal to the constant $1/C$ in this last proposition and (\ref{eq:boundphiN}). 
\end{proof}

We now introduce the random  $\Gamma \to \dC$ function 
$$
g_N (z) =  \sum_{k\geq 1}  \gamma_k (z),
$$
where for each $k \ne \ell$, $z,w \in \Gamma$, $\gamma_k(z)$ and $\gamma_{\ell}(w)$ are independent complex Gaussian variables and for $k = \ell$, 
\begin{eqnarray*}
\dE  \gamma_k (z) \bar \gamma_k (w) & = &u_N^* R'_N(z) (R'_N(w) )^* u_N v_N^* (R'_N(w) )^*  R'_N(z)  v_N \varphi_N(z,w)^{k-1} \\
\dE  \gamma_k (z)  \gamma_k (w) & = & u_N^* R'_N(z) R'_N(w)  \bar u_N v_N^T (R'_N(w)) ^T  R'_N(z)  v_N \psi_N(z,w)^{k-1} (\dE X_{11} ^2 )^k .
\end{eqnarray*}
Observe that $g_N$ is the Gaussian function defined in Theorem \ref{th:nilpotent}. The next lemma implies in particular that $g_N$ is properly defined (for $N$ large enough). 

\begin{lemme}\label{le:gnanalytic}
There exists $N_0\geq 1$ such that the random functions $(g_N)_{N \geq N_0}$ are a tight sequence of random analytic functions. Moreover any weak accumulation point of $(g_N)$ is a.s.  non-zero.  
\end{lemme}

\begin{proof}

We first prove that $\gamma_k(z)$ is indeed a random analytic function. We use an idea borrowed from Najim and Yao \cite{najimyao}. For each $n \geq 1$, let us consider the diagonal matrix $A'_{N,n}$ of size $Nn$ whose diagonal entry $( k N + i )$, $0 \leq k \leq n-1$, $1 \leq i \leq N$, is equal to the $i$-th diagonal entry of $A'_N$. The vectors $u_{N,n}$ and $v_{N,n}$ in $\dC^{nN}$ are built similarly  from $u_N$ and $v_N$: their $( k N + i )$-th entry   is equal to $1  / {\sqrt n}$ times the $i$-th entry of $u_N$ or $v_N$.  We set $R'_{N,n} ( z) = ( A'_{N,n} - z ) ^{-1}$ and consider the random analytic function 
$$
X_{N,n} (z) = \sqrt{ N n} u^*_{N,n}  ( R'_{N,n} (z) Y_{Nn} )^k R'_{N,n}(z) v_{N,n}.
$$
By Proposition \ref{prop:fTCL}, for $N$ fixed, as $n\to \infty$, $X_{N,n}$ converges weakly to $\gamma_k$ for the finite dimensional convergence. Also, from Proposition \ref{prop:momentunif} and \eqref{eq:boundR'}-\eqref{eq:boundphiN}, for some new constant $C$, if $k \leq ( Nn )^{1/C}$,   
\begin{equation}\label{eq:XNntight}
\dE | X_{N,n} (z) | \leq C' k^2  ( 1- \delta)^{k/2} \leq \sum_{l \geq 0} C' l^2  ( 1- \delta)^{l/2} =C". 
\end{equation}
We define $\Omega_{N,n}$ as being the event that for all $z \in \Gamma$ and $k \geq 1$, $ \| \left(R'_{N,n}(z)  Y_{Nn}  \right)^k \| \leq C_0 ( 1- \delta)^k$. From Proposition \ref{borne}, for some $C_0 > 0$ and $0 < \delta <1$, this event has probability tending to $1$.   
Moreover from (\ref{maja}), on $\Omega_{N,n}$, for any $k>( Nn )^{1/C}$,
$$ | X_{N,n} (z) |  \leq \sqrt{ Nn} C ( 1- \delta)^{( Nn )^{1/C}} \leq C'''.$$
Therefore, Lemma \ref{criterion} implies that $(X_{N,n} )_{n \geq 1} $ is a tight sequence of random analytic functions.  It follows that $\gamma_k$ is indeed a random analytic function.

We now check $g_N$ is a  random analytic function. It is sufficient to check that $\dE| \gamma_k (z) |^2 \leq C ( 1 - \delta)^{k-1}$ for some $C,\delta >0$ (see e.g. Shirai \cite{MR3014854} Proposition 2.1).  This follows from \eqref{eq:boundR'} and \ref{eq:boundphiN}.
We thus have proved that $g_N$ is a random analytic function for $N \geq N_0$. Moreover, from what precedes  for some constant $C$, for all $N \geq N_0$ and $z  \in \Gamma$, 
$
\dE | g_N (z) |^2 \leq C.
$
Using again Lemma \ref{criterion}, we find that $(g_N )_{N \geq N_0} $ is a tight sequence of random analytic functions. The fact that the accumulation points are non-zeros follows (1) an immediate uniform lower bound on $\dE |\gamma_k (z) |^2$ and (2) any accumulation point will be a random analytic function with Gaussian finite dimensional marginals. \end{proof}

We are now in position to conclude the proof of Theorem \ref{th:nilpotent}. First, for any $k \geq 1$ and $z_1 ,\ldots, z_k \in \Gamma$,  from Proposition \ref{prop:fTCL}, the random vectors $(a_j(z_i))_{ 1 \leq i , j \leq k}$ and $(\gamma_j(z_i))_{ 1 \leq i , j \leq k}$ have a L\'evy-Prohorov distance going to $0$. Secondly, we have seen in the proofs of Lemmas \ref{le:hnanalytic} and \ref{le:gnanalytic} that $\sum_{ \ell \geq k} a_l(z)$ and $\sum_{ \ell \geq k} \gamma_l(z)$ converge uniformly in $N$ and $z$ in probability to $0$ as $k \to \infty$. It implies that $(h_N(z_i) ) _{1 \leq i \leq k}$ and $(g_N(z_i))_{1 \leq i \leq k}$ have a L\'evy-Prohorov distance going to $0$. Using Lemma \ref{le:roucheRAF} along any converging subsequence of $g_N$, we deduce the statement of  Theorem \ref{th:nilpotent}.

\begin{remarque}[Extension to $r$ arbitrary]\label{re:rgeneral} We have assumed for simplicity that $r = 1$. If $r \geq 2$, from \eqref{eq:idcle}, we have to deal with $r \times r$ determinants for the analog of $h_N$ and the assumption \eqref{eq:ratiodets} needs to be adapted accordingly. The argument of tightness would remain unchanged. However, for the asymptotic normality, a multi-variate  generalization of Proposition \ref{prop:fTCL} for a collection of vectors $u^{(i)}_{N}, v^{(i)}_{N}$, $1 \leq i \leq r$, would be necessary.  
\end{remarque}

\subsection{Proof of Corollary \ref{cor:nilpotent}}

Observe that $\nu_z$ is the law of $|L - z |^2$ where $L$ follows the uniform distribution on the unit circle. Hence, by proposition \ref{prop:supportbeta}, we have
$$
\supp ( \beta ) = \left\{ z\in \dC : \frac {1}{2 \pi}  \int_0 ^{2\pi}  \frac{ dx }{ | e^{ix}- z |^2} \geq \sigma^{-2} \right\}. 
$$

Let us evaluate the expression 
$$
\varphi(z,w) =  \frac {1}{2 \pi}  \int_0 ^{2\pi}  \frac{ dx }{ (e^{ix}-z)(e^{-ix} - \bar w)  }
$$
We rewrite it as 
$$
\varphi(z,w) =  \frac {1}{2  i \pi}  \int_0 ^{2\pi}  \frac{  i e^{ix} dx }{ (e^{ix}-z)(1 - \bar w e^{ix} )  } =  \frac {1}{2  i \pi}  \oint   \frac{  d u  }{ (u -z)(1 - \bar w u  )  }, 
$$
where the contour is the unit circle oriented counter-clockwise. The integrand as a pole at $z$ with residue $1/(1 - \bar w z)$ and another at $1/ \bar w$ with residue $1/(\bar w z - 1)$. It is then immediate to estimate the above integral with Cauchy's residue formula. We find if $ z  = w$, 
$$
\varphi(z,z) = \frac{1}{ | |z|^2 -1 | }, 
$$ 
and if $|z|, |w| < 1$, 
$$
\varphi(z,w) = \frac{1}{ 1 - \bar w z}.
$$ 

We now turn to the computation of $\dE g_N(z)  \bar g_N(w)$ and $\dE g_N(z) g_N(w)$. We set $\rho= e^{2 i \pi/N}$. We apply Theorem \ref{th:nilpotent} with for $1 \leq k \leq N$, $(u_N)_k = - 1 / \sqrt N$, $(v_N)_k =\rho^k / \sqrt N$ and $R'_N(z)_{kk} = (z - \rho^k)^{-1}$. In the evaluation of $\dE g_N(z) \bar g_N(w)$ and $\dE g_N(z) g_N(w)$, we recognize Riemann sums:  we have that 
$$
u_N^* R'_N(z) R'_N(w)^* u_N = v_N^* R'_N(w)^* R'_N(z) v_N = \varphi_N(z,w) =  \frac 1 N \sum_{k=1}^N \frac{1}{(\bar  w- \rho^{-k})( z - \rho^k) } \to \varphi(z,w).  
$$
Similarly, 
$$
u_N^* R'_N(z) R'_N(w) \bar u_N = \psi_N(z,w) =  \frac 1 N \sum_{k=1}^N \frac{1}{(   w- \rho^{k})( z - \rho^k) } \to \psi(z,w),  
$$
with,  $$
\psi(z,w) =  \frac {1}{2 \pi}  \int_0 ^{2\pi}  \frac{ dx }{ (e^{ix}-z)(e^{ix} - w)  } = \frac{1}{2 i \pi} \oint \frac{ du }{ u( u -z)(u - w)  },
$$
where the contour is the unit circle oriented counter-clockwise.  Another straightforward residue computation gives that if $|z|, |w| < 1$ then $\psi(z,w) = 0$. It concludes the proof of Corollary \ref{cor:nilpotent}.

\subsection{How to obtain an unstable outlier ?}
\label{subsec:recipe}

Theorems \ref{th:main} and \ref{th:nilpotent} reveal the importance of the ratio 
$$
\veps_N (z) = \frac{\det( A_N - z I_N)}{\det ( A'_N - z I_N)}  
$$
to determine the nature of an outlier: stable or unstable. In the simplest case $r=1$, $A''_N = v_N u_N^*$, $A'_N$ diagonal,  the nature of an outlier can be guessed.  

First, as already noted, from \eqref{eq:identitecle0} applied $Y_N = 0$ and $A'_N$ diagonal gives 
$$
\veps_N(z) = 1 - u_N^* R'_N(z) v_N = 1 - \frac 1 N \sum_{k=1}^N \frac{w_{N,k}}{z - \lambda_{N,k}}, 
$$
where $R'_N(z) = (z I_N - A'_N)^{-1}$, $(\lambda_{N,k})_{1 \leq k \leq N}$ are the eigenvalues of $A'_N$ and $w_{N,k} = N (\bar u_N)_k (v_N)_k$.  Assume further that 
$$
\rho_{N} = \frac 1 N \sum_{k=1}^N \delta_{ (\lambda_{N,k} , w_{N,k} )  } 
$$
converges weakly to a probability measure $d\rho(\lambda,w)$ on $\dC^2$ such that its first marginal is $\alpha$, a  probability measure with compact support. Under assumption (A4') and $\| u_N \|_ {\infty}, \|v_N\|_{\infty}$ of order $O(1 / \sqrt N)$,  we get that for all $z \notin S = \supp( \alpha)$, 
$$
\veps_N (z) \to \veps(z) = 1 - \int \frac{ \omega(\lambda)}{ z - \lambda }d \alpha(\lambda) ,
$$
where $\omega(\lambda)$ is the conditional expectation under $\rho$ of the second variable given the first is equal to $\lambda$. To obtain an unstable outlier, it is necessary that $\veps(z)$ vanishes on a domain $\Gamma$ in $\supp( \beta)^c$.  A more precise convergence estimate on $\rho_N$ is also necessary to guarantee also that $|\veps_N(z) - \veps(z) | = o ( 1/ \sqrt N)$.

The function $\veps(z)$ is analytic outside $S = \supp(\alpha)$ and it can be computed in concrete examples such as the one of Corollary \ref{cor:nilpotent}, where $\alpha$ is the uniform distribution on the unit disc and $\omega(\lambda) = - \lambda$. More generally, assume that $\alpha$ is radial with support $\{ z \in \dC : |z| \in [a, b]\}$ with $a >0$. We write $d \alpha (r e^{i \theta} ) = \frac 1 {2 \pi} r F(dr)$ and consider a measurable function $f$ on $\dR$ such that  $\int f (r ) r F(dr)  = 1$ (e.g.  $f = 1$). Interestingly,  if $\omega(\lambda) = - \lambda f(|\lambda|)$ then $\veps (z)= 0$ for all $z$ with $|z| < a$. Indeed, in this case, if $|z| < a$, we have 
\begin{eqnarray*}
\int \frac{ \omega(\lambda)}{ z - \lambda }d \alpha(\lambda)&  = &  \int_a ^b \left(\frac{1}{2 \pi}  \int_0 ^{2 \pi} \frac{r e^{i \theta} d \theta }{r e^{i \theta} - z }  \right) f(r) r F( dr) \\
& = & \int_a ^b \left(\frac{1}{2 i  \pi}  \oint_{C_r} \frac{du }{u - z }  \right) f(r) r F( dr) \\
& = &  \int_a ^b  f(r) r F( dr),
\end{eqnarray*}
where at the second line, $C_r$ is the disc of radius $r$ oriented counter-clockwise.

\subsection{Proof of Theorem \ref{th:largenorm}}

The proof is a variant of the proof of Theorem \ref{th:nilpotent}. We may assume $\sigma =1$. Arguing as above \eqref{eq:identitecle000}, we need to consider the zeros in $\Gamma$ of the random analytic function 
\begin{eqnarray}
f_N(z)   &  = &  1 - \sqrt N u_N^* R_N (z) v_N \nonumber\\
&= & 1 - \sqrt N u_N^* R'_{N}(z) v_N - \sqrt N \sum_{k\geq 1} u^*_N  \left(R'_{N}(z) Y_N \right)^k R'_{N}(z) v_N  .\label{eq:devTaylor34}
\end{eqnarray}

We have seen in Theorem \ref{th:nilpotent} that $ h_N (z)  =  \sqrt N \sum_{k\geq 1} u^*_N  \left(R'_{N}(z) Y_N \right)^k R'_{N}(z) v_N $ defines a tight sequence of random analytic functions in $\Gamma$ such that for any $z_1, \cdots, z_k$ in $\Gamma$, $(h_N(z_i))_{1 \leq i  \leq k}$ and $(g_N(z_i))_{1 \leq i \leq k}$ have a L\'evy distance going to $0$.  Moreover, by assumption, $1 - \sqrt N u_N^* R'_{N}(z) v_N $ is a sequence of bounded analytic functions on $\Gamma$. It follows that for any $z_1, \cdots, z_k$ in $\Gamma$, $(f_N(z_i))_{1 \leq i  \leq k}$ and $(1 - \sqrt N u_N^* R'_{N}(z_i) v_N  +  g_N(z_i))_{1 \leq i \leq k}$ have a L\'evy distance going to $0$. Using Lemma \ref{le:roucheRAF} along any converging subsequence of $g_N$, we deduce the statement of  Theorem \ref{th:largenorm}.

\subsection{Proof of Corollary \ref{cor:largenorm}}

If $A'_N = 0$, we simply have $R'_N(z) = z^{-1} I_N$. The expression \eqref{eq:devTaylor34} simplifies to 
$$
f_N(z)    =  1 - z^{-1} \theta_N u_N^\top v_N - \frac{ \theta_N}{\sqrt N}  \sum_{k\geq 1} z^{-k - 1} \left( \sqrt N u^\top_N  Y_N ^k v_N  \right).
$$
Since $\theta_N \ne 0$, we are thus interested by the zeros outside $B(0, 1 + \veps)$ of
$$
 z  \frac{ \sqrt N }{\theta_N} \left(z  -  \theta_N u_N^\top v_N \right)- \sum_{k\geq 0} z^{-k} \left( \sqrt N u^T_N  Y_N ^{k+1} v_N  \right).
$$
The corollary follows easily.

\section*{Appendix: central limit theorem for bilinear forms with random vectors}
\label{sec:appendix}

In this appendix, we establish the following central limit theorem.
\begin{proposition}\label{TCLFQ}
Let $x$ and $y$ be independent centered complex random variables such that $\mathbb{E}(\vert x \vert^2)= \mathbb{E}(\vert y \vert^2)=1$ and $\mathbb{E}(\vert x \vert^4)< +\infty$, $ \mathbb{E}(\vert y \vert^4)<+\infty$.
Let $\{x^{(p)}_i;y^{(q)}_j ; 1 \leq i, j \leq N, 1 \leq p, q \leq r \}$ be independent random variables such that the $x^{(p)}_i$'s are  copies of $x$ and the $y^{(q)}_j$'s are  copies of $y$.
Set for any $ 1 \leq  p \leq r$, $x^{(p)}= (x^{(p)}_i)_{1 \leq i \leq r} \in \dC^r$ and  $y^{(p)}= (y^{(p)}_i)_{1 \leq i \leq r} \in \dC^r$
Let $B_N \in M_N(\dC)$ be a sequence of deterministic matrices such that 
\begin{enumerate}[(i)]
\item There exists $C>0$ such that  $\sup_N \Vert B_N \Vert \leq C$.
\item The following limit exists 
$$\tau =\lim_{N \rightarrow +\infty} \frac{1}{N} \Tr B_N B_N^*.$$
\item Either $\mathbb{E}(x^2)=0$ or $ \mathbb{E}(y^2)=0$ or   the following limit exists $$\zeta = \lim_{N \rightarrow +\infty} \frac{1}{N} \Tr B_N B_N^\top.$$
(in the first two cases, we set $\zeta=0$).
\end{enumerate}
Then $\left\{\frac{1}{\sqrt{N}}  {x^{(p)}}^\top B_Ny^{(q)} ; 1\leq p,q \leq r \right\}$ converges weakly  towards $r^2$ independent copies of a centered complex gaussian variable $g_1+ig_2$ such that the covariance matrix of the Gaussian vector 
$(g_1,g_2)$ is $$\frac{1}{2}\begin{pmatrix} \Re \left\{ \mathbb{E}(x^2) \mathbb{E}(y^2)  \zeta \right\}+\tau & \Im \left \{\mathbb{E}(x^2) \mathbb{E}(y^2)  \zeta \right\} \\ \Im \left \{ \mathbb{E}(x^2) \mathbb{E}(y^2)  \zeta \right\} &  \tau-\Re \left\{ \mathbb{E}(x^2) \mathbb{E}(y^2) \zeta \right\} \end{pmatrix}.$$ 
\end{proposition}

The  proof  follows the approach Baik and Silverstein in the Appendix of \cite{MR2489158} and  uses the following CLT.

\begin{theoreme} \label{Theo-Bil}(Theorem 35.12 of \cite{MR1324786}) For each $N$, let $Z_{N 1}, \ldots , Z_{N m_N}$ be a real martingale difference sequence
  with respect to the increasing $\sigma$-field $\{ \mathcal F _{N,j} \}$
  having second moments. If as $N \to \infty$, 
\begin{equation}{\label{Condition1}}
\sum _{j=1}^{m_N} \mathbb E ( Z_{Nj}^2 |   \mathcal F _{N,j-1})
\overset{P}{\longrightarrow} v^2
\end{equation}
where $v^2$ is a positive constant, and for each $\epsilon >0$,
\begin{equation}{\label{Condition2}}
\sum _{j=1}^{m_N} \mathbb E ( Z_{Nj}^2 \, 1 _{| Z_{Nj}| \geq \epsilon} ) \,
{\rightarrow} \, 0
\end{equation}
then  $\sum _{j=1}^{m_N}  Z_{Nj} $ converges in distribution to $\mathcal N (0,v^2)$.
\end{theoreme}

We will use a standard lemma.
\begin{lemme}\label{BaiSilver98}
Let $B \in M_N ( \dC)$ and $X = ( x_1, \cdots , x_N)^{\top}$, $Y = ( y_1, \cdots , y_N)^{\top}$ be independent random complex vectors with independent
 entries such that   $\dE (x_i )  = \mathbb{E}(y_i)=0$ and $\max_i \dE ( |x_i|^4 ) \vee \dE ( |y_i|^4 ) \leq \kappa $. Then, there is a universal constant $c>0$ such that 
\begin{eqnarray*}
\dE \vert Y^\top B Y  -  \dE (y_1^2) \Tr (B) \vert^2 & \leq & c \kappa \Tr (B B ^*) \\
\mathbb E\vert  Y^* B Y -  \dE (|y_1|^2) \Tr(B) \vert^2 &\leq& c \kappa  \Tr (BB^*)\\
\mathbb E\vert  X^* B Y  \vert^2 &\leq& c \kappa  \Tr (BB^*).
\end{eqnarray*}
\end{lemme}
\begin{proof}
Let us start with the first statement. Note that $$Y^\top B Y - \dE (y_1^2) \Tr B = \sum_i \sum_{j<i} y_i y_j  B_{ij} + \sum_i \sum_{j<i} y_i y_j  B_{ji}   + \sum_i  ( y^2_i  - \dE (y ^2 _i) ) B_{ii}.$$
Now, $$ \mathbb{E}\left( \left| \sum_i \sum_{j<i} y_i y_j B_{ij} \right|^2\right) = (\mathbb{E}\vert y_1\vert^2)^2 \sum_i \sum_{j<i} \vert B_{ij}\vert^2$$
and 
$$ \mathbb{E}\left( \left| \sum_i ( y^2_i  - \dE (y ^2 _i) )  B_{ii} \right|^2\right) = \mathbb{E}( \vert  y^2_i  - \dE (y ^2 _i)  \vert^2)  \sum_i  \vert B_{ii}\vert^2$$
We conclude the proof of the first statement by using $|a + b + c|^2 \leq 3 (|a|^2 + |b|^2+ |c|^2)$ and $\sum_{i} \sum_{j} |B_{ij}|^2 = \Tr (B B ^*)$. The two last statements are proved similarly, (see e.g. \cite[Lemma B.26]{MR2567175}). \end{proof}

\begin{proof}[Proof of Proposition \ref{TCLFQ}]
For any vector $u= (u_1,\ldots , u_N)^\top$ in $\mathbb{C}^N$, we will denote by $\overline u$ the vector in $\mathbb{C}^N$ defined by  
 $\bar u = (\bar{u}_1,\ldots , \bar{u}_N)^\top$. 
For any $\alpha=\{\alpha_{pq} \in \mathbb{C}, (p,q)\in \{1,\ldots,r\}^2\}$, define
$$\xi_\alpha=\frac{1}{\sqrt{N}}  \sum_{p,q} \{ \alpha_{pq} { x^{(p)}}^\top B_N y^{(q)}  + \bar \alpha_{pq} {y^{(q)}}^* B_N^* \bar x^{(p)} \}.$$
The proof of the  proposition is based on the writing of $\xi_\alpha$
as a sum of martingale differences in order to apply Theorem \ref{Theo-Bil}. We define $L$  as the lower triangular part of $B_N$ (including the diagonal) and $U$
the strictly upper triangular part of $B_N$.  We write
\begin{eqnarray*}
{ x^{(p)}}^\top B_N y^{(q)}  =   { x^{(p)}}^\top L y^{(q)}  +  { x^{(p)}}^\top U y^{(q)}  =  { x^{(p)}}^\top L y^{(q)}  +    {y^{(q)}}^\top U^\top  { x^{(p)}}.
\end{eqnarray*}
Thus $$\xi_\alpha= \sum_{i=1}^N Z_i,$$
where \begin{eqnarray*}Z_i&=& \frac{1}{\sqrt{N}}  \sum_{p,q} \left(  \alpha_{pq} \left\{  x_i^{(p)} (L y^{(q)})_i +  y_i^{(q)} ( U^\top  { x^{(p)}})_i \right\} + \bar \alpha_{pq}  \left\{ \bar  x_i^{(p)} ( \bar L \bar y^{(q)} )_i  +  \bar y_i^{(q)} ( U^*  { \bar x^{(p)}})_i \right\} \right).
\end{eqnarray*}

Let ${\cal F}_i$ be the $\sigma$-field generated by $\{ x_j^{(p)},  \bar x_j^{(p)},  y_j^{(q)},  \bar y_j^{(q)}, 1\leq p,q \leq r, j \leq i \}$. Since $L$ and $U ^\top$ are lower triangular, $Z_i$ is measurable with respect to ${\cal F}_i$ and satisfies $\dE ( Z_i |   \mathcal F _{i-1})=0$.

We start by verifying the Lindeberg's condition (\ref{Condition2}). We will show that it is true for each element in the finite sum on $(p,q)$ defining $Z_i$ since  this property
is closed under addition as explained in (A4) of 
 the Appendix of \cite{MR2489158}. Setting $B_N=(b_{ij})_{1\leq i , j \leq N}$, we find
\begin{eqnarray*}
\dE \left(\bigm| (L y^{(q)})_i  \bigm|^4\right)  & = & \mathbb{E}\left(\bigm|\sum_{j\leq i}y_j^{(q)}b_{ij}\bigm|^4\right) \\
& = &
 \sum_{j_1\leq i,j_2\leq i, j_3\leq i,j_4\leq i} \mathbb{E}\left(y_{j_1}^{(q)}b_{ij_{1}}y_{j_2}^{(q)}b_{ij_2}  y_{j_3}^{(q)} b_{ij_3}  y_{j_4}^{(q)}  b_{ij_4}\right)\\&=&
 \mathbb{E}|y|^4 \sum_{j\leq i} | b_{ij}|^4 +
2 \sum_{*} 
\vert b_{ij_{1}} b_{ij_2}\vert^2 + \vert \mathbb{E}(y^2)\vert^2  \sum_{*} 
( b_{ij_{1}})^2 (b_{ij_2})^2,
  \end{eqnarray*}
where the sum over $*$ is over $\{ j_1 \leq  i , j_2 \leq  i , j_1 \ne j_2\}$. Note that $\sum_{j_1,j_2} \vert b_{ij_{1}} b_{ij_2}\vert^2= \left[ (B_NB_N^*)_{ii}\right]^2 \leq C^4$, by Assumption (i). It readily follows that  
$$\frac{1}{N^2} \dE \left(\bigm | x_i ^{(p)} (L y^{(q)})_i  \bigm|^4\right)  =o\left(\frac 1 N \right),$$ and therefore, from Markov inequality, for any $\epsilon>0$,   as $N\rightarrow\infty$
\begin{eqnarray*}
\sum_{i=1}^N\mathbb{E}\left(\frac{1}{N} |x_i^{(p)}  (L y^{(q)})_i  |^2 \, \1_{\left\{\frac{1}{\sqrt{N}} |x_i^{(p)}  (L y^{(q)})_i  |\geq\epsilon\right\}}\right)\leq(1/\epsilon^2)
\sum_{i=1}^N\frac{1}{N^2}\mathbb{E}|x_i^{(p)} (L y^{(q)})_i|^4  = o(1).
\end{eqnarray*}

Similarly, $$ \sum_{i=1}^N\mathbb{E}\left(\frac{1}{N} |y_i^{(q)} ( U^\top  { x^{(p)}})_i |^2 \, \1_{\left\{\frac{1}{\sqrt{N}} |y_i^{(q)} ( U^\top  { x^{(p)}})_i|\geq\epsilon\right\}}\right) = o(1).$$
Thus $\{Z_i\}$ satisfies  (\ref{Condition2}).

It remains to check that $\{Z_i\}$ satisfies  (\ref{Condition1}) on conditional variances. We further decompose $L = K + D$ where $K$ is strictly lower triangular and $D$ is the diagonal part of $B_N$. The explicit development of $Z_i ^2$ gives
\begin{align}
&  \sum_{i=1}^N \mathbb{E} \left( Z_i^2|   \mathcal F_{i-1}\right)  \, =  \,  2 \Re \left\{ \mathbb{E} \left( x^2 \right) \mathbb{E} \left( y^2 \right) \sum_{p,q} \alpha_{pq}^2
\left( \frac{1}{N} \Tr D^2  \right) \right\}  + 2 \sum_{p,q}| \alpha_{pq}|^2
\left( \frac{1}{N} \Tr D D^* \right) \nonumber \\
& \quad + 2 \Re \left\{  \mathbb{E} \left( x^2 \right) \sum_{p,q_1,q_2}\alpha_{pq_1} \alpha_{pq_2} \frac{1}{N} \sum_{i=1}^N (K y^{(q_1)})_i (K y^{(q_2)})_i \right\}  \nonumber
\\
&\quad + 2 \Re \left\{  \mathbb{E} \left( y^2 \right) \sum_{p_1,p_2,q}\alpha_{p_1q} \alpha_{p_2 q} \frac{1}{N}  \sum_{i=1}^N (U^\top x^{(p_1)})_i (U^\top  x^{(p_2)})_i\right\} \nonumber\\
&\quad +2 \sum_{p,q_1,q_2}\alpha_{pq_1} \bar  \alpha_{p q_2} \frac{1}{N}(K y^{(q_1)})_i (\bar K \bar y^{(q_2)})_i  +2 \sum_{p_1,p_2,q}\alpha_{p_1q} \bar \alpha_{p_2 q}  \frac{1}{N}  \sum_{i=1}^N (U^\top x^{(p_1)})_i (\bar U^\top  \bar x^{(p_2)})_i. \label{suite2}
\end{align}
We note that 
\begin{equation}\label{mat1}\sum_{i=1}^N (K y^{(q_1)})_i (K y^{(q_2)})_i  = {y^{(q_1)}}^\top K^{\top} K y^{(q_2)},\end{equation}
\begin{equation}\label{mat3}  \sum_{i=1}^N (K y^{(q_1)})_i (\bar K \bar y^{(q_2)})_i = {y^{(q_2)}}^* K^*  K  y^{(q_1)},\end{equation}
and similarly for the expressions with $x^{(p)}$ and $U$.  Also, according to Mathias \cite{MR1238930}, we have the following inequalities, where $\gamma_N =O( \log  N)$,
\begin{equation}\label{mathias}\Vert K \Vert \leq \gamma_N \Vert B_N\Vert ,~\hbox{ and }~\Vert U \Vert \leq \gamma_N \Vert B_N\Vert.\end{equation}
Therefore Assumption (i), Lemma \ref{BaiSilver98} and identities (\ref{mat1}) readily imply that for any $p,q$,
\begin{eqnarray*} \sum_{p,q_1,q_2}\alpha_{pq_1} \alpha_{pq_2} \frac{1}{N} \sum_{i=1}^N (K y^{(q_1)})_i (K y^{(q_2)})_i&=&  \mathbb{E}(y^2) \left(  \sum_{q} \alpha^2 _{pq} \right)   \frac 1 N \Tr K K^\top
 +o_{\mathbb{P}}(1),\\
 \sum_{p_1,p_2,q}\alpha_{p_1q} \alpha_{p_2 q} \sum_{i=1}^N (U^\top x^{(p_1)})_i (U^\top  x^{(p_2)})_i &=&  \mathbb{E}(x^2)\left(   \sum_{p} \alpha^2 _{pq} \right) \frac{1}{N}   \Tr U U^\top +o_{\mathbb{P}}(1).
\end{eqnarray*}

Similarly, for the two terms in (\ref{suite2}), we use  Lemma \ref{BaiSilver98}, (\ref{mat3}) and (\ref{mathias}). We find that for any $p,q$
\begin{eqnarray*}
 \sum_{p,q_1,q_2}\alpha_{pq_1} \bar  \alpha_{p q_2} \frac{1}{N}(K y^{(q_1)})_i (\bar K \bar y^{(q_2)})_i  &=&   \left( \sum_q  | \alpha_{pq} |^2 \right) \frac{1}{N}    \Tr K K^*+o_{\mathbb{P}}(1)\\
 \sum_{p_1,p_2,q}\alpha_{p_1q} \bar \alpha_{p_2 q}  \frac{1}{N}  \sum_{i=1}^N (U^\top x^{(p_1)})_i (\bar U^\top  \bar x^{(p_2)})_i &= & \left( \sum_p  | \alpha_{pq} |^2 \right)\frac{1}{N}     \Tr U U^* +     o_{\mathbb{P}}(1).
\end{eqnarray*}

We also observe that 
$$
\Tr B_N B_N ^\top = \Tr D^2 + \Tr K K^\top + \Tr UU^\top.   
$$
and similarly for $\Tr B_N B_N ^* $. 
We thus have proved that 
\begin{eqnarray*}
\sum_{i=1}^N \mathbb{E} \left( Z_i^2|   \mathcal F_{i-1}\right)  =  2 \Re \left\{\mathbb{E} \left( x^2 \right) \mathbb{E} \left( y^2 \right)\left(  \sum_{p,q} \alpha_{pq}^2\right) \frac{1}{N}\Tr B_N B_N^\top \right\}  + 2 \left( \sum_{p,q} \vert \alpha_{pq} \vert^2 \right) \frac{1}{N}\Tr B_N B_N^*   +o_{\mathbb{P}}(1).
\end{eqnarray*}
Using Assumption (ii)-(iii), we may thus apply Theorem \ref{Theo-Bil}, we get that for any $\alpha=\{\alpha_{pq} \in \mathbb{C}, (p,q)\in \{1,\ldots,r\}^2\}$, 
$$\frac{1}{\sqrt{N}}  \sum_{p,q} \{ \alpha_{pq} { x^{(p)}}^\top B_N y^{(q)}  + \bar \alpha_{pq} {y^{(q)}}^* B_N^* \bar x^{(p)} \}$$
weakly converges towards a centered gaussian variable with variance $ 2\Re \left\{\mathbb{E} \left( x^2 \right) \mathbb{E} \left( y^2 \right) \left(\sum_{p,q} \alpha_{pq}^2 \right)\zeta\right\}  + 2 \left( \sum_{p,q} \vert \alpha_{pq} \vert^2 \right) \tau.$ It concludes the proof of Proposition \ref{TCLFQ}. \end{proof}

\bibliographystyle{abbrv}

\bibliography{mat}

\bigskip
\noindent
 Charles Bordenave \\
 Institut de Math\'ematiques de Toulouse. CNRS and University of Toulouse III. \\
118 route de Narbonne. 31062 Toulouse cedex 09.  France. \\
\noindent
{E-mail:} {\tt bordenave@math.univ-toulouse.fr} \\
\noindent
\url{http://www.math.univ-toulouse.fr/~bordenave}

\bigskip
\noindent
Mireille Capitaine \\
 Institut de Math\'ematiques de Toulouse. CNRS and University of Toulouse III. \\
118 route de Narbonne. 31062 Toulouse cedex 09.  France. \\
\noindent
{E-mail:} {\tt mireille.capitaine@math.univ-toulouse.fr} \\
\noindent
\url{http://www.math.univ-toulouse.fr/~capitain}

%\author{Charles Bordenave\thanks{CNRS, Institut de Math\'ematiques de Toulouse,  F-31062 Toulouse Cedex 09.  E-mail: charles.bordenave@math.univ-toulouse.fr } \and Mireille Capitaine\thanks{CNRS, Institut de Math\'ematiques de Toulouse,  F-31062 Toulouse Cedex 09.  E-mail: mireille.capitaine@math.univ-toulouse.fr }}

\end{document}